\documentclass[english]{article}

\usepackage{preamble}

\begin{document}

\title{Relative de Rham Theory on Nash Manifolds}
\author{Avraham Aizenbud and Shachar Carmeli}
\date{November 2018}

\maketitle

\begin{abstract}
For a Nash submersion $\phi\colon X\to Y$, we study the complex $\Sdr(\phi)$ of Schwartz sections of the relative de Rham complex of $\phi$. We define the notion of Schwartz sections of constructible sheaves on Nash manifolds and prove that $\Sdr(\phi)$ is homotopy equivalent to the Schwartz sections of the proper push-forward $\phi_!\RR_X$ of the constant sheaf $\RR_X$. Using this equivalence, we show that $\Sdr(\phi)$ depends (up to homotopy equivalence) only on the homology type of the map $\phi$. We also deduce that $\Sdr(\phi)$ has Hausdorff homology spaces.
\end{abstract}

\tableofcontents
\section{Introduction}
\label{sec:intro}

The theory of Nash manifolds provides a convenient interplay between differential geometry and algebraic geometry. Nash manifolds are rigid enough to define classes of functions according to their growth rate, and on the other hand, flexible enough to share properties of differentiable manifolds such as being locally isomorphic to open subsets of Euclidean spaces. 

One fundamental property of differentiable manifolds is that their cohomologies can be computed using the de Rham complex. This feature was adapted to the Nash settings in \cite{AizDer,DerShapPrell}, where the Schwartz sections replace the compactly supported sections of the de Rham complex, and this does not change the cohomologies of the complex.

In this paper, we extend this result to the relative context of a Nash submersion. For a submersion of Nash manifolds $\phi:X\to Y$, we study the relative de Rahm complex of $\phi$ and the associated complex of Schwartz sections. Unlike the absolute case, this complex can not be identified with its compactly supported version. Instead, it turns out that the Schwartz sections of the relative de Rham complex identify with the Schwartz sections of the constructible sheaf $\phi_!\RR$ on $Y$. A prime goal of this paper is to make this statement precise and prove it.

\subsection{Main Results}

For a Nash manifold $X$, let $\mathrm{D}(X,\RR)$ denote the derived category of sheaves of real vector spaces on $X$, and by $\mathrm{D}_{c}(X,\RR)\subseteq \mathrm{D}(X,\RR)$ the full subcategory consisting of bounded complexes with constructible cohomologies  (with respect to a semi-algebraic stratification). Let $\Ch_b(\Fre)$ denote the homotopy category of bounded complexes of Fr\'{e}chet spaces. For a Nash submersion $\phi\colon X\to Y$, let $\Sdr(\phi)$ denote the complex of Schwartz sections of the relative de Rham complex of $\phi$. 
The main results of this paper imply the following
\begin{introthm}(Main Theorem)
\label{intro-rel_der}
    There is a triangulated functor $\Sc \colon \mathrm{D}_{c}(X,\RR) \to \Ch_b(\Fre)$ satisfying the following properties: 
    \begin{enumerate}
    \item For a Nash submersion $\phi\colon X\to Y$, we have 
    \[
        \Sc(\phi_!\RR_X) \simeq \Sdr(\phi) \qin \Ch_b(\Fre). 
    \]
    Here, $\RR_X$ stands for the constant sheaf on $X$ with value $\RR$. 
    \item For every $\sF \in \mathrm{D}_c(X,\RR)$, the homologies  $H_i(\Sc(\sF))$ are Hausdorff topological vector spaces.   
    \end{enumerate}
\end{introthm}

In fact, we prove a stronger result. Namely, we construct $\infty$-categories of constructible sheaves, $\Shv_c(X,\Der(\RR))$, and of complexes of Fr\'{e}chet spaces, $\Fre_\infty$, and we then build a functor 
\[
\Shv_c(X,\Der(\RR))\oto{\Sc} \Fre_\infty
\]
of $\infty$-categories satisfying these properties. 
Also, we allow replacing the Schwartz sections functor $\Sc$ with the functor of Schwartz sections of a Nash vector bundle.

The main theorem has several consequences for the relative de Rham complexes. 
First, we have 
\begin{introcor}(\Cref{rel_der_haus}) \label{intro-rel_der_haus} 
Let $\phi\colon X\to Y$ be a Nash submersion. The complex $\Sdr(\phi)$ has Hausdorff homology spaces. 
\end{introcor}

The second consequence is easiest to present for a slightly modified version of the relative de Rham complex. For a Nash submersion $\phi\colon X \to Y$, we consider the twisted relative de Rham complex $\Sdrt(\phi)$ (see \Cref{def:tw_rel_der_comp}), which coincides with $\Sdr(\phi)$ up to a shift and a twisting by the local systems of relative orientations along $\phi$. As a consequence of the main theorem, we then deduce
\begin{introcor} (\Cref{homology_inv_rel_der})
\label{intro-acyc_fibs}
Let $\phi\colon Y\to X$ and $\psi\colon Z\to X$ be Nash submersions, and let $\alpha \colon Y\to Z$ be a map over $X$ which induces an isomorphism on the homologies of the fibers of $\phi$ and $\psi$ with coefficients in $\RR$. Then, we have a natural (continuous) homotopy equivalence 
\[
\Sdrt(\phi)\simeq \Sdrt(\psi).
\] 
\end{introcor} \label{intro-acyc_fib}
In particular, if $\phi$ is a Nash submersion with contractible fibers, then $\Sdrt(\phi)$ is homotopy equivalent to the complex, concentrated in degree $0$, consisting of Schwartz functions on $Y$. In this case, this equivalence can be shown to be given by the integration of top relative measures along the fibers of $\phi$. Hence, this result allows to "resolve" the Schwartz space of $Y$ by Schwartz spaces of various bundles on $X$. Note also that for general $\alpha$, there is no obvious map between $\Sdrt(\phi)$ and $\Sdrt(\psi)$. However, such a map is given by the (twisted version of) \Cref{intro-rel_der}(1). 
\subsection{Related Work}
The first part of \Cref{intro-rel_der} for the special case where  $Y=\term$, is essentially established in \cite{DerShapPrell}. 
The special case where $\phi$ is a locally trivial fibration of affine Nash manifolds has essentially been carried in \cite{AizDer}. The passage to the general relative case requires different methods, and in particular, the construction of the functor guaranteed in \Cref{intro-rel_der}. Indeed, unlike the absolute case, one can not compare the relative de Rham complex with a smooth counterpart, such as the complex of differential forms with compact support. 

In fact, one can develop a parallel, much simpler relative de Rham theory for submersions of smooth manifolds, and prove an analog of \Cref{intro-rel_der} for the smooth, with or without compact support, sections of the relative de Rham complex of a submersion of Nash manifolds. This version is expected to be much simpler due to the fact that submersions of smooth manifolds are, locally over the source, just projections from a product with $\RR^n$. Moreover, in this simpler case, one can avoid the usage of cosheaves by using the well-developed theory of sections of sheaves with compact support. We shall not pursue this parallel theory here and leave the necessary modifications to the interested reader.

Tempered sections of tempered bundles on Nash manifolds, and more generally, sub-analytic manifolds, are studied in \cite{ShvMan}. One can deduce a version of our results for distributional sections of the de Rham complex using \cite{ShvMan}.  Our pre-dual version, working with cosheaves of topological vector spaces rather than sheaves of real vector spaces, has the advantage of carrying information about the \emph{topology} of the homologies of the complex. Such information is crucial for \Cref{intro-rel_der_haus}. Another advantage of our result is that we identify the \emph{homotopy type} of the relative de Rham complex, and not only its quasi-isomorphism class.

The study of Schwartz spaces from a homological perspective, in the context of Nash stacks, has been initiated in \cite{SchStack}. We expect our methods to be adaptable to this more general framework.

We partially develop a 6-functor formalism for sheaves valued in $\infty$-categories on semi-algebraic spaces during our proof. Treatment of such theory in the more general context of $o$-minimal structures, but restricted to the framework of derived categories, is developed in \cite{edmundo2014six}. 
\subsection{Outline of the Proof}
We shall explain the main ideas behind the proof of \Cref{intro-rel_der} and its corollaries. 

\subsubsection{Constructing the Functor}
It turns out to be more convenient to prove a (stronger) $\infty$-categorical version of \Cref{intro-rel_der} then \Cref{intro-rel_der} itself, and we shall sketch now this stronger version.

First, it is known that for a Nash manifold $X$, the Schwartz space $\Sc(X)$ is the global sections of a cosheaf $\Sc$ on $X$. Moreover, due to the existence of tempered partition of unity, this cosheaf is \emph{acyclic}. To exploit this fact, we introduce an $\infty$-category $\infFre$, consisting of complexes of Fr\'{e}chet spaces, and consider $U\mapsto \Sc(U)$ as a $\infFre$-valued functor on the poset of open semi-algebraic subsets $U\subseteq X$, which lands in complexes concentrated in degree $0$. Then, the acyclicity of $\Sc$ becomes its cosheaf property as a $\infFre$-valued functor. 

We now know what the functor $\Sc$ from \Cref{intro-rel_der} should be on the extensions by $0$ of the constant sheaf, $\RR_U$, for semi-algebraic open sets $U\subseteq X$. Namely, it is just the space of sections $\Sc(U)$. The point is to show that this definition extends overall of $\mathrm{D}_c(X,\RR)$. While this is not completely obvious, since the sheaves $\RR_U$ are not projective objects in the abelian category $\Shv(X,\RR)$, it is a formal consequence of the ($\infty$-categorical) cosheaf condition on $\Sc$. This observation is the content of the leading abstract idea behind the proof.

In complete generality, if $\cC$ is a Grothendieck site and $\cD$ is any nice enough (i.e., presentable) $\infty$-category, we can form the $\infty$-categories $\Shv(\cC,\cD)$ and $\CShv(\cC,\cD)$, consisting of $\cD$-valued sheaves and cosheaves on $\cC$. Under mild assumptions on $\cC$, we recover the (unbounded) derived category of sheaves on $\cC$ as $\Shv(\cC,\Der(\RR))$, where $\Der(\RR)$ is an $\infty$-categorical version of the derived category of $\RR$. 
In fact, the $\infty$-category of cosheaves can be computed, using abstract categorical constructions, from the $\infty$-category of sheaves. Specifically, if $\cA$ is an \emph{$\RR$-linear $\infty$-category}, so that it admits an action of $\Der(\RR)$, then we can identify  $\CShv(\cC,\cA)$ with the "$\RR$-linear internal hom object" $\hom_{\Der(\RR)}(\Shv(\cC,\Der(\RR)),\cA)$. Specializing to the case $\cA := \infFre$ and $\cC$ the site of open sets on a Nash manifolds $X$, we see that we can identify the $\infty$-category $\CShv(\cC,\infFre)$ with the $\infty$-category of certain functors $\Shv(X,\Der(\RR))\to \infFre$. As explaind above, the co-sheaf of Schwartz functions can be considered as an object of $\CShv(X,\infFre)$, and hence corresponds to a functor $\Shv(X,\Der(\RR))\to \infFre$.
Restricting this functor to the sub-($\infty$-)category of constructible sheaves, and passing to the (triangulated) homotopy categories of these stable $\infty$-categories, we get the desired functor in \Cref{intro-rel_der}.

Before we explain why this functor satisfies conditions $(1)$ and $(2)$ of \Cref{intro-rel_der}, we would like to make two remarks about the argument above.

\begin{rmk}
The extension of $\Sc$ from open sets to constructible sheaves involves using homotopy colimits, and this is one reason why the $\infty$-categorical version of the various triangulated categories involved is needed to carry this proof. A classical variant of this argument would be based on resolving every constructible sheaf by a complex of sheaves of the form $\bigoplus_i\RR_{U_i}$, and use such resolutions and the formula $\Sc(\RR_U) = \Sc(U)$ to define $\Sc$ on every complex with constructible cohomologies. Since the sheaves of the form $\RR_U$ are not projective objects, it is then not clear why the result is independent of the resolution, and this independence is essentially what the $\infty$-categorical machinery extracts from the acyclicity of $\Sc$. Note, however, that if we are interested only in the value of the functor $\Sc$ on a specific constructible sheaf, or even a morphism of such, we can simply resolve this sheaf by a complex of sheaves of the type above and apply the functor to this resolution. The $\infty$-categorical language is essential only in order to prove that this computation does not depend on the choices made, and is (coherently) functorial in the sheaf. We prove the existence of this kind of resolutions in \Cref{exs_ps_free_res}.

\end{rmk}

\begin{rmk}
Note that the definitions of $\Sc$ and $\mathrm{D}_c(X,\RR)$ are not entirely compatible. Since Schwartz functions are defined in terms of "growth conditions," they form a cosheaf only for the \emph{restricted topology} on $X$. In other words, we can define $\Sc(U)$ only for a semi-algebraic set $U$, and it satisfies the cosheaf condition only for finite covers. 
On the other hand, the constructible derived category, $\mathrm{D}_c(X,\RR)$, consists of sheaves on $X$ with respect to the classical topology. 
However, this is not a severe issue, as we show in \Cref{comp_const_st} that the two topologies give rise to the same theory of constructible sheaves. 
\end{rmk}

\subsubsection{Relative de Rham Theorem}
We turn to explain why the functor $\Sc$ satisfies condition $(1)$ of \Cref{intro-rel_der}. For several technical reasons, it is worth twisting a little bit both $\Sdr(X)$ and $\phi_!$. We introduce a modified version of the functor $\phi_!$ that we denote by $\phi_\sharp$ (defined in the generality of essential geometric morphisms of $\infty$-topoi in \Cref{def:essential_geo_mor}), and a twited version of the relative de Rham complex, $\Sdrt(X)$ (see \Cref{def:tw_rel_der_comp}). The main advantage of these modified notions is that $\phi_\sharp$ is a left adjoint of $\phi^*$, and we have relative integration map $\int_\phi\colon \Sdrt(\phi)\to \Sc(Y)$.  We then aim to prove, instead of property $(1)$ of \Cref{intro-rel_der}, that 
\begin{equation}
\label{eq:intro_mod_rel_der}
\Sc(\phi_\sharp\RR_X) \simeq \Sdrt(X).
\end{equation}
This is a relative version of the de Rham theorem. Hence, a proof ought to start with a version of \emph{Poincar\'{e} Lemma}. 

In the classical story, Poincar\'{e} Lemma states that the de Rham cohomology of an open cube vanishes. Then, we can deduce the de Rham Theorem by covering a manifold with subsets isomorphic to open cubes. 
To have a relative version of this argument, we need to choose a relative, semi-algebraic version of an open cube. Unlike the absolute case, in the relative one, we can compose maps, and hence we are allowed to restrict to maps of relative dimension one. For such maps, the local model we choose is that of \emph{families of intervals}. Roughly, these are maps $\phi\colon X\to Y$ for which every fiber is an open interval, in a compatible way. 
We then prove, by direct computation entirely analogous for the one in the proof of Poincare Lemma, that for such a map, the integration map $\smallint_\phi\colon \Sdrt(X) \to \Sc(Y)$ is a homotopy equivalence of complexes of Fr\'{e}chet spaces. Since such families of intervals form a basis for the restricted topology on $X$, this gives the local version of the relative de Rham theorem. 

To prove the result for a general Nash submersion, we observe that the relative Poincar\'{e} lemma implies that, on the level of $\infFre$-valued cosheaves, we have $\phi^!\Sc \simeq \Sdrt_{\phi}$. Here, 
\[
\phi^!\colon \CShv(Y,\infFre)\to\CShv(X,\infFre)
\] 
is the pullback functor of cosheaves, analogous to the pullback of sheaves, and $\Sdrt_{\phi}\in \CShv(X,\infFre)$ is a cosheaf version of the twisted relative de Rham complex. 

Finally, we wish to use this fact to deduce the global relative de Rham theorem. Applying the functor $\phi_!$ to the above equivalence, we get 
\[
\phi_!\Sdrt_{\phi,\sE}\simeq \phi_!\phi^!\Sc_\sE.
\]
At this point, it is clear that we need a projection formula for the adjunction $\phi_!\dashv \phi^!$. 
To obtain such a formula, we note that the identification 
\[
\CShv(X,\infFre)\simeq \hom_{\Der(\RR)}(\Shv(X,\Der(\RR)),\infFre)
\] 
endows the $\infty$-category $\CShv(X,\infFre)$ with a $\Shv(X,\Der(\RR))$-linear structure, given by the action on the source. Hence, for a sheaf $\sF$ and a cosheaf $\sG$, we can form their tensor product $\sF\scten\sG$, which is again a cosheaf. 

In a sense, the operation $\sF\scten\sG$ is a localized version of the extension of a cosheaf to a functor valued on $\Shv(X,\Der(\RR))$. Namely, by construction we have $\sF\scten \sG(X)\simeq \sG(\sF)$. Thus, in order to prove \Cref{eq:intro_mod_rel_der}, it suffices to prove a localized version: 
\begin{equation}
\label{eq:intro_cosh_rel_der}
\phi_!\Sdrt_{\phi,\sE}\simeq \phi_\sharp(\RR_X)\scten \Sc_\sE.
\end{equation}
Since $\phi_!\Sdrt_{\phi}\simeq \phi_!\phi^!\Sc$, this reduces to showing that the functor $\phi_!\phi^!$ identifies with tensoring with $\phi_\sharp \RR_X$.  
This statement is essentially obtained from the projection formula for (our modified) proper push forward functor on sheaves by applying the functor $\hom_{\Der(\RR)}(-,\infFre)$. 

\subsubsection{Hausdorffness} 
To prove property $(2)$ of \Cref{intro-rel_der}, we return to the strategy of resolving a constructible sheaves by sheaves of the form $\bigoplus_i\RR_{U_i}$. The idea is that, while this strategy is problematic as a way to \emph{define} $\Sc$ on $\mathrm{D}_c(X,\RR)$, it is an effective way to \emph{compute} it. 
As we show in \Cref{exs_ps_free_res}, every constructible sheaf on a Nash manifold admits a finite resolution by sheaves of this form. 

Using formal properties of the tensor product operation $\scten$, we can then present the complex $(\sF\scten \Sc)(X)$, up to continuous homotopy equivalence, in the form 
\[
\dots \to \bigoplus_{i=1}^{n_{1}}\Sc(U_{1,i}) \to \bigoplus_{i=1}^{n_0}\Sc(U_{0,i}) \to \bigoplus_{i=1}^{n_{-1}}\Sc(U_{-1,i})  \to  \dots
\]
Here, the $j$-th differential is given by multiplication with an $n_j \times n_{j-1}$ real-valued matrix satisfying certain compatibility relation with respect to the $U_{j,i}$-s (see \Cref{def:comb_map} for the precise formulation). This presentation reduces the Hausdorffness of the homologies to the claim that a map of the form 
\[
f\colon \bigoplus_{i=1}^n\Sc(U_i)\to \bigoplus_{j=1}^m\Sc(V_j)
\]
induced from a real-valued matrix as above has a closed image. 

The next step is to identify a closed subset $Z\subseteq X$ which we aim to eliminate, in order to simplify the problem. For detecting the property that $f$ has a closed image, we need to assume specific "niceness" property from that $Z$. We choose $Z$ such that, after taking the quotient of the source and target of $f$ by functions which are flat on $Z$, the map $f$ becomes the tensor product of the topological vector space $\Sc(X)/\Sc(X\setminus Z)$ with a fixed linear map of finite-dimensional vector spaces.
This property is achieved by stratifying $X$ such that all the $U_i$-s and $V_j$-s are unions of strata and taking $Z$ to be a closed stratum. 

The choice of $Z$ as above has the other advantage that eliminating it simplifies $X$, in the sense that $X\backslash Z$ admits a smaller stratification. This property ensures that the process terminates, and hence we can prove by descending induction that $f$ has a closed image.

To show that we can test whether $f$ has a closed image after eliminating $Z$, we develop a general method to study the property that a map of topological vector spaces has a closed image, using filtrations of its source and target. The precise details are rather technical and summarized in \Cref{quasi_pseudo_inverse_then_closed}.

\subsection{Acknowledgments} 
We would like to thank Joseph Bernstein, Tomer Schlank and Clark Barwick for useful discussions. We would like to thank Gal Dor, Lior Yanovski, Dmitry Gourevitch, Shay Ben Moshe and Shaul Barkan for reading an early draft of this paper and for many corrections and suggestions. 
This project emerged from a long correspondence with Yiannis Sakellaridis and Dmitry Gourevitch about Nash stacks; we are grateful for both for years of interesting and fruitful conversations. 
Finally, we would like to thank Maxime Ramzi, Lior Yanovski and Clark Barwick for pointing us to several related references. 

A.A. was partially supported by ISF grant 249/17, and a Minerva foundation grant. 
S.C. is partially supported by the Adams fellowship of the Israel academy of science and humanities,  ISF  grant 249/17 and ERC StG grant 637912.

\section{Semi-algebraic Geometry}
\label{sec:semialg}

This section recalls some basic facts about restricted topological spaces, semi-algebraic spaces, and Nash manifolds. 
Most of the material in this section is very standard and can be found in \cite{Shiota}, \cite{CostelIntro}, and many other sources as well.
We shall use the standard terminology related to semi-algebraic geometry as in \cite{CostelIntro}.  

\subsection{Restricted Topological Spaces}
\label{subsec: restricted topological spaces}
A restricted topology is a variant of a  topology in which only finite unions of open sets are open, see \cite[Definition 3.2.1]{SchNash}). We denote by $\Restop$ the category of restricted topological spaces.
For a  restricted topological space  $X$, we denote by $\Op(X)$ the lattice of open sets in $X$. 

In a restricted topological space, not every set has a closure. 
This motivates the following definition. 
\begin{defn}
\label{closureful_restricted_topological_space} We say that a restricted
topological space $X$ is \tdef{closureful }if every locally closed
subset of $X$ has a closure in $X$. 
\end{defn}
Note that a locally closed subspace of a 
closureful restricted
topological space is closureful.
As for topological spaces, we define the notion of a basis for a restricted topology. 
\begin{defn}
\label{basis_of_restricted_topology}
 Let $X$ be a restricted topological space. A \tdef{basis} of $X$ is a subset $B\subseteq \Op(X)$ such that every open set $U\subseteq X$
is a finite union of elements of $B$. 
\end{defn}

We adopt the following convention:
we say that a property $P$ of open sets is satisfied \tdef{locally} on a restricted
topological space $X$ if there is a basis of $X$ consisting of subsets satisfying $P$.

The restricted topology of a restricted topological space $X$ induces a natural topology on $X$. 
\begin{defn}
    Let $X$ be a restricted topological space. We let $\mdef{X^\topify}$ be the topological space with underlying set $X$ and topology generated from $\Op(X)$. 
\end{defn}

\subsubsection{Inductive Dimension}
Nash manifolds are exceptionally nice restricted topological spaces. One key feature of them is that they are finite-dimensional. To explain this in precise terms, 
we shall recall the notion of the inductive dimension
of a restricted topological space.
\begin{defn}
 Let $S\subseteq X$ be a locally
closed subset of a closureful restricted topological space. The \tdef{boundary
} of $S$ is the subset $\mdef{\partial_{X}S}:=\overline{S}-S$. If
$X$ is clear from the context we denote it simply by $\partial S$.
\end{defn}

Note that $\partial_{X}S$ is closed. Indeed, if $S=U\cap Z$ with
$U$ open and $Z$ closed, then $\partial S=\overline{S}-U$. 
The inductive dimension of a restricted topological space is defined inductively by looking at boundaries of open subsets of $X$.  
\begin{defn}
\label{def:Dimension_Restricted_Topological_Space} Let $X$ be a
closureful restricted
topological space and $d\ge -1$. 
\begin{itemize}
    \item We say that $X$ is of \tdef{inductive dimension $\le d$}
if $X$ is empty, or, for every open subset $U\subseteq X$, the boundary
$\partial U$ is of inductive dimension
$\le d-1$. 
\item We say that $X$ is of inductive dimension $d$ if it
is of inductive dimension $\le d$ but not of inductive dimension $\le d-1$. 
\end{itemize}
  
\end{defn}

The version of the inductive dimension defined above behaves as expected from dimension theory in
many aspects. For example, we have

\begin{prop}
\label{Locally_closed_subset_weakly_smaller_dimension} Let $X$
be a closureful restricted topological space, and let $S\subseteq X$
be a locally closed subspace. If $X$ is of inductive dimension $\le d$
then so is $S$. 
\end{prop}

\begin{proof}
We prove the result by induction on $d$. Since $S$ is the intersection
of an open subset and a closed subset, we may consider the cases of
a closed and an open subspace separately. If $U\subseteq X$ is
open and $V\subseteq U$ is open, then the closure of $V$ in $U$
is just $\overline{V}\cap U$. Hence, 
\[
\partial_{U}V=\partial_{X}V\cap U.
\]
Since $\partial_{X}V$ is of inductive dimension $\le d-1$, by the
inductive hypothesis $\partial_{X}V\cap U$ is of inductive dimension
$\le d-1$, and so $U$ is of inductive dimension $\le d$. 

Now, suppose
that $Z\subseteq X$ is a closed subset. Let $V\subseteq Z$ be
an open subset, and let $U\subseteq X$ be an open subset such that
$U\cap Z=V.$ Then $\partial_{Z}V$ is a closed subset of $\partial_{X}U$.
Since $\partial_{X}U$ is of inductive dimension $\le d-1$, by the
inductive hypothesis, we deduce that $\partial_{Z}V$ is of inductive
dimension $\le d-1$, so that $Z$ is of inductive dimension $\le d$. 
\end{proof}
\begin{corl}
Let $X$ be a closureful restricted topological space, and let $S\subseteq X$
be a locally closed subset. If $X$ is of inductive dimension $\le d$
then $\partial S$ is of inductive dimension $\le d-1$. 
\end{corl}

\begin{proof}
By \Cref{Locally_closed_subset_weakly_smaller_dimension}, $\overline{S}$
is of inductive dimension $\le d$, and so the result follows from
the fact that $\partial_{X}S=\partial_{\overline{S}}S$. 
\end{proof}

The inductive dimension of a restricted topological space is also bounded by the dimensions of member of an open (or closed) cover of it. Namely, 
\begin{prop}
\label{Inductive_dimension_bounded_by_cover} Let $X$ be an
closureful restricted topological space and let  
$
    \{U_{i}\}_{i\in I}
$
be an open, or closed, covering of $X$. Then, $X$ is of inductive dimension $\le d$
if and only if, for every $1\le i\le n$, $U_{i}$ is of inductive
dimension $\le d$.
\end{prop}

\begin{proof}
The ``only if'' part follows from \Cref{Locally_closed_subset_weakly_smaller_dimension}.
We prove the result by induction on $d$. Let $V\subseteq X$ be an
open set. Then 
\[
{\displaystyle \partial_{X}V=\bigcup_{i=1}^{n}\partial_{U_{i}}(V\cap U_{i})}.
\]
Since $\partial_{U_{i}}V$ is of inductive dimension $\le d-1$ and
since the covering $\{\partial_{U_{i}}(V\cap U_{i})\}_{i=1}^n$
is an open covering of $\partial_{X}Z$, by the inductive hypothesis
$\partial_{X}V$ is of inductive dimension $\le d-1$ and so $X$
is of inductive dimension $\le d$. For a closed covering $X=\bigcup_iZ_{i}$
we have
\[
\partial_{X}V\subseteq\bigcup_{i=1}^{n}\partial_{Z_{i}}(V\cap Z_{i})
\]
and we proceed similarly, using \Cref{Locally_closed_subset_weakly_smaller_dimension}. 
\end{proof}

\subsubsection{Stratifications}
To study constructible sheaves on restricted topological spaces, we need the notion of stratification in the context of restricted topology theory. We shall model such stratifications using maps to finite posets. From now on,
we equip a finite poset, $P$, with the restricted topology
\[
    \mdef{\Op(P)}:=\left\{ U\subseteq P:\text{ if }x\in U\text{ and }y>x\text{ then }y\in U\right\}.
\]

\begin{defn} 
    Let $X$ be a restricted topological
    space. 
    A \tdef{stratification} of $X$ is a continuous map
    $\alpha:X\to P$ to a finite poset $P$.
 \end{defn}
In the language of posets, the notion of refinement and the intersection of stratifications take a particularly nice form. 

\begin{defn}
 A \tdef{refinement} of a stratification $\alpha\colon X\to P$ is a factorization 
\[
\alpha:X\oto{\beta}Q\oto{\psi}P.
\]
The \tdef{intersection} of two stratifications
$\alpha_{1}:X\to P_{1}$ and $\alpha_{2}:X\to P_{2}$ is the stratification 
\[
\alpha_{1}\times\alpha_{2}:X\to P_{1}\times P_{2}.
\]

\end{defn}

The restricted topology on a finite poset $P$ has a canonical basis, given by the \tdef{open stars}  
\[
\mdef{p^{\star}}:=\left\{ q\in P:q\ge p\right\}
\quad \forall p\in P.\]
Accordingly, to check that a map $\alpha \colon X\to P$ gives a stratification, it suffices to check that the open stars of elements of $P$ has open preimages under $\alpha$.  
\begin{example}
\label{cov_Strat} 
Let $X$ be a restricted topological
space and let ${\cal U}=\left\{ U_{i}\right\} _{i\in I}$ be a finite
collection of open sets in $X$. We can associate with ${\cal U}$
a stratification of $X$ given by the map $\alpha:X\to\Pow( I),$ 
\[
\alpha(x)=\left\{ i\in I:x\in U_{i}\right\} .
\]
Note that $\alpha$ is continuous. Indeed, for
every $A\in\Pow(I)$ we have 
\[
\alpha^{-1}(A^{\star})=\bigcap_{j\in A}U_{j}\in\Op(X).
\]
\end{example}

Families of stratifications of a space $X$ indexed by a poset can be glued as follows.
\begin{defn}
\label{Amalgamation_Of_Posets} Let $P$ if a finite poset and
let $\left\{ Q_{p}\right\} _{p\in P}$ be a family of posets indexed
by $P$. We denote 
\[
\mdef{\int_{P}Q_{p}dp}:=\left\{ (p,q):p\in P,q\in Q_{p}\right\} 
\]
 with the order given by $(p,q)\le(p',q')$ if $p<p'$ or $p=p'$
and $q\le q'$. 
\end{defn}

\begin{defn}
\label{Amalgamation_Of_Stratifications} Let $\alpha\colon X\to P$
be a stratification of a restricted topological space $X$, and for
each strata $X_{p}$ let $\beta_{p}\colon X_{p}\to Q_{p}$ be a stratification.
We define a new stratification
\[
\mdef{\int_{P}\beta_{p}dp}\colon X\to\int_{P}Q_{p}dp
\]
by 
\[
\left(\int_{P}\beta dp\right)(x)=(\alpha(x),\beta_{\alpha(x)}(x)).
\]
\end{defn}

It is easy to verify that $\int_{P}\beta_{p}dp$ is indeed a stratification. 
Using this construction, we now relate the notion of stratification we use with the more standard one, i.e., a disjoint locally closed cover of $X$ satisfying a frontier condition.

\begin{defn}
    Let $X$ be a closureful restricted topological space. A stratification $\phi\colon X\to P$ is called a \tdef{proper stratification} if it has no empty strata and for every $p\in P$ we have
\[
    \overline{X_p} = \cup_{q\le p} X_q.
\]
\end{defn}

Note that we always have $\overline{X_{p}}\subseteq\bigcup_{q\le p}X_{q}$,
whenever the left-hand side is well defined. However, the converse
is not true in general.
\begin{example}
Let $X=\RR$ endowed with its semi-algebraic topology, let $P=\left\{ 0<1\right\} $
and let $\alpha\colon X\to P$ be the map 
\[
\alpha(x)=\begin{cases}
0 & x\le0\\
1 & x>0
\end{cases}.
\]
The map $\alpha$ is continuous, but 
\[
\overline{\alpha^{-1}(\left\{ 1\right\} )}=\RR_{\ge0}\ne\RR=\alpha^{-1}(\overline{\left\{ 1\right\} }).
\]
so that $\alpha$ is \emph{not} a proper stratification. 
\end{example}

We shall now relate the proper stratifications with the more classical notion of stratification.
\begin{prop}
\label{stratification_frontier_correspondence } Let $X$ be
a closureful restricted topological space.  
The construction \[
(\alpha\colon X\to P)\mapsto \{X_p\}_{p\in P}
\]
gives a bijection between isomorphism classes of proper stratifications $\alpha \colon X\to P$ and finite collections 
$\Theta \subset \Pow(X)$ consisting of mutually disjoint non-empty locally closed subsets of $X$ which satisfy the following properties: 
\begin{itemize}
    \item $\bigcup_{S\in \Theta} S = X$.
    \item For every $S\in \Theta$, the closure $\overline{S}$ is a union of members of $\Theta$. 
\end{itemize}
\end{prop}

\begin{proof}
We shall describe the inverse map. For a collection $\Theta$ as in the statement of the proposition, we define a partial order on $\Theta$ by 
\[
S < S' \iff S\subseteq \partial_X{S'}. 
\]
Consider the map $\alpha \colon X\to \Theta$ taking $x\in X$ to the unique $S\in \Theta$ for which $x\in S$. 
It is easy to see that $\alpha$ is continuous and that the construction $\Theta \mapsto (\alpha\colon X \to \Theta)$ is an inverse to the map $(\alpha\colon X\to P)\mapsto \{X_p\}_{p\in P}.$
\end{proof}

For finite-dimensional restricted topological spaces,  
up to refinement, there is no difference between proper
and general stratifications. To see this, we need some preparation. 
\begin{prop}
\label{Criterion_for_classical_stratification_be_refinement}
Let $X$ be a closureful restricted topological space and let $\alpha\colon X\to P$
be a stratification. Let $\Theta$ be a collection of locally closed
subsets determining a proper stratification $\beta\colon X\to\Theta$.
Then $\beta$ is a refinement of $\alpha$ if and only if, for every
$p\in P$, the stratum $X_{p}$ is a union of members of $\Theta$. 
\end{prop}

\begin{proof}
The ``only if" direction is clear. For the``if" part, suppose that every stratum $X_{p}$ is a union of members of $\Theta$. Let $\psi\colon\Theta\to P$ be
the map determined by the requirement 
\[
\psi(S)=p\iff S\subseteq X_{p}.
\]

We claim that $\psi$ is continuous. Indeed, for every $p\in P$,
we have 
\[
\psi^{-1}(\overline{\left\{ p\right\} })=\left\{ S\in\Theta:\exists q\le p,S\subseteq X_{q}\right\} 
\]
 and since $\bigcup_{q\le p}X_{q}$ is closed, this is a downwardly-closed
subset of $\Theta$. We see that $\psi$ exhibits $\beta$ as a refinement
of $\alpha$. 
\end{proof}
Hence, to refine a stratification $\alpha \colon X\to P$ to a proper one, we only need to find a proper stratification compatible with all the strata of $\alpha$. More generally, we have
\begin{prop} \label{Existence_classical_stratification_compatible_with_collection}
    Let $X$ be a closureful restricted topological space of finite inductive
    dimension.
    Let $\Theta$ be a finite collection of locally closed subsets of
    $X$. There is a proper stratification $\alpha\colon X\to P$ such that every element of $\Theta$ is a union of strata of $\alpha$.
\end{prop}

\begin{proof}
    Assume that $X$ is of inductive dimension $\le d$. We prove the
    result by induction on $d$ and then
    by induction on $\#\Theta$. Choose $S_{0}\in\Theta$, and let $\mdef{U}:=X\setminus\overline{S_{0}}$,
    $\mdef{V}:=X\setminus\overline{U}$ and $\mdef{Z}=\partial U$ so that $X=V\cup Z\cup U$.
    Set 
\[
    \mdef{\Theta_{0}}:=\left\{ S\cap U:S\in\Theta\setminus\left\{ S_{0}\right\} \right\} 
\]
    and 
\[
    \mdef{\Theta_{1}}:=\left\{ S\cap V:S\in\Theta\setminus\left\{ S_{0}\right\} \right\} .
\]
    Since both $\Theta_{0}$ and $\Theta_{1}$ have fewer elements than
    $\Theta$, we can find proper stratifications $\Theta_0'$ and $\Theta_1'$ of $U$ and $V$ compatible with $\Theta_0$ and $\Theta_1$ respectively.
    Consider the following collection of locally closed subsets of $Z$: 
\[
    \mdef{\Theta'}:=\left\{ \overline{S}\cap Z:S\in\Theta'_{0}\right\} \cup\left\{ \overline{S}\cap Z:S\in\Theta'_{1}\right\} \cup\left\{ \overline{S_{0}}\cap Z\right\} \cup\left\{ S\cap Z:S\in\Theta\right\} .
\]
    Then $\Theta'$ is a finite collection of locally closed subsets of
    $Z$, and since $Z$ is of inductive dimension $\le d-1$, by the
    inductive hypothesis on $d$, there is a proper stratification,
    $\Theta''$, of $Z$, such that each element of $\Theta'$ is a union
    of elements of $\Theta''$. We claim that 
\[
    \mdef{\widetilde{\Theta}}:=\Theta'_{0}\cup\Theta'_{1}\cup\Theta''
\]
    is a classical stratification compatible with $\Theta$.
    Indeed, note first that the union of the elements of $\Theta'_{0}\cup\Theta'_{1}\cup\Theta''$
    is $U\cup V\cup Z=X$ and that each element of $\widetilde{\Theta}$
    is locally closed. Moreover, by construction, every element of $\Theta$
    is a union of elements of $\widetilde{\Theta}$. It remains to show
    that the closure of an element of $\widetilde{\Theta}$ is a union
    of members of $\widetilde{\Theta}$. Let $S\in\Theta'_{0}\cup\Theta'_{1}\cup\Theta''$.
    If $S\in\Theta'_{0}$ then 
\[
    \overline{S}=(\overline{S}\cap U)\cup(\overline{S}\cap Z).
\]
    The set $(\overline{S}\cap U)$ is a union of elements of $\Theta'_{0}$
    by construction, and the set $(\overline{S}\cap Z)$ is a union of
    elements of $\Theta''$. A similar argument shows that the closure
    of an element of $\Theta'_{1}$ is a union of elements of $\widetilde{\Theta}$.
    Finally, for $S\in\Theta''$, the closure of $S$ in $X$ coincides with its
    closure in $Z$, which is a union of elements of $\Theta''\subseteq\widetilde{\Theta}$
    by construction. We finish by observing that 
$
    \overline{S_{0}}=V\cup(\overline{S_{0}}\cap Z)
$
    is again a union of elements of $\widetilde{\Theta}$. 
\end{proof}

We now ready to relate the classical and parametrized versions of a stratification.

\begin{thm}
\label{Existance_of_classical_refinement} Let $X$ be a closureful
restricted topological space of finite inductive dimension. Every
stratification of $X$ admits a refinement by a proper stratification. 
\end{thm}

\begin{proof}
Let $\phi\colon X\to P$ be a stratification, and set $\Theta=\left\{ X_{p}\right\} _{p\in P}$.
By \Cref{Existence_classical_stratification_compatible_with_collection}
there is a proper stratification $\beta\colon X\to Q$ compatible
with $\Theta$. By \Cref{Criterion_for_classical_stratification_be_refinement},
the stratification $\beta$ is a refinement of $\alpha$.
\end{proof}

\subsection{Semi-algebraic Sets}
\label{subsec: semi_algebraic_sets}
The basic notion in real semi-algebraic geometry is that of a semi-algebraic
set. We use this standard notion as defined in \cite[\S 2.1.1]{CostelIntro}. Namely, the collection of \tdef{semi-algebraic subsets} of $\RR^n$ is the minimal collection closed under finite unions, intersections and taking complement, containing the sets of the form $\{x\in \RR^n : p(x) <0\}$ for a polynomial function $p\colon \RR^n\to \RR$.

A semi-algebraic subset of $\RR^{n}$ admits a canonical restricted
topology in which the open sets are the subsets given by finitely
many strict polynomial inequalities.
If $X\subseteq\RR^{n}$ is a semi-algebraic set, we denote by $\Op(X)$
this restricted topology, and refer to elements of $\Op(X)$ as \tdef{semi-algebraic open subsets} of $X$.

\subsubsection{Triangulations}


For a finite set $A$, let $\RR[A]$ be the free real vector space spanned by $A$. 
Recall that, if $A$ is a finite set, then a simplicial complex with vertices $A$ is a subset $K\subseteq \Pow(A)$ closed under taking subsets. 

More generally, we define
\begin{defn} \label{def:almost_simp_comp}
An \tdef{almost simplicial complex} on the set of vertices $A$ is a subset of $\Pow(A)$. 

We define the \tdef{geometric realization} of an almost simplicial complex $K\subseteq \Pow(A)$ to be the semi-algebraic subset $
\mdef{\triangle_K}\subseteq \RR[A]$ given by the union of (open) simplices in spanned by elements of $K$. In general, $\Delta_K$ is not a closed or even a locally closed subset of $\RR[A]$. 
\end{defn}
We can consider an almost simplicial complex $K$ as a poset with respect to inclusion. Then, we have a canonical $K$-shaped stratification
$\alpha_{K}:\triangle_{K}\to K$ which takes $x\in \triangle_K$ to the unique simplex $\sigma \in K$ such that $x$ is contained in the interior of the open simplex spanned by the vertices of $\sigma$ in $\RR[A]$.
We refer to the stratification $\alpha_K$ as the \tdef{simplicial stratification} of $\triangle_K$.
Note that the simplicial stratification of $\triangle_{K}$ is proper. 

We now discuss the notion of a triangulation of a semi-algebraic set. 
\begin{defn}
Let $X$ be a semi-algebraic set. 
A \tdef{triangulation} of $X$ is an almost simplicial complex $K$
and a semi-algebraic isomorphism
\[
\psi:\triangle_{K}\iso X
\]
For a stratification $\alpha:X\to P$, we say that the triangulation
$\psi$ is \tdef{compatible} with $\alpha$ if it fits into a commutative
diagram 
\[
\xymatrix{\triangle_{K}\ar^{\alpha_{K}}[r]\ar_{\wr}^{\psi}[d] & K\ar[d]\\
X\ar^{\alpha}[r] & P
}
\]
\end{defn}

Similarly, if $X$ is a semi-algebraic space and $S_{1},...,S_{n}$
are subsets of $X$, we say that a triangulation $\psi\colon\triangle_{K}\oto{\sim}X$
is compatible with the $S_{i}$-s if every $S_{i}$ is a union open
simplices from the triangulation. 

Triangulations always exist for
semi-algebraic sets. 
\begin{prop}
\label{Existence_of_triangulation} Let $X\subseteq\RR^{n}$
be a semi-algebraic set, and let $S_1,...,S_k$ be locally closed subsets of $X$. 
There exists a semi-algebraic triangulation of $X$ compatible with the $S_i$-s. In particular, every stratification of $X$ admits a refinement by the simplicial stratification of a triangulation of $X$.  
\end{prop}

\begin{proof}
If $X$ is compact this is \cite[Theorem 3.12]{CostelIntro}. In the general case, consider the embedding of $\RR^n$ in the unit sphere $S^{n}\subseteq \RR^{n+1}$ by the inverse of the stereographic projection. Since $S^n$ is compact, it admits a triangulation compatible with $X,S_1,...,S_k$. The result follows by intersecting this triangulation with $X$, which is then the realization of an almost simplicial complex compatible with the $S_i$-s.
\end{proof}

Using triangulations, we now show that semi-algebraic sets are nice restricted topological spaces. 
\begin{prop}
\label{Semi_algebraic_set_closureful_finite_dimension} Let $X\subseteq\RR^{n}$
be a semi-algebraic set. Then $X$ is closureful and the dimension of $X$ in the sense of \cite[\S 3.3.1]{CostelIntro} is its inductive dimension. 
\end{prop}

\begin{proof}
Let $V\subseteq X$ be a semi-algebraic locally closed subset. Then
the closure  $\overline{V^\topify}\subseteq X^{\topify}$,
is semi-algebraic (see, e.g., \cite[Corollary 2.2.7]{SchNash}), and hence it is the closure of
$V$ in $X$. If follows that $X$ is closureful.

We now compare the dimension and inductive dimensions of $X$. 
Let $d$ be the dimension of $X$ as a semi-algebraic set. We wish to prove that $X$ is of inductive dimension $d$. We show this by induction on $d$. 

Let $U\subseteq X$ be an open subset. By \Cref{Existence_of_triangulation}, we can find a triangulation $\phi \colon X \iso \triangle_K$ such that $U$ is a union of open simplices of the triangulation. The assumption $\dim(X) = d$ implies that $K$ is $d$-dimensional. The boundary, $\partial U$, is the realization of an almost simplicial complex of dimension at most $d-1$. By induction, this implies that $\partial U$ is of inductive dimension $\le d-1$. Hence, $X$ is of inductive dimension $\le d$. 

For the opposite inequality, note that since $X$ contains an open subset isomorphic as a restricted topological space to $\RR^d$, it suffices to show that $\RR^d$ is of inductive dimension $\ge d$, which is clear.     
\end{proof}

Because of this result, we shall refer to the inductive dimension of a semi-algebraic set simply as its dimension.

\subsection{Semi-Algebraic Spaces}
\label{subsec:semi_algebraic_spaces}
Just like affine varieties are local models for varieties, semi-algebraic sets serve as a local model for semi-algebraic spaces, in the sense of \cite[Definition 5]{SemiAlgRealClosed}.  
\begin{defn} 
A \tdef{semi-algebraic space}
is a restricted topological space $X$ together with a sheaf of $\RR$-valued
functions $\mathcal{SA}_{X}$, locally isomorphic to a semi-algebraic
set with its continuous semi-algebraic functions. We say that
$X$ is \tdef{affine} if it is isomorphic to a semi-algebraic set.  
\end{defn}

 We shall denote by $\Semialg$ the category of semi-algebraic spaces.
Semi-algebraic spaces are well behaved restricted topological spaces. Namely, 
\begin{prop}
\label{Semialgebraic_Spaces_closureful_and_finite_dimensional} 
Let
$X$ be a semi-algebraic space. Then $X$ is closureful of finite
inductive dimension. 
\end{prop}

\begin{proof}
The statement is local over $X$, and hence the result follows from \Cref{Semi_algebraic_set_closureful_finite_dimension}.
\end{proof}

\begin{corl}
\label{Classical_refinement_semialgebraic_space} Every stratification
of a semi-algebraic space admits a proper refinement. 
\end{corl}
\begin{proof}
Since semi-algebraic spaces are closureful of finite inductive dimension, the result follows from \Cref{Existance_of_classical_refinement}.
\end{proof}

We now recall an important feature of maps between semi-algebraic spaces; they
can all be trivialized along a stratification. 
\begin{defn}
Let $\phi\colon X\to Y$ be a morphism of semi-algebraic spaces. A
\tdef{stratified trivialization} of $\phi$ is a pair of compatible
stratifications 

\[
\xymatrix{X\ar^{\phi}[d]\ar^{\alpha}[r] & P\ar[d]\\
Y\ar^{\beta}[r] & Q
}
\]
such that, for every $p\in P$ the map $\phi_{p}\colon X_{p}\to Y_{\psi(p)}$
is isomorphic to a projection $F_{p}\times Y_{\beta(p)}\to Y_{\beta(p)}$
for some semi-algebraic space $F_{p}$. 
\end{defn}

Stratified trivializations exist for maps of semi-algebraic spaces, namely
\begin{prop} 
\label{Trivialization_of_semialgebraic_map} Let $\phi\colon X\to Y$
be a morphism of semi-algebraic spaces. Then $\phi$ admits a stratified
trivialization. 
\end{prop}

\begin{proof}
If $X$ and $Y$ are affine, this is a direct consequence of Hardt Theorem (see \cite[Theorem 4.1]{CostelIntro}). We show how to reduce the general case to this one. Choose an open affine cover $Y = \bigcup_{i\in I} U_i$ and a compatible affine cover $X = \bigcup_{j\in J} V_j$, i.e. such that for every $j$ there is $i$ with $\phi(V_j)\subseteq U_i$. Let $\gamma\colon X\to R$ be the cover stratification associated with the cover $\{V_j\}_{j\in J}$. 

By the special case where $X$ and $Y$ are affine, for every $r\in R$ there is a stratification $\alpha_{r} \colon X_{r} \to P_{r}$ such that each stratum of $\alpha_{r}$ is trivial over its image in $Y$. Let 
\[
\Theta = \{\phi((X_{r})_p) :r\in R, p\in P_r\}.
\]
By \Cref{Existence_classical_stratification_compatible_with_collection}, we can find a stratification $\beta\colon Y\to Q$ which is compatible with $\Theta$. We now take the stratification of $X$ to be the intersection of $\beta\circ \phi$ and $\int_{r\in R} \alpha_r$. It is easy to see that this stratification $\beta$ assemble to a stratified trivialization of $\phi$.     
\end{proof}

\subsection{Nash Manifolds \& Nash Submersions }
\label{subsec:nash_manofolds}
The relative de Rham theorem concerns Nash submersions between Nash manifolds. For this reason, we shall now recall some elementary properties of such.
For a general treatment of Nash manifolds, see \cite[\S 3]{SchNash}. 

\begin{defn} \label{def:Nash_Submersion} 
A morphism $\phi:X\to Y$ of Nash manifolds
is called a \tdef{Nash submersion}, if it is a Nash map which is
a submersion of smooth manifolds.
\end{defn}

Nash submersions have their version of the Implicit Function Theorem (see, e.g.,  \cite[Theorem 2.2]{SchNash}). As a consequence, we have: 
\begin{prop}
\label{Implicit_Function_Theorem} Let $\phi:X\to Y$
be a Nash submersion. There is a cover ${\displaystyle X=\bigcup_{i=1}^{n}U_{i}}$
such that for every $1\le i\le n$, the map $\phi|_{U_{i}}$ is isomorphic
to a composition of an open embedding $U_{i}\into\RR^{n}\times Y$
followed by the projection $\RR^{n}\times Y\to Y$.
\end{prop}

\subsubsection{Families of Intervals}
Our proof of the relative de Rham theorem is conducted by reducing it to a local statement for Nash submersions of relative dimension 1. For this reason, we need a good local model of such Nash submersions.
Our chosen local model is the following
\begin{defn}
Let $\phi\colon X\to Y$ be a Nash
submersion. We say that $\phi$ is a \tdef{family of intervals}
if it has connected fibers and can be written as a composition of
an open embedding $j\colon X\into Y\times\RR$ followed by the projection
$p\colon Y\times\RR\to Y$. A family of intervals is called  \tdef{pointed} if it is endowed with a Nash section.  
\end{defn}
Note that, while one might take the section as part of the data of a pointed family of intervals, we will only need its existence and hence usually ignore it. 

Locally over the source, every Nash submersion of relative dimension
one is a pointed family of intervals. The proof of this fact depends on the following construction:
\begin{defn}
\label{def:Convex_Hull_of_Section} Let $Y$ be a set and let $U\subseteq Y\times\RR$.
Given a section $\nu:Y\to U$ of the form $\nu(x)=(x,\alpha(x))$,
we denote 
\[
\mdef{F_{\nu,U}}:=\left\{ (y,a)\in Y\times\RR:y\times[a,\alpha(y)]\subseteq U\right\} .
\]
\end{defn}

Namely, $F_{\nu,U}$ consist of all points $z\in U$ for which the
vertical segment between $z$ and $\nu(\pi_{Y}(z))$ is contained
in $U$. 
The construction $(U,\nu) \mapsto F_{\nu,U}$ admits several desired properties.
\begin{prop}
\label{properties_of_the_partition_into_intervals} Let $Y$
be a Nash manifold and $U\subseteq Y\times\RR$. Denote by $\pi:Y\times\RR\to Y$
the projection. 
\end{prop}

\begin{enumerate}
\item If $U$ and $\nu$ are semi-algebraic, then so is $F_{\nu,U}$. 
\item If $U$ is open and $\nu$ is continuous, then $F_{\nu,U}$ is open. 
\item If $U$ is open semi-algebraic and $\nu$ is a Nash map then the projection $F_{\nu,U} \to U$ and the section $\nu$
exhibit $F_{\nu,U}$ as a pointed family of intervals over $U$.  
\end{enumerate}
\begin{proof}
(1) Follows directly from the  Zeidenberg–Tarski Theorem (see \cite[Theorem 2.6]{CostelIntro}). 

For (2), let $(y,a)\in F_{a,U}$. we need to find a neighborhood of
$(y,a)$ which is inside $F_{\nu,U}$. For every $b\in[\alpha(y),a]$
there is a neighborhood $V_{b}$ such that $(y,b)\in V_{b}\subseteq U$.
Choosing a finite subcover of the $V_{b}$-s, we can find an open
$V\subseteq Y$ such that $y\in V$ and $V\times[c,d]\subseteq U$
where \[c<\alpha(y)\le a<d.\] Since $\alpha$ is continuous, we can
find an open $V'\subseteq V$ such that $y\in V'$ and $\alpha(y')\in[c,d]$
for every $y\in V'$. This implies that $V'\times[c,d]\subseteq F_{\nu,U}$. 

(3) now follows from the definition of $F_{\nu,U}$ and the previous
items. 
\end{proof}

We can always cover the source of a Nash submersion of relative dimension 1 by sets of the form $F_{\nu,U}$. Namely,
\begin{prop}
\label{Exist_Finite_Cover_With_Good_Sections} Let $Y$ be a Nash
manifold and let $X\subseteq Y\times\RR$ be an open subset. Let $\phi:X\to Y$
denote the projection. Then there is a finite open cover ${\displaystyle Y=\bigcup_{i}U_{i}}$
and sections $\nu_{i}:U_{i}\to X$ such that ${\displaystyle X=\bigcup_{i}F_{\nu_{i},\phi^{-1}(U_{i})}}$.
\end{prop}

\begin{proof}
By \Cref{Trivialization_of_semialgebraic_map}
We can choose a stratified trivialization 
\[
\xymatrix{X\ar^{\alpha}[r]\ar^{\phi}[d] & P\ar^{\psi}[d]\\
Y^{\beta}\ar[r] & Q
}
\]
For every $p\in P$, we have an isomorphism 
\[
X_{p}\simeq Y_{\psi(p)}\times F_{p}
\]
where $F$ a fiber of $\phi$ at a point of $X_{p}$. In particular,
$F$ is isomorphic to a finite union of open intervals. Hence, by
choosing a point in every connected component of $F_{p}$ we ontain
sections $\nu_{p,i}:Y_{\psi(p)}\to X_{p}$ such that \[
{\displaystyle \bigcup_{i}F_{Y_{\psi(p)},\nu_{p,i}}=X_{p}}.
\]
By \cite[Proposition A.0.3]{SchNash},
and \Cref{Existance_of_classical_refinement}
we can refine $\beta$ to a proper stratification with Nash strata,
and such that the restriction of $\nu_{p,i}$ to every stratum of the
refinement is a Nash map. By
restricting the $\nu_{p,i}$-s to the strata of such a refinement,
we obtain a stratification $\beta':Y\to Q'$ with Nash strata and
for every $q\in Q'$ a set of sections 
\[
\nu_{q,i}:Y_{q}\to X\times_{Y}Y_{q}
\]
such that ${\displaystyle X=\bigcup_{i,q}F_{Y_{q},\nu_{q,i}}}$. By
\cite[Proposition A.0.4]{SchNash} there is a finite collection of open sets $U_{q,i,j}\subseteq Y$
such that ${\displaystyle Y_{q}\subseteq\bigcup_{i,q}U_{q,i,j}}$
and for every $j$ the section $\nu_{q,i,j}|_{U_{q,i,j}\cap Y_{q}}$
extends to a Nash section 
\[
\widetilde{\nu}_{q,i,j}:U_{q.i,j}\to X\times_{Y}U_{q,i,j}.
\]
By the construction, we now have 
\[
    {\displaystyle X=\bigcup_{q,i,j}F_{U_{q,i,j},\tilde{\nu}_{q,i,j}}}.
\] 
\end{proof}

We are ready to give the desired local model for Nash submersions of relative dimension $1$. 

\begin{thm}
\label{fam_int_base}
Let $\phi \colon X\to Y$ be a Nash submersion of relative dimension one. The open sets $U\subseteq X$ for which the restriction $\phi_U \colon U\to \phi(U)$ is a pointed family of intervals, form a basis for the restricted topology on $X$. 
\end{thm}

\begin{proof}
The statement is local over $X$, and hence, by the Implicit Function Theorem (\Cref{Implicit_Function_Theorem}), we may assume that $\phi$ is a composition of an open embedding $X \into Y\times \RR$, followed by the projection $Y\times \RR \to Y$. Moreover, replacing $X$ with an open subset if necessary, it suffices to show that $X$ itself can be covered by open sets, which are pointed families of intervals over their projections into $Y$. 
By  \Cref{Exist_Finite_Cover_With_Good_Sections}, we can
find a finite cover ${\displaystyle Y=\bigcup_{i=1}^{n}U_{i}}$ and
sections $\nu_{i}:U_{i}\to X$ such that \[
{\displaystyle X=\bigcup_{i=1}^{n}F_{\nu_{i},\phi^{-1}(U_{i})}}.
\]
The result now follows from the fact that $F_{\nu_{i},\phi^{-1}(U_{i})}$
is a pointed family of intervals over $U$ (see \Cref{properties_of_the_partition_into_intervals}(3)).
\end{proof}

\section{\texorpdfstring{$\infty$-}
{Infinity}-Topoi}
\label{sec:infty_top}
After we introduced the geometric objects that we study, we now turn to discuss sheaves on them. We develop the theory in the abstract framework of $\infty$-topoi.
Topos theory can be seen as a ``coordinate-free'' way to study sheaf theory on sites. Similarly, $\infty$-topoi provide a coordinate-free framework for studying sheaves valued in an $\infty$-category. We first remind the reader what $\infty$-categories are. 

$\infty$-categories are a generalization of the notion of a category in which one may have homotopies between morphisms. 
Our main source for $\infty$-category theory is \cite{HTT}.
In particular, we use the model of \tdef{quasi-categories}, introduced by Joyal (see \cite[\S 1.4]{JoyalQuas}).
The role of the category of sets for classical category theory is played by the $\infty$-category of spaces in $\infty$-category theory. This can be seen, for example, as the quasi-category of simplicial sets satisfying the Kan condition (see \cite[\S 1.3]{JoyalQuas}). We denote the $\infty$-category of spaces by $\mdef{\Spaces}$. 

An $\infty$-category $\cC$ has an underlying space of objects $\mdef{\cC^\simeq} \in \Spaces$. Moreover, every pair of objects $X,Y$ in $\cC$, has a  space of morphisms  
$\mdef{\Map_{\cC}(X,Y)}\in\Spaces$. On one hand, one can work with $\infty$-categories as with ordinary categories, e.g. we can define limits and colimits (\cite[\S4]{HTT}), adjunctions between functors (\cite{nguyen2020adjoint}), presentable $\infty$-categories (\cite[\S5]{HTT}), symmetric monoidal structures (\cite[\S2]{HA}) and so on. On the other hand, the possibility of higher homotopies inside mapping spaces allows one to model derived phenomena within this theory without referring to an abelian category or ordinary functors in the process. 

With this in mind,
recall the notion of an $\infty$-topos.
\begin{defn}(\cite[Definition 6.1.0.4]{HTT})
An \tdef{$\infty$-topos} is a presentable $\infty$-category which is a left exact localization of a presheaf category.
A \tdef{geometric morphism} $\phi\colon \fX\to\fY$ of $\infty$-topoi is an adjunction $\phi^{*}\colon\fY\adj\fX\colon\phi_{*}$
for which $\phi^{*}$ is left exact. 
\end{defn}
Let $\Prl$ be the $\infty$-category of presentable $\infty$-categories and left adjoint functors beween them. We denote by $\mdef{\RTop}\subseteq\Prl^{\op}$ the (non-full) subcategory of
$\infty$-topoi with geometric morphisms between them.
For $\phi \colon \fX \to \fY \in \RTop$, we refer to $\phi_{*}$ as the functor of \tdef{pushforward} and $\phi^{*}$ as the functor of \tdef{pullback}
along the geometric morphism $\phi$.
The following example of $\infty$-topoi is essentially the only one we shall use.
\begin{example} \label{sheaves_on_sites}
    Let $\cC$ be an ($\infty$-)category endowed with a Grothendieck topology  and let 
    \[
        \mdef{\Shv(\cC)}\subseteq \PShv(\cC)
    \] 
    be the full subcategory spanned by the presheaves $\sF\colon \cC^\op \to \Spaces$ satisfying the following property 
    \begin{itemize}
        \item[(*)] For every object $U\in \cC$ and every covering $\cU = \{U_i\}_{i\in I}$ of $U$ in $\cC$, the canonical map 
        \[
            \sF(U)\to      \invlim\left(\xymatrix{\prod\limits_{i\in I}\sF(U_i)\ar@<0.5ex>[r] \ar@<-0.5ex>[r] &  \prod\limits_{(i,j)\in I^2} \sF(U_i\times_U U_j)\ar@<1ex>[r] \ar[r] \ar@<-1ex>[r] & \dots}\right)
        \]
        is an isomorphism.
    \end{itemize}
    The $\infty$-category $\Shv(\cC)$ is an $\infty$-topos, that we refer to as the $\infty$-topos of \tdef{sheaves} on $\cC$. 
\end{example}
The fully-faithful embedding $\Shv(\cC)\into \PShv(\cC)$ admits a left exact left adjoint $L_\cC\colon \PShv(\cC) \to \Shv(\cC)$, known as the \tdef{sheafification functor}. The next example will be used in the theory of locally constant sheaves. 
\begin{example} 
For a space $A\in \Spaces$, the $\infty$-category $\Spaces_{/A}\simeq \Fun(A,\Spaces)$ is an $\infty$-topos that we denote again by $\Shv(A)$. 
\end{example}
For every $\infty$-topos $\fX$ there is a unique\footnote{Here, and in the rest of this note, uniqueness is always ``up to contractible space of choices".} geometric morphism $\mdef{\Gamma} \colon \fX \to \Spaces,$
which we refer to as the geometric morphism of \tdef{global sections}.
Another important example of geometric morphisms is the following.
\begin{defn}
(\cite[\S 6.3.5]{HTT}) An \tdef{\'{e}tale} geometric morphism is a geometric morphism
of the form $j\colon\fX_{/U}\to\fX$, for
which $j^{*}$ is the right adjoint of the forgetful functor $j_{!}\colon\fX_{/U}\to\fX$. 
We say that the \'{e}tale geometric morphism $j$ is an \tdef{open embedding} if $j_*$ is fully faithful, or equivalently, if $U$ is $(-1)$-truncated in $\fX$. In this case, we refer to $\fX_{/U}$ as an \tdef{open subtopos} of $\fX$.  
\end{defn}

For an open subtopos $\fX_{/U}\subseteq \fX$, we denote by $\mdef{\fX\setminus U} \subseteq \fX$ the full subcategory spanned by the objects $V\in \fX$ for which the projection $V\times U \to U$ is an isomorphism. Equivalently, these are the objects for which $j^*V\simeq \term \in \fX_{/U}$, where $j\colon \fX_{/U}\to \fX$ the \'{e}tale geometric morphism associated with $U$. 

The fully faithful embedding $\fX\setminus U\into \fX$ is the pushforward along a geometric morphism of $\infty$-topoi  (see \cite[Proposition 7.3.2.3]{HTT}), that we refer to as the \tdef{closed complement} of $U$ in $\fX$. A geometric morphism of the form $\fX\setminus U\to \fX$ is called a \tdef{closed immersion} of $\infty$-topoi. 

As for classical sheaves on a topological space, an $\infty$-topos is glued from an open subtopos and its closed complement. 
To make this statement precise, we recall the notion of recollement (or gluing) of $\infty$-categories.

\begin{defn}
    Let $\cC$ be an $\infty$-category which admits finite limits. A pair of left exact localizations $j^*\colon \cC \to \cC_U$ and $i^*\colon \cC \to \cC_Z$ with (fully faithful) right adjoints $i_*$ and $j_*$ exhibit $\cC$ as a \tdef{recollement} of $\cC_U$ and $\cC_Z$ if $j^*$ and $i^*$ are jointly conservative and the functor $j^*i_*\colon \cC_Z \to \cC_U$ takes every object to the final object.
\end{defn}

Intuitively, the categorical notion of recollement expresses the idea that $\cC$ decomposes into an open stratum, $\cC_U$, and a closed stratum, $\cC_Z$. The following justifies this intuition. 

\begin{prop} \label{top_rec}
    Let $\fX$ be an $\infty$-topos and $U\in \fX$ be a $(-1)$-truncated object of $\fX$. Then $\fX$ is a recollement of $\fX_{/U}$ and $\fX\setminus U$.
\end{prop}

\begin{proof}
    Let $i\colon \fX\setminus U \to \fX$ and $j\colon \fX_{/U}\to \fX$ denote the canonical geometric morphisms. Then, by definition, the functors $i^*$ and $j^*$ are left exact and $j^*i_*$ is terminal. Hence, it would suffice to show that $i^*$ and $j^*$ are jointly conservative. This is \cite[Lemma A.5.11]{HA}.
\end{proof}

\subsubsection{Sheaves on $\infty$-Topoi}
The prototypical example of an $\infty$-topos is that of sheaves of \emph{spaces} on a Grothendieck site. However, the language of $\infty$-topoi allows us to model sheaves with general coefficients. 
Recall from \cite[\S4.8]{HA} that the $\infty$-category $\Prl$ of presentable $\infty$-categories admits a symmetric monoidal structure, called the \tdef{Lurie tensor product}, for which 
\[
    \cC \otimes \cD \simeq \Fun^R(\cC^\op,\cD).
\]
Here, $\Fun^R(\cC^\op,\cD)\subseteq \Fun(\cC^\op,\cD)$ is the full subcategory spanned by the functors which are right adjoints, or equivalently, accessible and limit-preserving.
 
\begin{defn}
    Let $\fX$ be an $\infty$-topos and $\cD\in \Prl$. We define the $\infty$-category of \tdef{$\cD$-valued sheaves} on $\fX$ by 
\[
    \mdef{\Shv(\fX,\cD)}:=\fX \otimes \cD\simeq \Fun^R(\fX^\op,\cD).
\]
\end{defn}
In other words, a $\cD$-valued sheaf on $\fX$ is a limit-preserving accessible functor $\sF\colon \fX^\op \to \cD$.

\begin{rmk}
    Though the definition of $\Shv(\fX,\cD)$ makes sense for arbitrary presentable $\infty$-category $\cD$, it turns out that the theory of sheaves behaves nicer when restricted to compactly generated coefficients. Since this generality suffices for our needs, we shall assume that $\cD$ is compactly generated whenever convenient. 
\end{rmk}

A geometric morphism $\phi\colon \fX \to \fY$ gives an adjunction 
\[
    \Shv(\fY,\cD)\adj \Shv(\fX,\cD).
\]
We shall abuse notation and denote this adjunction again by $\phi^*\dashv \phi_*$.
 
\begin{example}
    Let $\cC$ be an ($\infty$-)category endowed with a Grothendieck topology, and let $\fX = \Shv(\cC)$. Then, the sheafification functor $L_\cC\colon \PShv(\cC) \to \Shv(\cC)$ induces a localization   
    \[
        \Fun(\cC^\op,\cD)\simeq \PShv(\cC)\otimes\cD \oto{L_\cC \otimes \Id_\cD} \fX \otimes \cD \simeq \Shv(\fX,\cD). 
    \]
        Via this localization, we can identify $\Shv(\fX,\cD)$  with the full subcategory of $\Fun(\cC^\op,\cD)$ spanned by the functors satisfying the same descent condition as in \Cref{sheaves_on_sites}. 
\end{example}

In general, the fact that the $\infty$-category $\Shv(\fX,\cD)$ is expressible as the tensor product $\fX\otimes\cD$ allows promoting many properties of $\infty$-topoi to properties of sheaf categories with more general coefficients. As a first example, we can define the tensor product of sheaves using the cartesian product in $\infty$-topoi as follows.

For a symmetric monoidal $\infty$-category $\cC$, we denote by $\Calg(\cC)$ the $\infty$-category of commutative (a.k.a $\EE_\infty$) algebras in $\cC$. For $\cC = \Prl$, we can identify $\Calg(\Prl)$ with the $\infty$-category of presentable symmetric monoidal $\infty$-categories in which the tensor product distributes over all colimits. Since colimits in an $\infty$-topos are universal (\cite[Theorem 6.1.0.6]{HTT}), an $\infty$-topos can be regarded as an object of $\Calg(\Prl)$ via the symmetric monoidal structure given by cartesian products.  

Since the tensor product in $\Prl$ induces the coproduct in $\Calg(\Prl)$, for $\cD\in \Calg(\Prl)$ we can endow $\Shv(\fX,\cD)\simeq\fX\otimes\cD$ with a canonical symmetric monoidal structure. For $\sF,\sF'\in \Shv(\fX,\cD)$ their tensor product $\sF\otimes \sF'$ is the sheafification of the presheaf given by $U\mapsto\sF(U)\otimes \sF'(U)$.  

Pullbacks along geometric
morphisms are left exact and hence canonically refine to symmetric monoidal functors. More precisely, we obtain a functor
\[
\Shv(-,-)\colon \RTop^{\op}\times\Calg(\Prl)\to\Calg(\Prl).
\]

Another application of the Lurie tensor product in sheaf theory is the extension of  \Cref{top_rec} to sheaves with more general coefficients.

\begin{lem}
    \label{rec_ten}
    Let $\cC\in \Prl$, and let $i^*\colon \cC\to \cC_Z$ and $j^*\colon \cC \to \cC_U$ be a pair of accessible localizations which exhibit $\cC$ as a recollement of $\cC_U$ and $\cC_Z$. For every compactly generated $\infty$-category $\cD$, the pair of localizations $i^*\otimes \Id_\cD$ and $j^*\otimes \Id_\cD$ exhibit $\cC\otimes \cD$ as a recollement of $\cC_U \otimes \cD$ and $\cC_Z \otimes \cD$. 
\end{lem}

\begin{proof}
    Since $\cD$ is compactly generated, we have $\cD \simeq \Ind(\cD_0)$ for a small $\infty$-category $\cD_0$ which admits finite colimits. By \cite[Proposition 2.3.2]{Edodi}, we can identify the functor of tensoring with $\cD$ with the functor $\Fun^\lex(\cD_0^\op,-)$ of left exact functors from $\cD^\op$. Since $i^*$ and $j^*$ are left exact, via this identification, the functors $i^*\otimes \Id_{\cD}$ and $j^*\otimes \Id_{\cD}$ identify with post-composing with $i^*$ and $j^*$ respectively, so that they are jointly conservative and left exact. Moreover, since the right adjoint of $i^*\otimes \Id_\cD$ is given by post-composition with $i_*$, we immediately get that the composition of $j^*\otimes \Id_\cD$ with the right adjoint of $i^*\otimes \Id_\cD$ maps every object to the final object.  
\end{proof}

As a result, we obtain 
\begin{corl}\label{shv_rec}
    Let $\fX$ be an $\infty$-topos, and let $\cD$ be a compactly generated $\infty$-category. For every $(-1)$-truncated object $U\in \fX$, the $\infty$-category $\Shv(\fX,\cD)$ is a recollement of $\Shv(\fX_{/U},\cD)$ and $\Shv(\fX\backslash U,\cD)$.   
\end{corl}

\begin{proof}
    By \cref{top_rec}, $\fX$ is a recollement of $\fX_{/U}$ and $\fX\setminus U$. Hence, the result follows from \Cref{rec_ten}.
\end{proof}

The property of an $\infty$-category $\cC$ being a recollement of $\cC_U$ and $\cC_Z$ has several consequences which will be helpful for us in the context of constructible sheaves, so we record them now.

\begin{prop} \label{rec_props}
    Let $\cC\in \Prl$ and let 
    $i^*\colon \cC \to \cC_U$ and $j^*\colon \cC \to \cC_Z$ be accessible left exact localizations which exhibit $\cC$ as a recollement of $\cC_U$ and $\cC_Z$. Then 
\begin{itemize}
    \item The functor $j^*$ admits a left adjoint $j_!\colon \cC_U \to \cC$, and the composition $i^*j_!$ is the initial functor.
    \item If $\cC$ is pointed, the functor $i_*$ admits a right adjoint $i^!\colon \cC \to \cC_Z$.  Moreover, we have (co)fiber sequences
    \[
        j_!j^* \to \Id_{\cC} \to i_*i^*
    \]
    and 
    \[
        i_*i^! \to \Id_{\cC} \to j_*j^*. 
    \]
\end{itemize}
\end{prop}

\begin{proof}
    The existence of a left adjoint to $j^*$ follows from \cite[Corollary A.8.13]{HA}. The fact that $i^*j_!$ is initial follows from the assumption that its right adjoint, $j^*i_*$, is terminal. 
    If $\cC$ is pointed, the existence of a right adjoint to $i_*$ follows from \cite[Remark A.8.5]{HA}. 

    We now prove the existence of the required cofiber (and fiber) sequences.
    We consider the first sequence, as the proof for the second is identical.  

    Since $i^*j_!$ is the initial functor, the composition of the counit $j_!j^*\to \Id$ with the unit $\Id \to i_*i^*$ factors canonically through the zero object of $\cC$. Hence, it remains to show that the resulting square 
\[
    \xymatrix{
    j_! j^*X \ar[r]\ar[d] & X \ar[d] \\ 
    0 \ar[r]& i_*i^*X
    }
\]
 is both a pushout and a pullback square for every $X\in \cC$. This property can be checked after applying the jointly conservative, left exact, colimit preserving functors $i^*$ and $j^*$. After applying $i^*$ both vertical maps become isomorphisms, and after applying $j^*$ both horizontal maps become isomorphisms, so the resulting squares are indeed pushout and pullback squares. 
\end{proof}

\subsection{Essential Geometric Morphisms}

By definition, the pullback along a geometric morphism admits a right adjoint. However, it usually fails to preserve infinite limits and hence does not admit a left adjoint. This fact motivates the following

\begin{defn} \label{def:essential_geo_mor}
Let $\phi \colon \fX \to \fY$ be a geometric morphism. We say that $\phi$ is \tdef{essential} if the pullback functor $\phi^*\colon \fY \to \fX$ admits a left adjoint $\phi_\sharp \colon \fX \to \fY$. 
\end{defn}

\begin{example}
The terminal geometric morphism $\Gamma\colon \fX \to \Spaces$ is essential if and only if $\fX$ is locally of constant shape, in the sense of \cite[Definition A.1.1]{HA}. We shall say that $\fX$ is essential in this case.
\end{example}

\begin{example}\label{etale_then_essential} 
    An \etale geometric morphism is essential.
    Moreover, if $j\colon\fX_{/U}\to\fX$ is
    the \etale geometric morphism associated with $U\in \fX$, then $j_{\sharp}$ coicides with the forgetful functor 
    $j_! \colon \fX_{/U} \to \fX$. 
\end{example}

Recall that, for an $\infty$-category $\cC$ which admits finite limits, the pro-category of $\cC$ is defined by 
\[
    \mdef{\Pro(\cC)}:=\Ind(\cC^\op)^\op\simeq\Fun^\lex(\cC,\Spaces)^\op.
\]
For a general geometric morphism $\phi\colon \fX \to \fY$, even if $\phi^*$ does not admit a left adjoint, it do so after passing to $\Pro$-categories (see, e.g., the discussion after  \cite[Remark 2.2]{hoyois2018higher}). We shall abuse notation and denote the left adjoint of 
\[
\Pro(\phi^*)\colon \Pro(\fY) \to \Pro(\fX)
\] 
again by $\phi_\sharp$.

Intuitively, the functor $\phi_\sharp$ models the homotopy type of the fibers of $\phi$. This idea is formalized using  (relative) \emph{shape theory}.

\begin{defn}
\label{shape} (Compare \cite[\S A.1.1]{HA}) Let $\fX$ be an $\infty$-topos
and let $\phi\colon\fX\to\fY$ be a geometric morphism.
Let $\Gamma\colon\fX\to\Spaces$ be the terminal geometric
morphism. The \tdef{shape} of $\phi$ is the pro-object  
\[
\mdef{|\phi|}:= \Gamma_{*}\phi^{*}\qin \Pro(\fY).
\]
\end{defn}
We say that $\phi$ is of \tdef{constant shape} if $|\phi|$ is pro-constant, i.e., if it is co-representable by an object of $\fY$. 

Note that for every geometric morphism $\phi\colon \fX\to \fY$, we have $
|\phi| \simeq \phi_\sharp (\term_\fX)$ in $\Pro(\fY)$.
For the terminal geometric morphism $\Gamma\colon \fX\to \Spaces$, we denote $\mdef{|\fX|}:=|\Gamma|$ and say that $\fX$ is of constant shape if $\Gamma$ is.

The formation of shape provides a functor $|-|\colon \RTop_{/\fY}\to\Pro(\fY)$. We shall use its functoriality on the level of homotopy categories, which can be described as follows.
For a commutative diagram 
\[
\xymatrix{\fX_{0}\ar^{\psi}[rr]\ar^{\alpha}[rd] &  & \fX_{1}\ar^{\beta}[ld]\\
 & \fY
}
\]
in $\RTop$, we have a unit natural transformation
\[
\beta_{*}\beta^{*}\to\beta_{*}\psi_{*}\psi^{*}\beta^{*}\simeq\alpha_{*}\alpha^{*}
\]
Post-composing with $\Gamma_{*}\colon\fY\to\Spaces$ we get
a map 
$|\alpha|\to|\beta|$ in $\Pro(\fY)\subseteq \Fun(\fY,\Spaces)^\op$. 

Using the language of shape theory, we can now roughly describe essential geometric morphisms as those which, locally over the source, have constant shape. 
More presicely, we have
\begin{prop}
\label{essential_locally_constant_shape}
Let $\phi \colon \fX \to \fY$ be a geometric morphism and let $\fB\subseteq \fX$ be a full subcategory which generates $\fX$ under small colimits. The geometric morphism $\phi$ is essential if and only if, for every $U\in \fB$, the composition
\[
\fX_{/U} \to \fX \oto{\phi} \fY
\]
is of constant shape.
\end{prop}

\begin{proof}
The shape of the composition 
$\fX_{/U} \to \fX \oto{\phi} \fY$ is isomorphic to  $\phi_\sharp(U)$ (see \Cref{etale_then_essential}).  Hence, if $\phi$ is essential, the shape of this composition is pro-constant. 

 Conversely, assume that for every $U\in \fX$, the shape of the composition above is pro-constant. Then, the functor $\phi_\sharp \colon \Pro(\fX) \to \Pro(\fY)$ takes all the elements of $\fB$ into $\fY$. Since $\fY$ is closed under small colimits in $\Pro(\fY)$ and $\fB$ generates $\fX$ under small colimits, we deduce that $\phi_\sharp$ restrict to a functor $\fX \to \fY$.
\end{proof}

\subsubsection{Local Systems}
Among all sheaves, locally constant ones play an important role in analyzing the homotopy type of a topological space. For example, local systems on a topological space valued in a 1-category $\cD$ coincide with $\cD$-valued representations of its fundamental group. This relation admits an analog in higher category theory that we shall now discuss.

\begin{defn}
    Let $\fX$ be an $\infty$-topos and let $\cD\in\Prl$. For $X\in \cD$, we refer to $\Gamma^*X\in \Shv(\fX,\cD)$ as the \tdef{constant sheaf} on $\fX$ with value $X$.
    A $\cD$-valued \tdef{local system}
    on $\fX$ is a sheaf $\sF\in\Shv(\fX,\cD)$ satisfying the following property: 
    \begin{itemize}
        \item[(*)] There is a set of objects $\{U_\alpha\}_{\alpha \in I} \subseteq \fX$ which cover the terminal object of $\fX$, and the restriction $\sF|_{U_\alpha}\in \Shv(\fX_{/U_\alpha},\cD)$ is a constant sheaf for every $\alpha \in I$.
    \end{itemize}
\end{defn}
We refer to a collection $\{U_\alpha\}_{\alpha \in I}$ as above as a  \tdef{trivializing cover} for $\sF$, and denote the full subcategory of $\fX$ spanned by the local systems by $\mdef{\Loc(\fX,\cD)}$. 
For $\cD=\Spaces$, we denote $\mdef{\Loc(\fX)}:=\Loc(\fX,\Spaces)$

While the $\infty$-category local systems on an $\infty$-topos $\fX$ are not determined by the fundamental group of $\fX$,
we will see that local systems on an \emph{essential} $\infty$-topos can be described in terms of the \emph{shape}, which we may therefore think of as the ''fundamental $\infty$-groupoid'' of $\fX$.

For an essential $\infty$-topos $\fX$, the functor $\Gamma_\sharp \colon \fX \to \Spaces$ refines to a left adjoint functor \[
\psi_\sharp \colon \fX \to \Spaces_{/|\fX|}\simeq  \Shv(|\fX|),\]
whose right adjoint consisting of the pullback along an essential geometric morphism
\[
\mdef{\psi} \colon \fX \to \Shv(|\fX|).
\]
We would like to show that $\psi^*$ identifies local systems on $\fX$ with functors $|\fX|\to \cD$. Before we show this, we would like to remark on the dependence of $\Shv(\fX,\cD)$ on $\cD$. 

Recall that we have $\Shv(\fX,\cD)\simeq \Fun^R(\fX^\op,\cD)$. In particular, a right adjoint functor $R\colon \cD_0\to \cD_1$ induces a functor 
\[
\mdef{R_*}\colon \Shv(\fX,\cD_0)\to \Shv(\fX,\cD_1)
\]
given by
\[
(R_*\sF)(U)= R(\sF(U)).
\]
For an essential geometric morphism $\phi\colon \fX \to \fY$, we have $\phi^*\sF(U)\simeq \sF(\phi_\sharp U)$, and since pre-composition and post-composition commute, we see that the canonical (Beck-Chevalley) map $\phi^*R_*\to R_*\phi^*$ is an isomorphism.   
Returning to our task, we have

\begin{prop} \label{Local_systems_sheaves_on_shape}
    Let $\fX$ be an essential $\infty$-topos, and let $\cD\in \Prl$. The functor 
\[
    \psi^*\colon \Shv(|\fX|,\cD) \to \Shv(\fX,\cD)
\]
    is fully faithful, with essential image the $\cD$-valued local systems. In particular, we have an equivalence $\Loc(\fX,\cD)\simeq \Loc(\fX)\otimes \cD$.
\end{prop}

\begin{proof}
In the case where $\cD = \Spaces$ this follows from \cite[Theorem A.1.15]{HA}. We now show how to reduce the general case to this one. 
First, the functor $\psi^*\colon \Shv(|\fX|)\to \fX$ is fully faithful and admits a left adjoint. Tensoring with $\cD$, we deduce from \cite[Lemma 5.2.1]{AmbiHeight} that the functor $\psi^*\colon \Shv(|\fX|,\cD)\to \Shv(\fX,\cD)$ is fully faithful and admits a left adjoint. 
It remains to show that the essential image of $\psi^*$ consists of precisely the $\cD$-valued local systems. 

If $\cD\simeq \PShv(\cC)$ for a small $\infty$-category $\cC$, then we have 
\[
\Shv(\fX,\cD)\simeq \Fun(\cC^\op,\fX)
\] 
and 
\[\Shv(|\fX|,\cD)\simeq\Fun(\cC^\op,\Shv(|\fX|)).
\]
Via this identifications, the functor $\psi^*\colon \Shv(|\fX|,\cD) \to \Shv(\fX,\cD)$ is given by post-composing with $\psi^*\colon \Shv(|\fX|)\to \fX$. Since $\Loc(\fX)$ is closed under all limits and colimits in $\fX$, a functor $\cC^\op \to \fX$ corresponds to a $\PShv(\cC)$-valued local system, if and only if its value at every object of $\cC$ is a local system. Since $\psi^*\colon \Shv(|\fX|)\to \Shv(\fX)$ is fully faithful, a functor $\cC^\op\to \fX$ factors through $\psi^*\colon \Shv(|\fX|) \to \fX$ if and only if all its values belong to the essential image of $\psi^*$. Hence, the claim for $\cD\simeq \PShv(\cC)$ follows from the case $\cD\simeq \Spaces$. 

For general $\cD\in \Prl$ we have an accessible localization $L\colon \PShv(\cC)\to \cD$. Choose such a localization $L$ and consider its right adjoint $R\colon \cD \into \PShv(\cC)$. We claim that $R_*\sF$ is a local system. Indeed, let $\{U_i\}_{i\in I}$ be a trivializing cover for $\sF$. Replacing $\fX$ with $\fX_{/U_i}$ for $i\in I$, we may assume that $\sF$ is constant, say, $\sF\simeq \Gamma^*X$. But then, since $\fX_{/U_i}$ is essential, 
\[
R_*\sF\simeq R_*\Gamma^*X\simeq \Gamma^*R_*X 
\]
which is a (constant) local system. 

By the case $\cD \simeq \PShv(\cC)$, we see that $R_*\sF \simeq \psi^*\sG$ for some $\sG\in \Shv(|\fX|,\PShv(\cC))$. Localizing both sides back to $\Shv(\fX,\cD)$ we deduce that 
\[
\sF \simeq L(\psi^*\sG)\simeq \psi^*L(\sG),
\]
so that it belongs to the essential image of $\psi^*$. 
\end{proof}

\subsection{Locally Contractible Geometric Morphisms}

In the classical 6-functor formalism, one essential property of the pushforward with proper support is the \emph{projection formula}. Namely, for a map $\phi\colon X\to Y$ of paracompact topological spaces, the functor 
$
\phi_! \colon \Shv(X)\to \Shv(Y)
$ 
is $\Shv(Y)$-linear. 
In higher topos theory, we may view the left adjoint $\phi_\sharp$ of the pullback along an essential geometric morphism $\phi$ as a substitute for the push forward with proper support. Indeed, while $\phi_!$ models cohomology with compact support, $\phi_\sharp$ models homology, and the two notions are closely related. 

According to this analogy, we expect the functor $\phi_\sharp$ to satisfy a form of projection formula. While this, to the extent of our knowledge, not automatic, we will eventually show that this projection formula holds for geometric morphisms induced from Nash submersions. For now, we discuss projection formulas in the abstract context of essential geometric morphisms.

Let $\phi \colon \fX \to \fY$ be an essential geometric morphism and consider a pair of maps 
$f\colon \phi_\sharp U \to V$ and $g\colon W\to V$ in $\fY$. 
We have a canonical commutative rectangle (see \cite[\S A.1]{HA}) 
\begin{equation}
\label{def_rect_pf}
\xymatrix{
\phi_\sharp(U\times_{\phi^*V}\phi^*W) \ar[r] \ar[d]& \phi_\sharp\phi^*W \ar^{\con}[r] & W \ar^{g}[d]  \\ 
\phi_\sharp(U) \ar^{f}[rr]  & &  V 
}
\end{equation}
 
\begin{defn}
    Let $\phi\colon \fX \to \fY$ be an essential geometric morphism, and let $\phi_\sharp U \to V$ and $W\to V$ be morphisms in $\fY$. 
    We refer to the map
    \[
    \mdef{\pf_\phi}\colon \phi_\sharp(U\times_{\phi^*V}\phi^*W) \to \phi_\sharp U\times_V W
    \]
    induced from the commutative rectangle (\ref{def_rect_pf}) as the \tdef{relative projection map} of $\phi$. 
We say that $\phi$ is \tdef{locally contractible} if the relative projection map of $\phi$ is a natural isomorphism. \end{defn}

\begin{example} \label{etale_iff_locally_contractible_and_conservative} 
    A geometric morphism
    $\phi\colon\fX\to\fY$ is \etale if and only if it
    is locally contractible and $\phi_{\sharp}$ is conservative, see 
    \cite[Proposition 6.3.5.11]{HTT}.
\end{example}

\begin{example}
\label{Global_sections_locally contractible_iff_essential} Let $\fX$ be
an $\infty$-topos. The terminal geometric morphism $\Gamma\colon\fX\to\Spaces$
is locally contractible if and only if it is essential, i.e., if $\fX$ is locally of constant shape, see \cite[Proposition A.1.9]{HA}. 
\end{example}

The relative projection map is compatible with the composition of geometric morphisms.
\begin{prop} \label{comp_rel_proj}
    Let $\phi \colon \fX\to \fY$ and $\psi \colon \fY \to \fZ$ essential geometric morphisms, and let 
    $f\colon (\psi \phi)_\sharp U \to V$ and $g\colon W\to V$ be morphisms in $\fZ$. The relative projection map 
\[
    \pf_{\psi \phi} \colon (\psi \phi)_\sharp(U\times_{(\psi\phi )^*V}(\psi\phi)^*W) \to (\psi \phi)_\sharp U \times_V W
\]
    of $\psi\phi$ is homotopic to the composition of relative projection maps
\[
    (\psi \phi)_\sharp(U\times_{(\psi\phi )^*V}(\psi\phi)^*W)
    \simeq \psi_\sharp \phi_\sharp(U\times_{\phi^*\psi^* V}\phi^*\psi^*W) \oto{\pf_\phi} 
\]
\[
    \psi_\sharp (\phi_\sharp U\times_{\psi^* V}\psi^*W)
    \oto{\pf_\psi} 
    \psi_\sharp  \phi_\sharp U\times_{V} W \simeq 
    (\psi \phi)_\sharp U\times_{V} W
\]
\end{prop}

\begin{proof}
Consider the diagram 
\[
\xymatrix{
\psi_\sharp\phi_\sharp (U\times_{\phi^*\psi^*V} \phi^*\psi^*W) \ar[r]\ar^{\psi_\sharp\pf_\phi}[d] & 
\psi_\sharp \phi_\sharp \phi^*\psi^*W \ar^-{\con_{\psi\phi}}[r]\ar^{\psi_\sharp\con_\phi}[d]
& 
W \ar@{=}[d]\\
\psi_\sharp (\phi_\sharp(U)\times_{\psi^*V} \psi^*W)\ar[r]\ar[d]
 & 
\psi_\sharp \psi^*W \ar^{\con_\psi}[r]
& W \ar^g[d]\\ 
\psi_\sharp\phi_\sharp(U)\ar^{f}[rr]& 
& 
V 
}
\]
The upper left square commutes by the definition of $\pf_{\phi}$ and the upper right square by the way counits are composed. 
The lower rectangle commutes (see the discussion above \cite[Proposition A.1.9]{HA}). Hence, the entire diagram commutes. 

The commutativity of this diagram formally gives us the result. Indeed, it shows that we can view $\psi_\sharp \pf_\phi$ as a morphism from the square classifying $\pf_{\psi\phi}$ to the square classifying the map $\pf_\phi$, hence providing the required homotopy $\pf_{\psi\phi} \simeq \pf_\psi \circ \psi_\sharp\pf_\phi$. 
\end{proof}

\begin{corl}
\label{loc_contr_composition}
A composition of locally contractible geometric morphisms is locally contractible. 
\end{corl}

As the name suggests, the local contractibility of a geometric morphism is a property that can be checked locally over its source.

\begin{prop}
\label{locality_locally_contractible}
Let $\phi\colon \fX \to \fY \in \RTop$, and let $\fB\subseteq \fX$ be a full subcategory which generates $\fX$ under colimits. 
The geometric morphism $\phi$ is locally contractible if and only if, for every $U\in \fB$, the composition 
\[
\fX_{/U} \to \fX \oto{\phi} \fY
\]
is locally contractible. 
\end{prop}

\begin{proof}
The "if" direction follows from \Cref{loc_contr_composition} in view of the assumption that $\phi$ is locally contractible and the local contractability of \'{e}tale geometric morphisms (\Cref{etale_iff_locally_contractible_and_conservative}).
We shall show the "only if" direction. 

First, note that since $\fB$ generates $\fX$ under small colimits, the assumption that the composition $\fX_{/U}\to \fX \oto{\phi} \fY$ is locally-contractible, and in particular essential, implies that $\phi$ is essential (see \Cref{essential_locally_constant_shape}).
Since both the source and target of $\pf_\phi$ preserve colimits in the $U$-variable, it suffices to show that it is an isomorphism at $U\in \fB$.

Fix $U' \in \fB$ and let $j\colon \fX_{/U'} \to \fX$ be the associated \'{e}tale geometric morphism.
We shall show more generally that the map $\pf_\phi$ is an isomorphism whenever $U$ is in the essential image of $j_\sharp$. Indeed, if $U= j_\sharp U_0$ then by \Cref{comp_rel_proj}, the isomorphism
\[
\pf_{\phi j} \colon (\phi j)_\sharp(U_0 \times_{(\phi j)^*V} (\phi j)^*W) \to
(\phi j)_\sharp(U_0) \times_{V} W,
\]
identifies with the composition 
\[
\phi_\sharp j_\sharp (U_0 \times_{j^*\phi^*V} j^*\phi^*W)
\oto{\pf_j}
\phi_\sharp (j_\sharp U_0 \times_{\phi^*V} \phi^*W) \oto{\pf_{\phi}}
\phi_\sharp j_\sharp U_0 \times_{V} W. 
\]
The map $\pf_j$ is invertible since $j$ is \'{e}tale (see \Cref{etale_iff_locally_contractible_and_conservative}), and hence $\pf_\phi$ is invertible as well. 
\end{proof}

We now give a criterion for local contractibility that makes no explicit reference
to projection formulas (compare \cite[Proposition A.1.11]{HA}).
\begin{prop}
\label{Criterion_locally contractibleness_fuly_faithful}Let $\phi\colon\fX\to\fY$
be an essential geometric morphism. Let $\fB \subseteq \fX$ be a full subcategory which generates $\fX$ under small colimits. The following are equivalent: 
\end{prop}

\begin{enumerate}
\item The geometric morphism $\phi$ is locally contractible. 
\item For every $U\in{\cal \fB}$, the induced functor $\phi_{\sharp}\colon\fX_{/U}\to\fY_{/\phi_{\sharp}U}$
admits a fully faithful right adjoint.
\end{enumerate}
\begin{proof}
Let $U\in\fB$.
The functor $\phi_{\sharp}\colon\fX_{/U}\to\fY_{/\phi_{\sharp}U}$
has a right adjoint given by 
\[
\widetilde{\phi}^{*}V=U\times_{\phi^{*}\phi_{\sharp}U}\phi^{*}V,
\]
for which the counit is the composition
\[
\con\colon\phi_{\sharp}\widetilde{\phi}^{*}V=\phi_{\sharp}(U\times_{\phi^{*}\phi_{\sharp}U}\phi^{*}V)\oto{\pf}\phi_{\sharp}U\times_{\phi_{\sharp}U}V\oto{\sim}V
\]
Hence, $\phi$ satisfies the second condition of the proposition if and only if the map $\pf_\phi$ is an isomorphism when $U\in \fB$ and the map $f\colon \phi_\sharp U \to V$ is an isomorphism. It remains to show that the validity of the projection formula in this special case implies its validity in general. 

Since $\pf_\phi$ preserves all small colimits in the $U$-variable, it suffices to show that $\pf_\phi$ is an isomorphism at $U\in \fB$.
Let $U\in \fB$ and let $\phi_\sharp U \to V$ and $g\colon W\to V$ be morphisms in $\fY$. 
The commutative diagram 
\[
\xymatrix{
\phi_\sharp U\ar@{=}[r]\ar@{=}[d] & \phi_\sharp U\ar^f[d] &\ar[l] \phi_\sharp U \times_V W \ar[d]\\ 
\phi_\sharp U \ar^f[r] & V & W\ar^g[l]
}
\]
induces, by the naturality of the relative projection map, a commutative square
\[
\xymatrix{
\phi_\sharp(U\times_ {\phi^*\phi_\sharp U}\phi^*(\phi_\sharp U \times_V W)) \ar^\wr[d] \ar^{\pf_\phi}[r] & \phi_\sharp U\times_ {\phi_\sharp U}(\phi_\sharp U \times_V W) \ar^\wr[d] \\ 
\phi_\sharp(U\times_ {\phi^*V}\phi^*W) \ar^{\pf_\phi}[r] & \phi_\sharp U\times_ {V}W 
}.
\] 
The upper horizontal map is a relative projection map of a diagram in which the map $f$ is an isomorphism and $U\in \fB$, and hence it is an isomorphism by our assumption on $\phi$. It follows that the lower horizontal projection map is an isomorphism as well, and the result follows.  
\end{proof}

Using the fact that $\cD$-valued sheaves over $\fX$ are just the tensor product $\fX\otimes \cD$, we immediately get consequences for sheaves with general coefficients. 
For an essential geometric morphism $\phi\colon \fX\to \fY$ and for $\cD\in \Calg(\Prl)$, the functor $\phi^*\colon \Shv(\fY,\cD)\to \Shv(\fX,\cD)$ is symmetric monoidal and admits a left adjoint $\phi_\sharp$. We thus obtain a projection map $\phi_\sharp(\sF \otimes \phi^*\sF') \to \phi_\sharp(\sF) \otimes \sF'$. 

\begin{prop}
\label{locally_cont_linear}
    Let $\phi\colon \fX \to \fY$ be a geometric morphism and let $\cD\in \Calg(\Prl)$. If $\phi$ is locally contractible then, for every pair of sheaves $\sF\in \Shv(\fX,\cD)$ and $\sF'\in \Shv(\fY,\cD)$, the projection map 
\[
    \phi_\sharp(\sF \otimes \phi^*\sF')\to \phi_\sharp(\sF)\otimes \sF'
\]
    is an isomorphism.
\end{prop}

\begin{proof}
    The local contractibility of $\phi$ implies that the colax $\fY$-linear functor $\phi_\sharp\colon \fX \to \fY$ is in fact $\fY$-linear, that is, it is a (strict) morphism of $\fY$-modules in $\Prl$. 
    The functor
\[
    \phi_\sharp\colon \Shv(\fX,\cD)\to \Shv(\fY,\cD)
\]
    is obtained from the functor $\phi_\sharp \colon \fX \to \fY$ by tensoring with $\cD$, and hence it is $\Shv(\fY,\cD)$-linear, i.e., satisfies the projection formula. 
\end{proof}

\subsection{Cosheaves}
\label{subsec:cosheaves}
Schwartz sections of a Nash vector bundle on a Nash manifold naturally give a cosheaf, rather than a sheaf, for the restricted topology (see, e.g., \cite[Definition 5.1.3]{SchNash}).
Consequently, we are led to discuss the notion of a cosheaf in our $\infty$-categorical context.
\subsubsection{The \texorpdfstring{$\infty$-}{Infinity }Category of Cosheaves}

The notion of a cosheaf is dual to that of a sheaf. 

\begin{defn} 
Let $\fX$ be an $\infty$-topos and let $\cA\in \Prl$.  
An \tdef{$\cA$-valued cosheaf} on $\fX$
is a colimit preserving functor $\sG\colon \fX \to \cA$. 
\end{defn}

Recall that $\Prl$, with its Lurie tensor product, is closed symmetric monoidal, with internal $\hom$ functor given by 
$
    \mdef{\hom(\cC,\cD)}\simeq \Fun^L(\cC,\cD). 
$ 
Hence, we can express the $\infty$-category of $\cA$-valued cosheaves as 
\[
    \mdef{\CShv(\fX,\cA)}= \hom(\fX,\cA).  
\]
In particular, it is a presentable $\infty$-category. 

\begin{example} \label{ex: cosh_res}
Let $X$ be a restricted topological space. Dualy to the case of sheaves, we can identify $\CShv(X,\cA)$ with the $\infty$-category of functors $\sG\colon \Op(X)\to \cA$ which take the empty set to $\varnothing_\cA$ and pushouts (i.e., unions) in $\Op(X)$ to pushouts in $\cA$.
\end{example}
A geometric morphism
$\phi \colon \fX \to \fY$ gives an adjunction
\[
\phi_{!}\colon\CShv(\fX,\cA)\adj\CShv(\fY,\cA)\colon\phi^{!} 
\]
in which the functor $\phi_!$ is given by $\phi_!\sG(U)\simeq \sG(\phi^*U)$.
In general, the functor $\phi^!$ is more complicated: it is given by right Kan extension along $\phi^*$, followed by cosheafification. For essential geometric morphisms, we have a simpler description of this functor. 
\begin{prop}
\label{formula_push_along_locally contractible}
Let $\phi\colon \fX \to \fY$ be a essential geometric morphism, and let $\cA\in \Prl$. The pullback of cosheaves along $\phi$ is given by 
\[
\phi^!(\sG)(U) \simeq \sG(\phi_\sharp(U)).
\]
\end{prop}

\begin{proof}
The adjunction $\phi_\sharp \colon \fX \adj \fY : \phi^*$ between left adjoint functors, induces an adjunction 
\[
\phi_! \simeq (-)\circ \phi^*\colon \Fun^L(\fX,\cA) \simeq \CShv(\fX,\cA) \adj \CShv(\fX,\cA) \simeq \Fun^L(\fY,\cA): (-)\circ\phi_\sharp, 
\]
and hence $(-)\circ\phi_\sharp$ is a right adjoint to $\phi_!$. 
\end{proof}

\subsubsection{The Tensor Product of Sheaves and Cosheaves}
We have seen that for a presentably symmetric monoidal $\infty$-category $\cD$, the category of sheaves $\Shv(\fX,\cD)$ admits a canonical tensor product operation. For cosheaves, the situation is slightly different. Instead of providing a tensor product for cosheaves, we shall now construct such a tensor product between a sheaf and a cosheaf, making the category of cosheaves linear over sheaves. 

For $\cD\in \Calg(\Prl)$, we can form the $\infty$-category of modules $\Mod_\cD(\Prl)$, consisting of $\cD$-linear presentable $\infty$-categories. Moreover, we have a free-forgetful adjunction 
\[
    (-) \otimes \cD \colon \Prl \adj \Mod_{\cD}(\Prl) : U_\cD 
\]

By definition, an $\cA$-valued cosheaf on an $\infty$-topos $\fX$ is a morphism $\sG\colon \fX \to \cA$ in $\Prl$. If $\cA$ is linear over $\cD\in \Calg(\Prl)$, then by the above free-forgetful adjunction, $\sG$ can be identified with a $\cD$-linear functor 
\[
    \Shv(\fX,\cD)\simeq \fX \otimes \cD \to \cA.
\]
We shall abuse notation and denote this extended functor again by $\sG$. 
The tensor product of $\sG$ with a $\cD$-valued sheaf is defined in terms of this extension.
\begin{defn}
Let $\cA$ be a $\cD$-linear presentable $\infty$-category and let $\fX$ be an $\infty$-topos. For a sheaf $\sF\in \Shv(\fX,\cD)$ and a cosheaf $\sG\in \CShv(\fX,\cA)$, we define their tensor product to be the $\cD$-linear functor
\[
\mdef{\sF \scten \sG} \colon  \Shv(\fX,\cD) \oto{ \sF \otimes (-)}\Shv(\fX,\cD) \oto{\sG} \cA.
\]
\end{defn}
Explicitly, we have 
\[
(\sF\scten \sG)(\sF') \simeq \sG(\sF \otimes \sF').
\]

In fact, the operation $\scten$ is associative, that is, it arises as the binary component of a $\Shv(\fX,\cD)$-linear structure on $\CShv(\fX,\cA)$. Indeed, by the (internal) free-forgetful adjunction we have 

\[
\CShv(\fX,\cA) \simeq \hom(\fX,\cA) \simeq \hom_\cD(\fX\otimes \cD,\cA)= \hom_\cD(\Shv(\fX,\cD),\cA),
\]
and the last term admits a canonical $\Shv(\fX,\cD)$-linear structure through the action of $\Shv(\fX,\cD)$ on the source of the internal $\hom$ object.

By the construction, a geometric morphism $\phi \colon \fX\to \fY$ induces a $\Shv(\fY,\cD)$-linear functor 
\[
    \phi_! \colon \CShv(\fX,\cA)\to \CShv(\fY,\cA),
\]
and hence we have a tautological projection formula  
\[
    \phi_!(\phi^*\sF \scten \sG)\iso \sF \scten \phi_!(\sG). 
\]
However, this projection formula is of little use for us. For the proof of the relative de Rham theorem, we need the following variant for locally contractible geometric morphisms:
\begin{thm}
\label{second_projection_formula}
Let $\phi \colon \fX \to \fY$ be a locally contractible geometric morphism, let $\cD\in \Calg(\Prl)$ and let $\cA\in \Mod_{\cD}(\Prl)$. For a $\cD$-valued sheaf $\sF$ and an $\cA$-valued cosheaf $\sG$ of $\fX$, there is a natural isomorphism 
\[
    \phi_!(\sF \scten \phi^!\sG)\iso \phi_\sharp(\sF)\scten \sG.
\]
\end{thm}

\begin{proof}
By \Cref{formula_push_along_locally contractible}, 
we have an isomorphism of $\cD$-linear functors $\Shv(\fY,\cD)\to \cA$ in the variable $\sH$ of the form 
\[
\phi_!(\sF \scten \phi^!\sG)(\sH) \simeq  (\sF \scten \phi^!(\sG))(\phi^*\sH)
\simeq 
\phi^!\sG(\phi^*(\sH) \otimes \sF) \simeq 
\sG(\phi_\sharp(\phi^*\sH \otimes \sF)).
\]
Similarly, we have 
\[
(\phi_\sharp(\sF) \scten \sG)(\sH) \simeq \sG(\phi_\sharp(\sF)\otimes \sH).
\]
Since $\phi$ is locally contractible we have, by \Cref{locally_cont_linear}, a natural $\Shv(\fY,\cD)$-linear isomorphism
\[
    \pf_\phi \colon \phi_\sharp(\sF \otimes \phi^*\sH) \iso \phi_\sharp(\sF)\otimes \sH,
\] 
which induces the required isomorphism
via pre-composition.
\end{proof}

\subsubsection{Relative Tensor Product}
As for the tensor product of sheaves, we have a relative tensor product for a sheaf of modules and a cosheaf of modules over a sheaf of rings. Before we construct this relative tensor product, we first explain what is a cosheaf of modules.

Let $\fX$ be an $\infty$-topos and $\cD \in \Calg(\Prl)$. Then $\Shv(\fX,\cD)$ is a symmetric monoidal $\infty$-category, so that we can consider commutative rings inside this $\infty$-category, that is, $\cD$-valued sheaves of commutative rings.  For such a ring $\sR\in \Calg(\Shv(\fX,\cD))$, we can then define modules over $\sR$ inside every $\Shv(\fX,\cD)$-linear $\infty$-category.
Applying this construction to the $\Shv(\fX,\cD)$-linear $\infty$-category  $\CShv(\fX,\cA)$, we obtain the $\infty$-category
 of \tdef{cosheaves of modules} over $\sR$, given by 
\[
\mdef{\Mod_\sR(\CShv(\fX,\cA))}\simeq \Mod_{\sR}(\Shv(\fX,\cD))\otimes_{\Shv(\fX,\cD)} \CShv(\fX,\cA).
\]
Roughly, an object of $\Mod_\sR(\CShv(\fX,\cA))$ is a cosheaf $\sG$ together with an action map $\sR \scten \sG \to \sG$, which is coherently associative. 

The presentable $\infty$-category $\Mod_\sR(\CShv(\fX,\cA))$ is linear over $\Shv(\fX,\cD)$, and in particular, for a sheaf of $\sR$-modules $\sF$ and a cosheaf of $\sR$-modules $\sG$, we have a relative tensor product 
\[
\mdef{\sF\scten_\sR \sG} \in \Mod_\sR(\CShv(\fX,\cA)).  
\]

As for the usual tensor product, promoting a cosheaf of modules $\sG$ over $\sR$ to a functor on sheaves of $\sR$-modules, we have 
\[
(\sF \scten_\sR \sG)(\sF')\simeq \sG(\sF \otimes_\sR \sF'). 
\]

\subsubsection{Enriched Hom for Cosheaves}

We now turn to describe the right adjoint of $\scten$, the enriched hom functor.   
\begin{defn}
Let $\fX$ be an $\infty$-topos, let $\cD\in \Calg(\Prl)$ and let $\cA\in \Mod_{\cD}(\Prl)$. For cosheaves $\sG,\sG'\in \CShv(\fX,\cA)$ We denote by 
$\mdef{
\hom(\sG,\sG')
} \in \Shv(\fX,\cD)$
the sheaf characterized by
\[
\Map_{\Shv(\fX,\cD)}(\sF,\hom(\sG,\sG')) \simeq \Map_{\CShv(\fX,\cA)}(\sF \scten \sG,\sG'),
\]
and refer to it as the \tdef{$\cD$-enriched $\hom$ functor}. 
We denote by $\mdef{\Hom(\sG,\sG')}$ the global sections of $\hom(\sG,\sG')$.
\end{defn}
 Note that $\hom(\sG,\sG')$ exists because the tensor product operation $\scten$ preserves colimits in both variables. 
 Similarly, for $\sR \in \Calg(\Shv(\fX,\cD))$ we have an enriched $\hom$ for $\sR$-module cosheaves, denoted $\mdef{\hom_{\sR}(\sG,\sG')}$

In fact, the values of  $\hom(\sG,\sG')$ are easy to describe.
\begin{prop}
\label{internal_hom}
Let $\fX\in \RTop$, let $\cD\in \Calg(\Prl)$ and let $\cA\in \Mod_\cD(\Prl)$. 
For cosheaves $\sG,\sG'\in \CShv(\fX,\cA)$, the enriched $\hom$ functor between them is given by 
\[
\hom(\sG,\sG')(U) \simeq  \Hom(\sG|_U,\sG'|_U) \qin \cD. 
\]
\end{prop}

\begin{proof}
Let $\one \in \cD$ be the unit and let $j\colon \fX_{/U}\to \fX$ be the \etale geometric morphism associated with $U$. By \Cref{formula_push_along_locally contractible} we have, functorialy in $U\in \fX$:
\begin{align*}
\hom(\sG,\sG')(U) &\simeq \Hom(j_\sharp j^*\one_\fX,\hom(\sG,\sG')) \simeq  
\Hom(j_\sharp j^*\one_{\fX} \scten \sG,\sG') \simeq \\ &\simeq\Hom(j_! j^!\sG,\sG') \simeq \Hom(j^!\sG,j^!\sG') = \Hom(\sG|_U,\sG'|_U).
\end{align*}

\end{proof}

\section{Sheaves on Semi-Algebraic Spaces} \label{shv:semialg}
Using the abstract theory of sheaves on $\infty$-topoi, we turn to study sheaves on semialgebraic spaces and Nash manifolds. In \Cref{subsec:sheaves_on_restricted_topological_spaces} we study sheaves on general restricted topological spaces. Then, in \Cref{subsec:comparison of topologies} we prove comparison results for local systems on a restricted topology and the topology it generates. In particular, we show (\Cref{Comparison_Shape_Semi_Algebraic_Space}) that these notions coincide for a semialgebraic space. Finally, in  \Cref{subsec:projection_formula_for_nash_submersion} we specialize to the case of Nash manifolds and show that every Nash submersion of Nash manifolds gives a locally contractible geometric morphism (\Cref{Nash_submersion_locally contractible}). 
\subsection{Sheaves on Restricted Topological Spaces}
\label{subsec:sheaves_on_restricted_topological_spaces}
    
    For a restricted topological space $X$, the poset $\Op(X)$ admits a Grothendieck topology in which the coverings are finite open covers. Consequently, we have an $\infty$-topos 
\[
\mdef{\Shv(X)}:=\Shv(\Op(X)).
\]
    and we can define the $\infty$-category $\mdef{\Shv(X,\cD)}$ as $\cD$-valued sheaves on this $\infty$-topos.
    Note that $\Shv(X)$ is, as usual, a left exact localization of the $\infty$-topos of presheaves $\PShv(X)$. Accordingly, $\Shv(X,\cD)$ is the full subcategory of $\PShv(X,\cD)= \Fun(\Op(X)^\op,\cD)$ spanned by the functors which satisfies descent with respect to finite covers. 

    A continuous map $\phi\colon X\to Y$ induces a geometric morphism $\phi\colon \Shv(X)\to \Shv(Y)$ such that $\phi_*\sF(U) \simeq \sF(\phi^{-1}(U))$. 
    The left adjoint $\phi^*$ takes a sheaf $\sF$ to the sheafification of the presheaf 
    \[
    V\mapsto \colim_{U\supseteq \phi(V)} \sF(U).
    \]
    
On restricted topological spaces, the sheaf condition
can be tested using only covers by pairs of open sets.  \begin{prop}
\label{shv_res_mv_norm}
Let $X$ be a restricted topological space and let $\cD\in\Prl$. A $\cD$-valued presheaf $\sF\in\PShv(X,\cD)$
is a sheaf if and only if it satisfies
\end{prop}

\begin{enumerate}
\item (\emph{Normalization}) $\sF(\es)$ is a final object of $\cD$. 
\item (\emph{Mayer-Vietoris}) For every pair of open sets $U,V\subseteq X$ we have 
\[
    \sF(U\cup V)\iso \sF(U)\underset{\sF(U\cap V)}{\times}\sF(V).
\]
\end{enumerate}
\begin{proof}
    Clearly, the sheaf condition implies the normalization and Mayer-Vietoris properties.  We shall show the converse.
    For an open set $V\subseteq X$ let $\sF_V$ be the sheaf given by $\sF_V(W) = \sF(V\cap W)$. For a finite open cover $\cU = \{U_i\}_{i\in I}$, consider the sheaf
\[
    \mdef{C(\cU,\sF)} := 
    \invlim_{\Delta}\left(\xymatrix{\prod\limits_{i\in I}\sF_{U_i} \ar@<0.5ex>[r] \ar@<-0.5ex>[r] &  \prod\limits_{(i,j)\in I^2} \sF_{U_i\cap U_j}\ar@<1ex>[r] \ar[r] \ar@<-1ex>[r] & \dots}\right),
\]
so that we have a canonical map $\comp\colon \sF \to C(\cU,\sF)$. 
The sheaf condition for $\sF$ with respect to the cover $\cU$ is equivalent to the condition that the map  that the map $\comp\colon \sF(U)\to C(\cU,\sF)(U)$ is an isomorphism. We shall prove, by induction on $|I|$, the stronger statement that $\sF_U \iso C(\cU,\sF)$. If $|I|=0$ this follows from the normalization axiom for $\sF$. For $|I|=1$ this is obvious, and does not depend on our assumptions on $\sF$. Hence, we may assume that $|I|\ge 2$. 

If an open set $W$ is contained in $\cup_{j\in J}U_j$ for some $J\subseteq I$, then, for $\cU' = \{U_j\}_{j\in J}$ we have $C(\cU,\sF)(W)\simeq C(\cU',\sF)(W)$. 
Thus, by induction, we may assume that $\comp$ is an isomorphism at open sets which contained in a union $\bigcup_{j\in J}U_j$ for a proper subset $J\subset I$. For every $W\subseteq U$, we can write $W = W' \cup W''$ where both $W'$ and $W''$ are contained in such a partial union.  

Now, we note that $C(\cU,\sF)$ itself satisfies the Mayer-Vietoris axiom, being a limit of functors that satisfy it. Hence, we can express $\comp$ as a composition of isomorphisms
\[
    \sF(W) \simeq \sF(W') \underset{\sF(W'\cap W'')} {\times} \sF(W'') \simeq C(\cU,\sF)(W') \underset{C(\cU,\sF)(W'\cap W'')}{\times} C(\cU,\sF)(W'')  \simeq  C(\cU,\sF)(W).  
\]
 \end{proof}

This result has immediate implications regarding the $\infty$-categories $\Shv(X,\cD)$. 

\begin{prop}\label{cpt_gen_shv_res}
Let $\cD$ be a compactly generated $\infty$-category. Then 
\begin{enumerate}
\item For every restricted topological space $X$, the full subcategory $\Shv(X,\cD)\subseteq \PShv(X,\cD)$ is closed under filtered colimits. In particular, $\Shv(X)$ is compactly generated, with compact generators the sheaves represented by open sets in $X$. 
\item For every continuous map $\phi\colon X\to Y$ of restricted topological spaces, the functor $\phi_*\colon \Shv(X,\cD)\to \Shv(Y,\cD)$ preserves filtered colimits.
\end{enumerate}
\end{prop}

\begin{proof}
    For $(1)$, since the sheaf conditions in \Cref{shv_res_mv_norm} involve only finite limits, the claim follows from the fact that in a compactly generated $\infty$-categoy filtered colimits commute with finite limits. For the ``in particular" part, it follows from \cite[Lemma 5.5.1.4]{HTT} that the sheafification functor $L\colon \PShv(X)\to \Shv(X)$ carries compact generators to compact generators. Hence, the claim follows from the fact that the Yoneda image of $\Op(X)$ in $\PShv(X)$ consists of compact generators. 

    For $(2)$, we can decompose $\phi_*$ as
\[
    \Shv(X,\cD) \into \PShv(X,\cD) \oto{\phi_*} \PShv(Y,\cD) \oto{L_Y}\Shv(Y,\cD).
\]
The first inclusion preserves filtered colimits by point $(1)$, the second for trivial reasons, and the last since $L_Y$ is a left adjoint. 
\end{proof}

\subsubsection{Closed Emdeddings}

Open embeddings give \'{e}tale geometric morphisms and hence are well understood. 
We turn to study the geometric morphisms induced from another simple type of continuous maps: closed embeddings. 
For closed embeddings, we have the following explicit formula for the pullback. 
\begin{prop}
\label{pull_closed_embedding}
    Let $\cD$ be a compactly generated $\infty$-category and let  $i\colon Z\into X$
   be a closed embedding. For $\sF\in\Shv(X,\cD)$,
    we have 
    \[
        i^{*}\sF(U)\simeq\colim\limits_{V\cap Z \supseteq U}\sF(V),
    \]
\end{prop}

\begin{proof}
Let $\sF'\in \PShv(Z,\cD)$ be the functor determined by the formula 
\[
    \mdef{\sF'(U)} = \colim\limits_{V\cap Z \supseteq U}\sF(V).
\] 
Then $i^*\sF$ is the sheafification of $\sF'$, so it would suffice to check that $\sF'$ is a sheaf. The normalization axiom for $\sF'$ is clear, so we have to prove the Mayer-Vietoris axiom. Let $U_1$ and $U_2$ be open sets in $X$, and let $T_i$ be the poset of open sets in $X$ which contain $U_i$. Then, by a straight-forward cofinality argument, we have 
\[
    \sF'(U_1 \cap U_2) \simeq \colim_{W_1\in T_1,W_2 \in T_2} \sF(W_1 \cap W_2)
\]
and 
\[
    \sF'(U_1 \cup U_2) \simeq \colim_{W_1\in T_1,W_2 \in T_2} \sF(W_1 \cup W_2).
\]

Since $\cD$ is compactly generated, filtered colimits in $\cD$ preserve pullbacks, so that 
\begin{align*}
&\sF'(U_1 \cup U_2) \simeq \colim_{W_1\in T_1,  W_2\in T_2} \sF(W_1\cup W_2) \simeq 
\colim_{W_1\in T_1,  W_2\in T_2} \sF(W_1)\underset{\sF(W_1\cap W_2)}{\times} \sF(W_2) \simeq \\ 
&\simeq \sF'(U_1) \underset{\sF'(U_1 \cap U_2)}{\times} \sF'(U_2). 
\end{align*}
\end{proof}

This formula for the pullback along a closed embedding allows us to show that a closed embedding of restricted topological spaces induces a closed immersion of $\infty$-topoi. Indeed, we have
\begin{prop} \label{open_closed_res_spc_shv}
    Let $X$ be a restricted topological space, and let $U\subseteq X$ be an open set with complement $i\colon Z \into X$. The geometric morphism $i\colon \Shv(Z)\to \Shv(X)$ exhibits $\Shv(Z)$ as a closed complement of the $(-1)$-truncated object  $U \in \Shv(X)$. In other words, we have 
\[
    \Shv(X\setminus U)\simeq \Shv(X)\setminus U.
\]
\end{prop}

\begin{proof}
Let $j\colon U\into X$ denote the embedding. Since $i_*\colon \Shv(Z)\to \Shv(X)$ is fully faithful, it suffices to show that its essential image consist of those sheaves $\sF$ such that $j^*\sF \simeq \term$. Namely, we wish to show that if $j^*\sF\simeq \term$ then the unit map $\sF\to i_*i^*\sF$ is an isomorphism. 

By \Cref{pull_closed_embedding}, the value of this unit map at $V\in \Op(X)$ is the canonical map 
\[
\sF(V) \simeq \colim_{V\cap Z \subseteq W\subseteq V} \sF(V) \to \colim_{V\cap Z \subseteq W\subseteq V} \sF(W). 
\]
We will finish by showing that, for every $W\in \Op(X)$ such that $V\cap Z\subseteq W \subseteq V$, the restriction map $\sF(V)\to \sF(W)$ is an isomorphism. 

The open cover $V = W \cup V\setminus Z$ gives a pullback square 
\[
\xymatrix{
\sF(V) \ar[r]\ar[d] & \sF(V\setminus Z) \ar[d]\\ 
\sF(W)\ar[r] & \sF(W\setminus Z). 
}
\]
Since $\sF$ takes open subsets of $U$ to the terminal object, we deduce that 
\[
\sF(V\setminus Z)\simeq \sF(W\setminus Z) \simeq \term \qin \Spaces
\]
and hence the right vertical map in the above square is an isomorphism. We deduce that the restriction map $\sF(V)\to \sF(W)$ is an isomorphism as well. 
\end{proof}

As a corollary, we obtain 

\begin{corl} \label{op_cl_dec_joint_cons_seqs}
Let $X$ be a restricted topological space and let $j\colon U\into X$ be an open embedding, with closed complement $i\colon Z\into X$. For every compactly generated $\infty$-category $\cD$, the functors $j^*\colon \Shv(X,\cD)\to \Shv(Z,\cD)$ and $i^*\colon \Shv(X,\cD) \to \Shv(U,\cD)$ are jointly conservative. If $\cD$ is also stable, then we have cofiber sequences 
\[
j_!j^* \to \Id \to i_*i^* 
\]
and
\[
i_*i^! \to \Id \to j_*j^* 
\]
in $\End(\Shv(X,\cD))$. 
\end{corl}

\begin{proof}
By \Cref{open_closed_res_spc_shv} for restricted topological spaces and \cite[Corollary 7.3.2.10]{HTT}, we have $\Shv(Z)\simeq \Shv(X)\setminus U$. Hence, by \Cref{shv_rec} we see that $\Shv(X,\cD)$ is a recollement of $\Shv(U,\cD)$ and $\Shv(Z,\cD)$. The result now follows from \Cref{rec_props}.  
\end{proof}

We now derive a version of the Mayer-Vietoris property for closed sets. 
For a restricted topological space $X$, a sheaf $\sF\in \Shv(X,\cD)$ and a subset $S\subseteq X$ we denote by $\sF(S)$ the global sections of the restriction of $\sF$ to $S$. 

\begin{prop}
\label{Mayer_Vietoris_closed_cover} Let $X$ be a (restricted)
topological space and let $\cD$ be a compactly generated $\infty$-category. Let $Z_1,Z_2\subseteq X$ be closed subsets of $X$. For every sheaf $\sF \in \Shv(X,\cD)$,  we have 
\[
\sF(Z_1 \cup Z_2) \iso \sF(Z_1)\underset{\sF(Z_1 \cap Z_2)}{\times} \sF(Z_2)
\]

\end{prop}

\begin{proof}
Replacing $X$ by $Z_1 \cup Z_2$ we may assume that $Z_1$ and $Z_2$ cover $X$. Let $i_1$ and $i_2$ denote the inclusions of $Z_1$ and $Z_2$ into $X$, and let $i$ be the inclusion of their intersection. 
It would suffice to show that 
\[
    \sF \iso (i_1)_* i_1^*\sF \underset{i_*i^*\sF}{\times} (i_2)_*i_2^*\sF
\]
Let $j\colon X\setminus(Z_1 \cap Z_2)\into X$ be the inclusion, so that by \Cref{op_cl_dec_joint_cons_seqs}, the functors
$j^{*}$ and $i^{*}$ are jointly conservative and left exact. After
applying $i^{*}$, the statement reduces to the case where $Z_1 = Z_2 = X$, in which it is clear. 
After applying $j^{*}$, we arrive at a situation in which $Z_{1}$ and $Z_{2}$ are disjoint and hence open.
The statement, in this case, is a consequence of the Mayer-Vietoris
property for pairs of \emph{open} sets.
\end{proof}

\subsubsection{Homotopy Dimension}
To a restricted topological space $X$, we have associated an $\infty$-topos $\Shv(X)$, which allows us to study invariantly sheaves on $X$ valued in $\infty$-categories. In general, we expect properties of $X$ to be reflected by $\Shv(X)$. We now show that this is the case for the dimension; the inductive dimension of $X$ is closely related to the \emph{homotopy dimension} of $\Shv(X)$, in the sense of \cite[\S 7.2]{HTT}. 

Recall that an object $U$ in an $\infty$-topos $\fX$ has homotopy sheaves $\pi_k(U)\in \Shv(\fX_{/U},\Sets)$ (see \cite[Denition 6.5.1.1]{HTT}). For $d\in \NN$, we say that $U$ is \tdef{$d$-connective} if it covers the terminal object and $\pi_i(f) \simeq \term$ for every $i<d$.
\begin{defn} (\cite[Definition 7.2.1.1]{HTT})
    An $\infty$-topos $\fX$ is of \tdef{homotopy dimension $\le d$} if every $d$-connective object $U\in \fX$ admits a global section $\term_\fX \to U$.    
\end{defn}

Under mild assumptions, the inductive dimension of $X$ bounds the homotopy dimension of $\Shv(X)$. 

\begin{defn} \label{def:Proper_Refinements} 
    Let $X$ be a closureful restricted
    topological space, and $\{U_i\}_{i=1}^n$ be an open cover of $X$. A \tdef{proper refinement} of $\{U_i\}_{i=1}^n$ is an open cover $\{V_i\}_{i=1}^n$ of $X$ such that, for every $1\le i \le n$ we have $\overline{V_i} \subseteq U_i$. 
\end{defn}

\begin{prop} \label{hypercompleteness_nice_restricted}
    Let $X$ be a closureful
    restricted topological space of inductive dimension $\le d$. 
    If every open cover of $X$ admits a proper refinement, 
    then $\Shv(X)$ is of homotopy dimension $\le d$. 
\end{prop}

\begin{proof}
    We prove the result by induction on $d$, the case $d=-1$ being trivial. 
    Assume that $d\ge 0$ and let $\sF\in \Shv(X)$ be a $d$-connective object.  
    We wish to show that $\sF$ admits a global section.

    Since $\sF$ is $d$-connective, the terminal map $\sF \to \term$ is an effective epimorphism, and in particular, there is a cover $X=\bigcup_{i=1}^n U_i$ such that $\sF(U_i)\ne \es$. By passing to a proper refinement, we may assume that $\sF(\overline{U_1})\ne \es$. 
    We shall proceed by induction on $n$, starting from $n=1$, for which the claim holds by the
    choice of $U_{1}$. 

    Let $\mdef{K_{1}}:=\overline{U_{1}}$,  $\mdef{K_{2}}:=X\backslash U_{1}$, and 
    $\mdef{K}:=K_{1}\cap K_{2}=\partial U_{1}$. By definition, $K$ is
    closureful of inductive dimension $\le d-1$, and it is clear that
    open covers of $K$ admit proper refinements. By the inductive hypothesis on $d$,
    the $\infty$-topos $\Shv(K)$ is of homotopy dimension $\le d-1$.
    It follows from \cite[Lemma 7.2.1.7]{HTT} that the space $\sF(K)$
    is a \emph{connected} space. By the
    inductive hypothesis on $n$, the space $\sF(K_2)$ in non-empty. 
    Finally, by \Cref{Mayer_Vietoris_closed_cover} we have $\sF(X)\simeq \sF(K_1)\times_{\sF(K)}\sF(K_2)$, so that $\sF(X)$ is non-empty as well. 
\end{proof}

In the special case where $X$ is a semi-algebraic space, we get
\begin{thm} \label{semialg_hom_dim} 
Let $X$
be a semi-algebraic space. The homotopy dimension of $\Shv(X)$ is
bounded by the dimension of $X$. 
\end{thm}

\begin{proof}
Given \Cref{hypercompleteness_nice_restricted}, it suffices to show that open covers in $X$ admit proper refinements.
This claim is local over $X$ and hence it suffices to show this when $X$ is affine. We can then embed $X$ as a locally closed subset of a Nash manifold, and then the result follows from \cite[Proposition A.3.3]{SchNash}.  
\end{proof}

\begin{rmk}
    Though we shall not need it, this bound on the homotopy dimension of $\Shv(X)$ is sharp.
\end{rmk}

Our primary application of this result is the following criterion for a map to induce an isomorphism on sheafifications. 

\begin{prop} \label{base_gen_shv}
    Let $X$ be a semi-algerbaic space and let $B$ be a basis for the restricted topology of $X$. Let $\cD$ be a compactly generated $\infty$-category. If a map $\sF \to \sF'$ in $\PShv(X,\cD)$  induces an isomorphism $\sF(U)\iso \sF'(U)$ for every $U\in B$, then it induces an isomorphism on sheafifications
    \[
    L_X(\sF) \iso L_X(\sF').
    \]
    In particular, the sheaves represented by the elements of $B$ generate $\Shv(X)$ under colimits.
\end{prop}

\begin{proof}
Let $A\in \cD$ be a compact object. Then, the functors $\Map(A,-)\colon \PShv(X,\cD) \to \PShv(X)$ commute with evaluations at open sets and with the sheafification functor. Since they are jointly conservative, we may apply these functors to our map and reduce the case $\cD \simeq \Spaces$. 

Let $f\colon \sF \to \sF'$ be a map satisfying the condition in the corollary. 
Since $\Shv(X)$ is of finite homotopy dimension, isomorphisms in $\Shv(X)$ are detected by the homotopy sheaves (\cite[Corollary 7.2.1.12]{HTT}). Hence, to check that $f\colon L_X(\sF) \iso L_X(\sF')$, it would suffice to show that the map $\pi_0(f) \colon \pi_0(\sF) \to \pi_0 (\sF')$ induces an isomorphism on sheafifications, and that for every open subset $U\subseteq X$ and section $u\in \sF(U)$, the map $\pi_k(f)\colon \pi_k(\sF,u) \to \pi_k(\sF',f(u))$ induces an isomorphism on sheafifications. All these maps are maps of presheaves of \emph{sets} satisfying the assumption of the corollary. Hence, we reduced to the case $\cD = \Sets$, which is clear.  

Finally, in particular we see that the restricted Yoneda functor $\Shv(X)\to \PShv(B)$ is conservative, so that by \cite[Proposition 2.3.2]{yanovski2021monadic} its left adjoint generates $\Shv(X)$ under colimits. This implies that the sheaves represented by members of $B$ generate $\Shv(X)$.  
\end{proof}

\subsection{Comparison of Topologies} \label{subsec:comparison of topologies}

For a restricted topological space $X$, there is a canonical morphism of sites $\real\colon X^{\topify}\to X$, which induces a geometric morphism $\real\colon\Shv(X^{\topify})\to\Shv(X)$. 

\begin{prop}
\label{Comparison_map_fuly_faithful} Let $X$ be a restricted
topological space, and let $\cD\in \Prl$. 
The functor 
\[
\real_* \colon \Shv(X^{\topify},\cD)\to \Shv(X,\cD)
\]
 is fully faithful. 
\end{prop}

\begin{proof}
In the case where $\cD=\Spaces$, this follows directly from the fact that $\Op(X)$ is a basis of $\Op(X^\topify)$ closed under finite intersections, as in the discussion above \cite[Warning 7.1.1.4]{HTT}. The general case follows from this case by tensoring with $\cD$ via \cite[Lemma 5.2.1]{AmbiHeight}.   
\end{proof}
In general, $\real_{*}$ is not an isomorphism, as the following example
shows.
\begin{example}
Let $X$ be a Nash manifold, and let $\mathcal{O}_{X}^{\temp}\in\Shv(X,\RR)$
be the sheaf of tempered functions on $X$ (see \cite[Definition 4.5.1]{SchNash}). Since every
smooth function is locally tempered with respect to the classical topology, we have
$\real^{*}\mathcal{O}_{X}^{\temp}\simeq C_{X^{\topify}}^{\infty}$. On
the other hand, not every section of $\real_{*}C_{X^{\topify}}^{\infty}$ is a tempered function. 
\end{example}
 This example suggests that we may think
of $\real_{*}\real^{*}$, informally, as the functor of ``eliminating
growth conditions'' from a sheaf $\sF$ on $X$.

Let $\phi\colon X\to Y$
be a continuous map of restricted topological spaces and let $\cD\in\Prl$.
We obtain a commutative square 
\begin{equation}
\label{eq:Comparison_square}
\xymatrix{\Shv(Y,\cD)\ar^{\phi^{*}}[r]\ar^{\real^{*}}[d] & \Shv(X,\cD)\ar^{\real^{*}}[d]\\
\Shv(Y^{\topify},\cD)\ar^{\phi^{*}}[r] & \Shv(X^{\topify},\cD)
}
\end{equation}
since the maps in this square have various adjoints,
we obtain several Beck-Chevalley maps.
\begin{prop}\label{Base_change_comparison} 
Let $\cD$ be a compactly generated $\infty$-category. Then
\begin{itemize}
    \item For an open embedding $j\colon U\into X$ of restricted topological spaces, we have $j^*\real_*\iso \real_*j^*$. 
    \item For a closed embedding $i\colon Z\into X$, we have $\real^*i_* \iso i_*\real^*$. 
\end{itemize}
\end{prop}

\begin{proof}
For the first item, we have $j^*\real_*\sF(V)\simeq \real_*\sF(V) \simeq \real_*j^*\sF(V)$ for every $V\subseteq U$. 
The second item follows from \Cref{top_rec} and \cite[Proposition 7.3.2.12]{HTT} for the case $\cD\simeq \Spaces$, and the general case follows from it by tensoring with $\cD$. 
\end{proof}

In general, the Beck-Chevalley map $\real^{*}j_{*}\to j_{*}\real^{*}$
is not an isomorphism.
\begin{example} For $x\in \RR$, let $\delta_x$ be the skyscraper real-valued sheaf at $x$.  
Consider the open subset $j\colon \RR-\{0\}\into \RR$. 
Set ${\displaystyle \sF=\bigoplus_{n=1}^{\infty}\delta_{\frac{1}{n}}}$.
Then 
\[
{\displaystyle j_{*}\real^{*}\sF(V)=\prod_{\frac{1}{n}\in V}\RR}
\]
while $\real^{*}j_{*}\sF(V)$ is its subspace ${\displaystyle \bigoplus_{\frac{1}{n}\in V}}\RR$
for every $V$ such that $0\in V$. 
\end{example}

Similarly, the Beck-Chevalley map $ i^{*}\real_{*}\to \real_{*}i^{*}$
might not be an isomorphism.
\begin{example}
Let $X=\RR^{2}$ with its canonical semialgebraic structure, and
let 
\[
\sF=i_{*}\RR_{Z}\qin\Shv(X,\RR)
\] 
for $i\colon Z\into(\RR^{2})^{\topify}$
the inclusion of $Z:=\left\{ (t,e^{-|t|}):t\in\RR\right\} $. Then $\real_{*}i^{*}\sF\simeq0$
while the section $1\in\sF(\RR^{2})$ does not restrict to $0$
in $i^{*}\real_{*}\sF(\RR)$. Indeed, if it would there would
be a cover of a semi-algebraic neighborhood of $\RR$ by open sets
on which the restriction of this section vanishes, and such cover
does not exist. 
\end{example}
We note that in this example, the problem is the unboundedness of $Z$. In particular, if $Z$ is a point, this deficiency disappears. 

\begin{prop}
\label{realization_stalks}
Let $X$ be a restricted topological space and let $x\in X$. For every $\cD\in \Prl$ and every sheaf $\sF \in \Shv(X^\topify,\cD)$, 
the canonical (Beck-Chevalley) map 
$\sF_x \to (\real_*\sF)_x$ is an isomorphism. 
In other words, the functor $\real_*$ does not change the stalks of a sheaf. 
\end{prop}

\begin{proof}
We have $\sF_x \simeq \colim_{x\in V} \sF(V)$ where the colimit is over open sets in $X^\topify$. 
Similarly, we have $(\real_*\sF)_x \simeq \colim_{x\in U} \sF(U^\topify)$, where the colimit is over open sets $U\subseteq X$. Via these identifications, the map $\sF_x \to (\real_*\sF)_x$ is induced from the morphism of posets
\[
\{U\in \Op(X): x\in U\} \to \{V\in \Op(X^\topify): x\in V\}.
\]
Since $\Op(X)$ is a basis to $X^\topify$, this map is cofinal, and the result follows.
\end{proof}

\subsubsection{Comparison for Local Systems}
In general, sheaves on $X$ and $X^\topify$ are very different objects, as sheaves on $X$ might contain information about ``growth conditions'' for sections. However, for local systems, the situation is different, and the comparison of topologies is better behaved.

For a (restricted) topological space $X$, we denote $|X|:=|\Shv(X)|$,
and say that $X$ is of constant shape if $|X|$ is pro-constant. For every restricted topological space $X$, the geometric morphism \[
\real\colon \Shv(X^\topify)\to \Shv(X)
\] induces a map $|X^\topify|\to |X|$.
\begin{defn}
\label{def-of_topological_shape}

Let $X$ be a restricted topological space. We say that $X$ is of \tdef{topological shape} if $|X|$ and $|X^\topify|$ are pro-constant, and the map $\real \colon \Shv(X^\topify)\to \Shv(X)$ induces an isomorphism $|X|\simeq |X^\topify|\in \Spaces$. 
We say that $X$ is \tdef{locally of topological shape} if every open subset of $X$ is of topological shape. 
\end{defn}

Note that, in particular, if $X$ is (locally) of topological shape, then it is (locally) of constant shape. For a restricted topological space of topological shape, local systems on $X$ and $X^\topify$ are comparable. 

\begin{prop}
\label{top_shape_loc_sys}
Let $X$ be a restricted topological space. If $X$ is locally of topological shape, then $X$ and $X^\topify$ are essential and the adjunction \[
\real_* \colon \Shv(X^\topify,\cD)\adj \Shv(X,\cD): \real^* 
\]
restricts to an equivalence 
\[
\Loc(X^\topify,\cD) \simeq \Loc(X,\cD).
\]
\end{prop}

\begin{proof}
Since $X$ is locally of topological shape, it is in particular essential by \Cref{essential_locally_constant_shape}. Since the open sets of $X$ form a basis of $X^\topify$ closed under finite intersections, they generate $\Shv(X^\topify)$ under small colimits. Since, by our assumption that $X$ is locally of topological shape, the $\infty$-topos $\Shv(X^\topify)_{/U^\topify}$ is of constant shape for every $U\in \Op(X)$, we deduce that $\Shv(X^\topify)$ is essential as well. In particular, by \Cref{Local_systems_sheaves_on_shape} we deduce that $\Loc(X,\cD)\simeq \Loc(X)\otimes \cD$ and similarly for $X^\topify$, and hence the general case follows from the case $\cD\simeq \Spaces$ by tensoring with $\cD$. 

The commutative diagram of $\infty$-topoi 
\[
\xymatrix{
\Shv(X^\topify) \ar^\psi[d] \ar^\real[r] & \Shv(X) \ar^\psi[d] \\
\Shv(|X|) \ar^\sim[r]      & \Shv(|X^\topify|)
}
\]
shows that we can express $\real^*\colon \Loc(X)\to \Loc(X^\topify)$ as a composition of equivalences
\[
\Loc(X) \simeq \Shv(|X|)\simeq \Shv(|X^\topify|)\simeq \Loc(X^\topify).
\]
To show that its inverse is given by $\real_*$, it suffices to show that $\real_*$ takes local systems to local systems. 

Let $\sF \in \Loc(X^\topify)$. We wish to show that $\real_*\sF$ is a local system. Since $\Loc(X)\simeq \Loc(X^\topify)$, 
we have $\sF\simeq \real^*\sF'$ for a local system $\sF'$ on $X$. Restricting to a member of a trivializing cover of $\sF'$, we may assume without loss of generality that $\sF$ is a constant sheaf with value $A\in \Spaces$. In this case, since $X$ is locally of topological shape, we have
\[
    \real_*\sF(U)\simeq \sF(U^\topify)\simeq \Map(|U^\topify|,A)\simeq \Map(|U|,A)\simeq (\Gamma^*A)(U)
\]
which implies that $\real_*\sF\simeq \Gamma^*A$. In particular, $\real_*\sF$ is a constant sheaf, and hence a local system.
\end{proof}

For semi-algebraic spaces, we have
\begin{thm}
\label{Comparison_Shape_Semi_Algebraic_Space} Let $X$ be a semi-algebraic
space. Then $X$ is locally of topological shape. In particular, $\Shv(X)$ is essential, we have $|X| \simeq |X^\topify|$ and $\real \colon \Shv(X^\topify)\to \Shv(X)$ induces an equivalence
$
\Loc(X,\cD)\simeq\Loc(X^{\topify},\cD)
$ 
for every $\cD\in\Prl$. 
\end{thm}

\begin{proof}
Let $B$ be the collection of contractible (in the topological sense) semialgebraic subsets of $X$. Then, $B$ is a basis of $X$ (e.g., every open star of a simplex in a triangulation on an open subset is contractible).  
For every $U\in B$, we have $|U^\topify|\simeq \term$, and in particular, it is pro-constant. Since the sheaves on $X^\topify$ represented members of $B$ generates $\Shv(X)$, their images under $\real^*$ generates $\Shv(X^\topify)$. We deduce from Propositions \ref{essential_locally_constant_shape} and \ref{base_gen_shv} that $X^\topify$ is essential. 
It remains to show that $\real$ induces an isomorphism $|U^\topify|\simeq |U|$ for every $U\in \Op(X)$. Replacing $X$ with $U$ if necessary, it would suffice to show this for $X$ itself. To show that $|X^\topify| \iso |X|$, we will prove more generally that for every constant sheaf $\sF\simeq \Gamma^*A$ on $X$, the unit $\un \colon \sF \to \real_*\real^*\sF$ is an iomorphism. 

Let us abuse notation and denote the constant \emph{presheaf} on $X$ with value $A$ again by $A$. We have a map $\alpha \colon A\to \real_*\real^*\sF$ which gives the map $\un$ after sheafifying $A$. Hence, by \Cref{base_gen_shv}, it would suffice to show that $\alpha$ induces an isomorphism 
$\alpha \colon A \iso \real_*\real^*\sF(U)$ for every $U\in B$. In fact, an inverse to $\alpha$ at $U$ is given by the composiiton of isomorphisms 
\[
\real_*\real^*\sF(U)\simeq (\Gamma^*A)(U^\topify) \simeq \Map_\Spaces(|U^\topify|,A)\simeq A.
\]

\end{proof}

Informally, this result means that, as long as locally constant sheaves are concerned, there is no difference between the classical and the semialgebraic topology. In other words, locally constant sheaves have no "growth conditions".

\subsection{Local Contrcatibility of Nash Submersions}
\label{subsec:projection_formula_for_nash_submersion}
For the proof of the relative de Rham theorem, we need a projection formula for the push-forward of cosheaves along a Nash submersion. A prototype for such projection formulas, in the context of locally contractible geometric morphisms, is established in \Cref{second_projection_formula}. Our aim in this section is to show that this result is applicable for Nash submersion by showing that Nash submersions induce locally contractible geometric morphisms on the $\infty$-topoi of sheaves (\Cref{Nash_submersion_locally contractible}). 
\subsubsection{Homotopy Invariance of the Shape}
Morally, Nash submersions give locally contractible geometric morphisms because they are locally contractible in a semialgebraic sense. This property, for Nash submersions of relative dimension 1, is the content of \Cref{fam_int_base}. To deduce the local contractibility of the geometric morphism associated with a Nash submersion from this "semialgebraic local contractability", we need to know that the topos-theoretic local contractability is, in the appropriate sense, invariant under semialgebraic homotopies. This is achieved by studying sheaves on the semialgebraic interval $I= [0,1]$. 

\begin{prop}
\label{Beck-Chevalley_for_interval} Let $I=[0,1]$ and let $i\colon X\into Y$ be a closed embedding
of semi-algebraic spaces. Let $p\colon X\times I\to X$ and $q\colon Y\times I\to Y$
denote the projections. The Beck-Chevalley map 
\[
\BC\colon i^{*}q_{*}\to p_{*}(i\times\Id_{I})^{*}
\qin \Fun(\Shv(Y\times I),\Shv(X))\]
induced from the square 
\[
\xymatrix{X\times I\ar^{i\times\Id_{I}}[r]\ar^{p}[d] & Y\times I\ar^{q}[d]\\
X\ar^{i}[r] & Y
}
\]
is a natural isomorphism.
\end{prop}

\begin{proof}
Let $\sF\in\Shv(Y\times I)$ and let $U\in X$. By \Cref{pull_closed_embedding} applied to $i$ and $i\times \Id_I$, the Beck-Chevalley
map at $\sF$ and $U$ is given by the comparison map of filtered colimits
\[
\colim\limits_{V\supseteq i(U)}\sF(V\times I) \to \colim\limits_{W\supseteq i(U)\times I}\sF(W)
\]
The result now follows by a direct cofinality argument. 
\end{proof}

From this, we deduce:
\begin{prop}
\label{triviality_of_unit_projection_from_interval} Let $X$
be a semi-algebraic space and let $I=[0,1]$ be the closed semi-algerbaic
interval. Let $p\colon X\times I\to X$ be the projection. The functor
$p^{*}\colon\Shv(X)\to\Shv(X\times I)$ is fully faithful. Equivalently,
the unit $\un_p\colon\Id\to p_{*}p^{*}$ is a natural isomorphism. 
\end{prop}

\begin{proof}
    We prove the result by induction on the dimension of $X$. If $X\simeq\es$, there is nothing to prove. By \Cref{cpt_gen_shv_res}, both the source and
    the target of $\un$ preserve filtered colimits.
    Thus, it suffices to prove that the unit $\un\colon\sF\to p_{*}p^{*}\sF$ is an isomorphism for $\sF$ which is a finite colimit of sheaves representable by open sets in $X$. 
    Passing to the complement of the boundaries of all the open sets involved, we may find an open dense subset $U\subseteq X$ such that $\sF|_U$ is locally constant. 
    
    Let $Z$ be the closed complement of $U$, and let $i\colon Z\into X$ and $j\colon U\into X$ denote the embeddings. Then, by \Cref{op_cl_dec_joint_cons_seqs}, the functors $i^{*}$ and $j^{*}$ are jointly conservative, and so it suffices to prove that 
\[
    i^{*}\un_{p}\colon i^{*}\sF\to i^{*}p_{*}p^{*}\sF
\] 
    and 
\[
    j^*\un_p \colon j^*\sF \to j^*p_*p^*\sF
\] 
    are isomorphisms. 

    By \Cref{Beck-Chevalley_for_interval} and its (trivial) analog for open embeddings, we reduce to the cases where $X=U$ and $X=Z$, i.e., when $\sF$ is locally constant, or $X$ is of smaller inductive dimension. The latter follows by induction on the dimension, so it suffices to prove the first case. 

    Shrinking $U$ if necessary, we may further assume that $\sF\simeq \Gamma^*A$ for some $A\in \Spaces$. Then, using \Cref{Comparison_Shape_Semi_Algebraic_Space} we can write the unit $\sF(V) \to \real_*\real^*\sF(V)$ for $V\in \Op(X)$ as a composition of isomorphisms 
\[
\sF(U)\simeq \Map(|U|,A)\simeq \Map(|U^\topify|,A)\simeq \Map(|(U\times I)^\topify|,A)\simeq \Map(|U\times I|,A)\simeq \real_*\real^*\sF(U). 
\]
\end{proof}
This result allows us to show that the relative shape of a morphism is a
relative homotopy invariant in the following sense. 
\begin{defn}
Let $X_0,X_1$ and $Y$ be semi-algebraic
spaces. Let $\phi\colon X_{0}\to Y$ and $\psi\colon X_{1}\to Y$
be maps in $\Semialg$ and let $\eta,\kappa\colon X_{0}\to X_{1}$
be morphisms in $\Semialg_{/Y}$. We say that
$\eta$ and $\kappa$ are \tdef{homotopic relative to $Y$} if there
is map $G\colon X_{0}\times I\to X_{1}$ in $\Semialg_{/Y}$ such that
with $G|_{X_{0}\times\left\{ 0\right\} }=\eta$ and $G|_{X_{0}\times\left\{ 1\right\} }=\kappa$.
\end{defn}

As in classical homotopy theory, this allows us to define homotopy equivalence relative to $Y$. We say that $\phi$ and $\psi$ are \tdef{homotopy equivalent relative to $Y$} if there are maps $\phi\to\psi$ and $\psi\to\phi$ in $\Semialg_{/Y}$
which are inverse to one another up to relative homotopy. 

\begin{prop}
Let $X_0,X_1$ and $Y$ be semi-algebraic spaces. Let $\phi\colon X_{0}\to Y$
and $\psi\colon X_{1}\to Y$ be morphisms in $\Semialg$ and
let $\eta,\kappa\colon X_{0}\to X_{1}$ be 
maps in $\Semialg_{/Y}$. If $\eta$ and $\kappa$ are homotopic relative
to $Y$, then the induced natural transformations
\[
    \un_{\eta}\colon\psi_{*}\psi^{*}\to\phi_{*}\phi^{*}
\] 
and 
\[
    \un_{\kappa}\colon\psi_{*}\psi^{*}\to\phi_{*}\phi^{*}
\]
are homotopic in the space $\Map_{\End(\Shv(Y))}(\psi_{*}\psi^{*},\phi_{*}\phi^{*})$.
In particular, $\eta$ and $\kappa$ induce homotopic maps $|\phi|\to|\psi|$. 
\end{prop}

\begin{proof}
Let $i_0,i_1 \colon X_0 \to X_0 \times I$ be the maps induced from the inclusions of the endpoints of $I$, and   
let $G\colon X_{0}\times I\to X_{1}$ be a semi-algebraic homotopy
between $\eta$ and $\kappa$ relative to $Y$. We have a commutative diagram 
\[
\xymatrix{
X_{0} \ar@/^2.0pc/^{\kappa}[rr] \ar@/^-2.0pc/^{\eta}[rr]  \ar@<-.5ex>_{i_1}[r] \ar@<.5ex>^{i_0}[r] & X_0 \times I \ar^G[r] & X_1 
}
\]
in $\Semialg_{/Y}$, and hence we can reduce to the case where $X_1 = X_0 \times I$, with structure map given by the composition $\widetilde{\phi}\colon X_0 \times I \oto{p} X_0 \oto{\phi} Y$, and that $\eta$ and $\kappa$ are the maps $i_0$ and $i_1$. 
In this case, we wish to show that the unit maps 
\[
\widetilde{\phi}_* \widetilde{\phi}^* \to \widetilde{\phi}_* (i_0)_* (i_0)^* \widetilde{\phi}^*\simeq  \phi_*\phi^*
\] 
and 
\[
\widetilde{\phi}_* \widetilde{\phi}^* \to \widetilde{\phi}_* (i_1)_* (i_1)^* \widetilde{\phi}^* \simeq  \phi_*\phi^*\] 
are homotopic. We finish by observing that both these maps are one-sided inverses to the isomorphism given by the unit map 
$\phi_*\phi^* \oto{\sim} \phi_* p_* p^* \phi^* \simeq \widetilde{\phi}_* \widetilde{\phi}^*$ (see \Cref{triviality_of_unit_projection_from_interval}).
\end{proof}

\begin{corl}[Homotopy Invariance] \label{hom_inv_shape} If $\phi\colon X_{0}\to Y$
and $\psi\colon X_{1}\to Y$ are relatively homotopy equivalent, then
\[
\phi_{*}\phi^{*}\simeq\psi_{*}\psi^{*}\qin\End(\Shv(Y))
\]
In particular, $|\phi|\simeq|\psi|$. 
\end{corl}

The homotopy invariance of the relative shape gives the desired "homotopical triviality" of pointed families of intervals.
\begin{prop}
\label{fam_int_triv_shape} 
Let $\phi\colon X\to Y$
be a pointed family of intervals over the Nash manifold $Y$. Then $\phi^{*}\colon \Shv(Y)\to \Shv(X)$ is fully faithful. In particular, we have $|\phi|\simeq\term_{Y}$. 
\end{prop}

\begin{proof}
    By \cite[Lemma 3.3.1]{AmbiHeight}, it suffices to show that $\phi_*\phi^*$ is homotopic to the identity functor of $\Shv(Y)$.   
    Let $s$ be a section of $\phi$. Then, $\phi s = \Id_Y$ and $s\phi$ is homotopic to the identity of $X$ relative to $Y$ via a linear homotopy in each fiber. Hence, by \Cref{hom_inv_shape} $s$ provides an equivalence 
\[
    \phi_*\phi^*\simeq (\Id_Y)_*(\Id_Y)^* \simeq \Id_{\Shv(Y)}.
\] 
\end{proof}

\subsubsection{The Proof of the Local Contractibility}
Using the triviality of families of intervals above, we are ready to show that Nash submersions induce locally contractible geometric morphisms.
\begin{thm}
\label{Nash_submersion_locally contractible}
Let $\phi\colon X\to Y$ be a Nash submersion. Then $\phi\colon\Shv(X)\to\Shv(Y)$
is a locally contractible geometric morphism. 
\end{thm}

\begin{proof}
    Let $\phi \colon X\to Y$ be a Nash submersion. By \Cref{Implicit_Function_Theorem}, the collection of open sets $U\subseteq X$ for which $\phi|_U$ can be written as a composition 
\[
    U\into Y\times \RR^k \onto Y
\] 
    form a basis for the restricted topology on $X$. Hence, by \Cref{locality_locally_contractible} and \Cref{base_gen_shv}, we may assume without loss of generality that $\phi$ is such a composition.  

    Next, since locally contractible geometric morphisms are closed under compositions (\Cref{loc_contr_composition}) and since \'{e}tale geometric morphisms are locally contractible (\Cref{etale_iff_locally_contractible_and_conservative}), we may further reduce further to the case where $\phi$ equals to the projection $\pi\colon Y\times \RR \to Y$. 

    We first show that $\pi$ is essential. By \Cref{fam_int_base} the open sets $U\subseteq Y\times \RR$ for which $\pi|_U \colon U\to \phi(U)$ is a pointed family of intervals form a basis for the restricted topology of $Y\times \RR$, and hence, using \Cref{essential_locally_constant_shape}, we see that it is enough to show that a pointed family of intervals has constant shape. By \Cref{fam_int_triv_shape} the shape of a pointed family of intervals is trivial and, in particular, constant.
    
    Finally, to show that $\pi$ is locally contractible, we use the criterion from \Cref{Criterion_locally contractibleness_fuly_faithful} for the basis $B$ of $Y\times \RR$ consisting of the open sets $U\subseteq Y\times \RR$ which are pointed families of intervals over their images in $Y$. Let us from now on abuse notation and denote the sheaf represented by an open set $U$ again by $U$. 
    By \Cref{fam_int_triv_shape} we have, for every $U\in B$, 
\[
    \pi_\sharp(U)\simeq \pi(U) \qin \Shv(Y),
\]
    and via this identification, the functor 
\[
    \pi_\sharp \colon \Shv(U)\simeq \Shv(Y\times \RR)_{/U} \to \Shv(Y)_{/\pi_\sharp(U)}\simeq \Shv(\pi(U)) 
\]
    is left adjoint to the pullback along the restricted map $\pi_U\colon U\to \pi(U)$. This pullback functor is fully faithful by \Cref{fam_int_triv_shape}. 
\end{proof}

Since every Nash submersion induces a locally contractible, and in particular, essential geometric morphism on the $\infty$-categories of sheaves, we obtain for every Nash submersion $\phi\colon X\to Y$ a functor $\phi_\sharp\colon \Shv(X,\cD)\to \Shv(Y,\cD)$, which we think of as a twist of the push-forward with proper support. In particular, we expect it to satisfy a form of proper base-change. Indeed, we have
\begin{thm}
\label{Proper base-change}
    Let 
\[
    \xymatrix{
    X\ar^{\psi'}[r]\ar^{\phi'}[d] & Y\ar^{\phi}[d] \\ 
    Z \ar^{\psi}[r] & W
    }
\]
    be a pullback square of Nash manifolds, where $\phi$ is a Nash submersion and let $\cD\in \Prl$. The Beck-Chevalley map 
\[
    \BC \colon \phi'_\sharp \psi'^* \to \psi^*\phi_\sharp \qin \Fun(\Shv(Y,\cD),\Shv(Z,\cD)).
\]
    is a natural isomorphism.
\end{thm}

\begin{proof}
    First, note that the map $\BC$ for $\cD$-valued sheaves is obtained from the map $\BC$ for $\Spaces$-valued sheaves by tensoring with $\cD$. Hence, it suffices to show the result in the case $\cD= \Spaces$. 

    Next, we note that the collection of Nash submersions for which the statement holds is closed under composition and restriction to a basis of $Y$. Moreover, the result holds if $\phi$ is an open embedding. Hence, as in the proof of \Cref{Nash_submersion_locally contractible}, it suffices to prove the claim for $Y=W\times \RR$ and $\phi$ the projection onto $W$. 

    Both the source and target of $\BC$ preserve all small colimits, and hence, by \Cref{base_gen_shv}, it suffices to check that $\BC$ is an isomorphism at $U\in \Shv(Y)$ which is a pointed family of intervals over its image in $W$ (here we keep identifying open sets with the sheaves they represent). In this case, by \Cref{fam_int_triv_shape}, we have 
\[
    \psi^*\phi_\sharp(U) \simeq \psi^*\phi(U)\simeq \psi^{-1}(\phi(U)).
\]
    Similarly, we have 
\[
    \phi'_\sharp \psi'^*(U) \simeq \phi'_\sharp\psi'^{-1}(U) \simeq \phi'(\psi'^{-1}(U)).
\]
    Via these identifications, the map $\BC$ corresponds to the equality of open sets 
\[
    \phi'(\psi'^{-1}(U)) = \psi^{-1}(\phi(U)) \qin \Op(Z)
\]
    and hence it is an isomorphism.
\end{proof}

\section{Relative de Rham Theorem}
\label{sec:rel_der}

In the previous section, we discussed the general theory of sheaves and cosheaves on restricted topological spaces. 
In this section, we describe the $\infty$-categorical version of cosheaves of Schwartz sections and prove the main result of this paper: the relative de Rham theorem (\Cref{global_relative_de_rham}).




\subsection{Cosheaves of Schwartz Sections }
\label{subsec:the_cosheaf_of_schwartz_sections}
For a Nash vector bundle $\sE$ on a Nash manifold $X$, the functor that takes an open set $U\subseteq X$ to the Fr\'{e}chet topological vector space of Schwartz sections $ \Sc(U,\sE)$ forms a cosheaf for the restricted topology on $X$. These cosheaves admit a canonical structure of modules over the sheaf of tempered functions on $X$ and consequently have partition unity. In particular, they are ''acyclic" sheaves. 
In this section we formalize this acyclicity by regarding the cosheaf of Schwartz sections, $\Sc_\sE$, as a functor valued in an $\infty$-category of Fr\'{e}chet complexes. Then, the acyclicity of $\Sc_\sE$ as a functor valued in Fr\'{e}chet \emph{spaces},  translates to the usual cosheaf property of $\Sc_\sE$ as a functor valued in Fr\'{e}chet \emph{complexes}.

\subsubsection{Nash Vector Bundles \& Nash connections}
For a Nash manifold $X$, let $\mdef{\sO_X}$ denote the sheaf of Nash functions on $X$.  
\begin{defn}
    Let $X$ be a Nash manifold. A \tdef{Nash vector bundle} on $X$ is a locally free sheaf of modules over $\mathcal{O}_{X}$.
\end{defn}

\begin{example}
    The tangent bundle $T_{X}$, the cotangent bundle $\Omega_{X}$ and
    their exterior powers are Nash vector bundles (see, e.g., \cite[Theorem 3.4.3]{SchNash}). 
\end{example}

As for smooth bundles, we have the notion of a flat connection on a Nash vector bundle.

\begin{defn}
    Let $\sE$ be a vector bundle on
    a Nash manifold $X$. A \tdef{Nash flat connection} on $\sE$ is
    a map of sheaves $\nabla:\sE\to \sE\otimes\Omega_{X}$ satisfying
    the Leibnitz rule and the flatness condition $\nabla^2 = 0$ for smooth sections.
\end{defn}

For a flat connection $(\sE,\nabla)$ on a Nash manifold $X$,
we denote by 
\[
\mdef{\flat(\sE,\nabla)}:=\ker(\nabla) \qin \Shv(X,\RR)
\]
 the sheaf of \tdef{flat sections} of $\sE$. If $\nabla$ is clear
from the context we denote it simply by $\flat(\sE)$. 

Let $\Vect_X^\nabla$ be the category of Nash vector bundles on $X$, and $\Loc(X,\RR)^\fin$ the category of finite-dimensional local systems of real vector spaces on $X$. We have a pair of adjoint functors 
\[
(-)\otimes_\RR \sO_X \colon \Shv(X,\RR) \adj \Vect_X^\nabla : \flat  
\]
However, unlike the situation for smooth manifolds, in our context, this construction \emph{does not} give an equivalence of categories. Instead, we have only the following

\begin{prop}
\label{Riemann-Hilbert_semi_alg}
Let $X$ be a Nash manifold.
The functor 
\[
(-)\otimes_\RR \sO_X \colon \Loc^\fin(X,\RR) \to \Vect^{\nabla}_X
\] 
is a fully faithful embedding, with essential image those vector bundles with a flat connection, locally generated over $\sO_X$ by their flat sections.   
\end{prop}
Note that here $\sL\otimes \sO_X$ is endowed with the connection $1\otimes d$.
\begin{proof}
For $\sL\in \Loc^\fin(X,\RR)$, the unit map $\sL \to \flat(\sL \otimes \sO_X)$ is an isomorphism. Indeed, this can be checked locally over $X$, so we can reduce to the case of a trivial local system, which is clear. 

On the other hand, is $\sE \in \Vect_X^\nabla$ is locally generated over $\sO_X$ by flat sections, then the counit map 
$\flat(\sE,\nabla) \otimes \sO_X \to \sE$ is an isomorphism locally on $X$, and hence an isomorphism of sheaves. 
\end{proof}

\subsubsection{The \texorpdfstring{$\infty$-}{Infinity }-Category of Fr\'{e}chet Complexes}
We shall now introduce our presentable $\RR$-linear $\infty$-category of complexes of Fr\'{e}chet spaces. Our choice is relatively arbitrary and should not be considered an exceptionally efficient $\infty$-category for performing derived functional analysis.
However, since the analysis in this paper is relatively soft, it will suffice for our needs.

Recall that, for an additive 1-category $\cA$, there is an $\infty$-category $\Ch(\cA)$ of chain complexes over $\cA$, with a canonical full subcategory $\Ch_b(\cA)\subseteq \Ch(\cA)$ spanned by the bounded complexes. The $\infty$-category $\Ch_b(\cA)$ is a stable $\infty$-category, and moreover, it is a universal stable $\infty$-category which receives an additive functor from $\cA$. Namely, the embedding $\cA \into \Ch_b(\cA)$ of the complexes concentrated in degree $0$, induces an equivalence 
\[
    \Fun^{\rex}(\Ch_b(\cA),\cD) \simeq \Fun^{\sqcup}(\cA,\cD)
\]
for every stable $\infty$-category $\cD$. Here, on the left-hand side, we have the $\infty$-category of right-exact functors, and on the right-hand side, the $\infty$-category of additive functors. 
(to see this, apply \cite[Corollary 7.59]{bunke2019controlled} to the exact structure consisting of only split sequences). 

Let $\kappa$ be a large enough cardinal and let $\Fre^{\kappa}$ be the category of Fr\'{e}chet topological vector spaces of dimension $\le \kappa$. 

\begin{defn}
We set
\[
    \mdef{\infFre}:=\Ind(\Ch_{b}(\Fre^{\kappa})).
\]
\end{defn}
 
Let $\mdef{\Der(\RR)}$ be the derived $\infty$-category of $\RR$-vector spaces (in the sense of \cite[\S 1.3.2]{HA}), and $\mdef{\Der(\RR)^\fin}\subseteq \Der(\RR)$ be the full subcategory spanned by the bounded complexes of finite dimensional vector spaces, so that $\Der(\RR)\simeq \Ind(\Der(\RR)^\fin)$. 
Clearly $\Ch_b(\Fre^\kappa)$ is tensored over $\Der(\RR)^\fin$, and hence $\infFre$ is a compactly generated $\infty$-category, linear over $\Der(\RR)$.
Moreover, we have a fully faithful embedding of stable $\infty$-categories $\Ch_{b}(\Fre^{\kappa})\into\infFre$. These are essentially the only properties we shall need from $\infFre$. 

For every 1-category $\cC$, we have an additive functor
$\Fun(\cC,\Fre^\kappa) \to \Fun(\cC,\infFre)$ which corresponds to an exact functor
\[
    \mdef{[-]}\colon \Ch_b(\Fun(\cC,\Fre^\kappa)) \to \Fun(\cC,\Ch_b(\Fre^\kappa) \into \Fun(\cC,\infFre).
\]
We refer to $[-]$ as the \tdef{realization functor}. Specifically, the value of $[\sG_\bullet]$ on an object $U\in \cC$ is the complex
\[
\dots\to \sG_k(U) \to \sG_{k-1}(U) \to \dots,
\]
considered as an object of $\infFre$.

Note that, if $\cC$ admits a Grothendieck topology, even if all the $\sG_k$-s are $\Fre^\kappa$-valued cosheaves, the realization of the complex $\sG_\bullet$ might not be a $\infFre$-valued cosheaf. However, it is a cosheaf if $\sG_\bullet$ is bounded and the $\sG_k$-s are \emph{acyclic}; that is, each of them is a cosheaf when considered as a $\infFre$-valued functor.   
\subsubsection{Schwartz Sections of Vector Bundles}

A Nash vector bundle has a notion of a \emph{Schwartz section}, as described in \cite[Definition 5.1.3]{SchNash}. The
space of Schwartz sections of a Nash vector bundle is a (nuclear) Fr\'{e}chet topological vector space.  
\begin{defn}
Let $\sE$ be a Nash vector bundle on a Nash manifold $X$.
We denote by 
\[
\mdef{\Sc_{\sE}}:\Op(X)\to\Fre^{\kappa}
\] 
the functor assigning to an open set $U\subseteq X$ the space of Schwartz
sections of $\sE$ on $U$.
\end{defn}

It is well known (see, e.g., \cite[Proposition 4.4.4]{SchNash}) that ${\cal S}_{\sE}$
is a cosheaf of real vector spaces with respect to the  restricted topology. We shall now show that it remains a cosheaf when considered as a $\infFre$-valued functor.
\begin{prop}
\label{schwarz_sections_cosheaf}
The composition 
\[\Op(X)\oto{{\cal S}_{\sE}}\Fre^{\kappa}\into\Fre_\infty\]
is a $\infFre$-valued cosheaf on $X$.
\end{prop}
\begin{proof}
Since ${\Sc}_{\sE}(\es)=0$, by (the dual of) \Cref{shv_res_mv_norm} its enough to verify the
Mayer-Vietoris axiom for ${\Sc}_{\sE}$. Namely, we have to show that for every pair of open sets $U,V\in \Op(X)$, the square 
\[
\xymatrix{{\cal S}_{\sE}(U\cap V)\ar[r]\ar[d] & {\cal S}_{\sE}(U)\ar[d]\\
{\cal S}_{\sE}(V)\ar[r] & {\cal S}_{\sE}(U\cup V)
}
\]
is a pushout square in $\infFre$. Since $\infFre$ is stable, this is the same as saying that the sequence
\[
0\to \Sc_\sE(U\cap V)\to\Sc_\sE(V)\oplus\Sc_\sE(U)\to\Sc_\sE(U\cup V) \to 0
\]
is (continuously) null-homotopic. Such a null homotopy is provided by a tempered partition of unity for the cover $\{U,V\}$ of $U\cup V$.  
\end{proof}

By abuse of notation, we denote the composition $\Op(X)\oto{\Sc_\sE}\Fre^{\kappa}\into\infFre$
again by $\Sc_{\sE}$.
For various $\sE$-s, the cosheaves complexes $\Sc_{\sE}$ are related. Indeed, let $\sO_X \in \Shv(X,\RR)$ be the sheaf of Nash functions on $X$. The action of $\sO_X$ on  Schwartz sections of $\sE$ gives $\Sc_{\sE}$ the structure of cosheaf of $\sO_X$-modules. Moreover, we have

\begin{prop}
\label{Sc_sect_change_values}
Let $X$ be a Nash manifold and let $\sE,\sE'\in \Vect_X$. We have a canonical isomorphism 
\[
\sE \scten_{\sO_X} \Sc_{\sE'} \simeq \Sc_{\sE\otimes \sE'}. 
\]
\end{prop}

\begin{proof}
First, we shall construct a map  
\[
\sE \scten_{\sO_X} \Sc_{\sE'} \to \Sc_{\sE\otimes \sE'}. 
\]
By the definition of the enriched $\hom$, this is the same as constructing a map 
\[
\sE \to \hom_{\sO_X}(\Sc_{\sE'},\Sc_{\sE\otimes \sE'}),
\]
which, by \Cref{internal_hom} amounts to constructing maps 
$\sE(U)\to \Hom_{\sO_U}(\sE_{\sE'|_U},\sE_{(\sE\otimes \sE')|_U})$, natural in $U\in \Op(X)$. Such a map is provided by the tensor product of Schwartz sections with Nash sections. 

Next, we verify that the resulting map is an isomorphism. It suffices to check this locally on $X$ and hence we may assume that $\sE$ and $\sE'$ are both a sum of copies of the trivial bundle. Since both the source and target of the map $\sE \scten_{\sO_X} \Sc_{\sE'} \to \Sc_{\sE\otimes \sE'},$ regarded as natural transformation of functors in the $\sE$-variable, commute with direct sums, we reduce to the case where $\sE\simeq \sO_X$, which is clear.   
\end{proof}
\subsection{The Relative de Rham Complex}
Let $\phi\colon X\to Y$ be a Nash submersion of relative dimension $n$, and let $\sE$ be a Nash vector bundle on $Y$. For every $0\le i< n$, we have a first-order differential operator 
\[
d\colon C^\infty(X,\phi^*\sE \otimes \Omega^i_\phi) \to C^\infty(X,\phi^*\sE \otimes \Omega^{i+1}_\phi).
\]
It is easy to check that this operator restricts to a map of Schwartz sections and gives a morphism of cosheaves 
\[
    d\colon \Sc_{\phi^*\sE \otimes \Omega^i_\phi} \to \Sc_{\phi^*\sE \otimes \Omega^{i+1}_\phi}.
\]

These maps assemble into a complex in $\CShv(X,\Fre^\kappa)$, which we can realize to an element of 
$\CShv(X,\infFre).$

\begin{defn}
Let $\phi:X\to Y$ be a Nash submersion of relative dimension $n$. For a vector bundle $\sE$
on $Y$, we define
\[
 \mdef{\Sdr_{\phi,\sE}} := 
 [\Sc_{\phi^*\sE} \oto{d} \Sc_{\phi^*\sE\otimes \Omega^1_\phi} \oto{d}... \oto{d} \Sc_{\phi^*\sE\otimes \Omega^n_\phi}]. 
\]
\end{defn}

For the proof of the relative de Rham theorem, it will be convenient to work with a slightly modified version of $\Sdr_{\phi,\sE}$. To introduce it, we recall the \emph{dualizing local system} of Nash submersions. 
\begin{defn}
For a Nash submersion $\phi\colon X\to Y$ of relative dimension $n$, let $\mdef{\Or_\phi} \in \Loc(X,\RR)^\fin$ be the Nash local system of \tdef{relative orientations} along $\phi$ (see \cite[A.1.1]{SchNash}). We denote by 
\[
\mdef{\Du_\phi} \simeq  \Or_\phi[n] \qin \Loc(X,\Der(\RR))
\] 
the \tdef{dualizing local system} of $\phi$. 
\end{defn}

We can now define the twisted relative de Rham-Schwartz complex of a Nash submersion.

\begin{defn}\label{def:tw_rel_der_comp}
Let $\phi\colon X\to Y$ be a Nash submersion, and let $\sE$ be a Nash vector bundle on $Y$. We define
\[
\mdef{\Sdrt_{\phi,\sE}} := \Du_\phi \scten_\RR \Sdr_{\phi,\sE} \qin \CShv(X,\infFre). 
\]
\end{defn} 

The complex $\Sdrt_{\phi,\sE}$, which is a-priori defined in terms of our abstract tensor product operation,
can be described explicitly as the realization of a complex of $\Fre^\kappa$-valued cosheaves. Indeed, the local system $\Or_\phi$ corresponds via \Cref{Riemann-Hilbert_semi_alg} to a vector bundle with flat connection $(\sOr_\phi,\nabla)$ on $X$. Accordingly, for every $0\le i \le n$, we have a covariant differentiation maps 
\[
    \nabla \colon \Sc_{\phi^*\sE \otimes \Omega^i_\phi \otimes \sOr_\phi} \to \Sc_{\phi^*\sE \otimes \Omega^{i+1}_\phi \otimes \sOr_\phi}
\]
which, by the flatness of $\nabla$, glue to a complex 
\[
\Sc_{\phi^*\sE \otimes \sOr_\phi} \oto{\nabla} \Sc_{\phi^*\sE \otimes \Omega^1_\phi \otimes \sOr_\phi} \to ... 
\]

\begin{prop}
The twisted de Rham-Schwartz complex of a Nash submersion $\phi \colon X\to Y$ with values in the vector bundle $\sE$ is given by  
\[
\Sdrt_{\phi,\sE} = [\Sc_{\phi^*\sE \otimes \sOr_\phi} \oto{\nabla}  \Sc_{\phi^*\sE \otimes \Omega^1_\phi \otimes \sOr_\phi} 
\to \dots \oto{\nabla}
\Sc_{\phi^*\sE \otimes \Omega^n_\phi \otimes \sOr_\phi}] \qin \infFre,
\] 
where the term $\Sc_{\phi^*\sE \otimes \Omega^n_\phi \otimes \sOr_\phi}$ lies in degree 0. 
\end{prop}

\begin{proof}
First, for each individual term of the complex, using formal properties of relative tensor products and \Cref{Sc_sect_change_values} we obtain 
\begin{align*}
&\Du_\phi \scten_\RR \Sc_{\phi^*\sE \otimes \Omega_\phi^i} \simeq 
(\Or_\phi \scten_\RR \Sc_{\phi^*\sE \otimes \Omega_\phi^i}) [n] =
(\flat(\sOr_\phi) \scten_\RR \Sc_{\phi^*\sE \otimes \Omega_\phi^i}) [n] \simeq \\
&(\flat(\sOr_\phi) \otimes_\RR \sO_X \scten_{\sO_X} \Sc_{\phi^*\sE \otimes \Omega_\phi^i}) [n] \simeq  
(\sOr_\phi \scten_{\sO_X} \Sc_{\phi^*\sE \otimes \Omega_\phi^i}) [n] \simeq \\
& (\Sc_{\phi^*\sE \otimes \Omega_\phi^i \otimes \sOr_\phi}) [n]
\end{align*}
Moreover, by the construction of the flat connection $\nabla$ on $\sOr_\phi$, the operation $(-)\scten \Or_\phi$ takes the de Rham differential  
$d$ to the covariant derivative $\nabla$. The result now follows from the fact that the functor $\Du_\phi \scten_\RR(-)$ is colimit preserving and hence
\begin{align*}
&\Sdrt_{\phi,\sE}= \Du_\phi \scten_\RR[\Sc_{\phi^*\sE} \oto{d} \Sc_{\phi^*\sE\otimes \Omega^1_\phi} \oto{d}...] \simeq   
\Or_\phi \scten_\RR[\Sc_{\phi^*\sE} \oto{d} \Sc_{\phi^*\sE\otimes \Omega^1_\phi} \oto{d}...][n] \simeq \\
&[\Or_\phi \scten_\RR\Sc_{\phi^*\sE} \oto{d} \Or_\phi \scten_\RR\Sc_{\phi^*\sE\otimes \Omega^1_\phi} \oto{d}...][n]
\simeq 
[\Sc_{\phi^*\sE \otimes \sOr_\phi} \oto{\nabla}  \Sc_{\phi^*\sE \otimes \Omega^1_\phi \otimes \sOr_\phi} 
\to \dots]
\end{align*}
\end{proof}

\subsubsection{The Integration Map}
The last term of the complex $\Sdrt_{\phi,\sE}$ is the cosheaf $\Sc_{\phi^*\sE \otimes \Omega^n_\phi \otimes \sOr_\phi}$, which is the cosheaf of $\sE$-valued Schwartz relative measures along $\phi$. Hence, for every open set $U\subseteq Y$, we have an integration map $\Sc_{\phi^*\sE \otimes \Omega^n_\phi \otimes \sOr_\phi}(\phi^{-1}(U)) \to \Sc_\sE(U)$.
These maps assemble to a morphism of cosheaves 
\[
     \phi_!\Sc_{\phi^*\sE \otimes \Omega^n_\phi \otimes \sOr_\phi} \to \Sc_\sE 
\]
which vanishes after composition with the covariant derivative, by Stokes Theorem. Hence, we obtain a morphism
\[
\mdef{\smallint\nolimits_{\phi}}:\phi_{!}\mathcal{S}DR_{\phi,\sE}^{t}\to\mathcal{S}_{\sE}\qin\CShv(Y,\infFre)
\]
which we refer to as integration along the fibers of $\phi$. 

The primary step in the proof of the relative de Rham Theorem is a local statement about the map $\smallint_\phi$. To state it, we need a local version of the integration map. 
\begin{defn}
    Let $\phi\colon X\to Y$ be a Nash submersion.
    The map $\smallint_\phi \colon \phi_!\Sdrt_{\phi,\sE} \to  \Sc_{\sE}$ induces, by the adjunction $\phi_! \dashv \phi^!$, 
    a map 
\[
    \mdef{\smallint\nolimits^\vee_\phi} \colon \Sdrt_{\phi,\sE} \to \phi^!\Sc_{\sE} \qin \CShv(X,\infFre),
\]
\end{defn}

We may think of $\smallint\nolimits^\vee_\phi$ as the map of integration along the fibers of $\phi$, localized over $X$.

\subsection{The Proof of the Relative de Rham Theorem }
\label{subsec:the_proof_of_RDS}
In this section, we finally turn to prove the relative de Rham Theorem. 
We start by proving a version of Poincar\'{e} Lemma for families of intervals. Then, using the structure theory of Nash submersions, we derive a local version of the general theorem (\Cref{local_relative_de_rham}), and finally, using the local contractability of Nash submersions, we obtain the global version (\Cref{global_relative_de_rham}). 
\subsubsection{The Local Relative de Rham Theorem}
We shall now prove the local version of the relative de Rham theorem. 
We start with
\begin{prop}[Poincar\'{e} Lemma] \label{poincare_lemma}
Let $\pi \colon X\to Y$ be a pointed family of intervals. The map 
\[
\smallint_\phi \colon \phi_!\Sdrt_{\phi,\sE} \to \Sc_{\sE} 
\]
is an isomorphism.
\end{prop}

\begin{proof}
By \Cref{base_gen_shv}, it would suffice to show that, for every open set $U\subseteq Y$ such that $U\simeq \RR^k$ and $\sE$ is trivial on $\sE$, the map $\smallint_\phi\colon \Sdrt_{\phi,\sE}(\phi^{-1}(U))\to \Sc_\sE(U)$ is a homotopy equivalence. Without loss of generality, we may assume that $Y$ itself satisfies these properties and we have to show that $\smallint_\phi$ induces an isomorphism on the global sections. 

Using the decomposition of $\phi$ into an open embedding
$X \into Y\times \RR$ followed by the projection $Y\times \RR \onto Y$,
and a coordinate $t$ on $\RR$, we can write the cofiber of the map $\smallint_\phi$ above as the realization of a 3-step complex 
\[
    0\to \Sc(X)\oto{\partial_t} \Sc(X) \oto{\int_{-\infty}^\infty(-)dt} \Sc(Y) \to 0.
\]
Moreover, we can identify Schwartz functions on $X$ with Schwartz functions on $Y\times \RR$, which are flat on the complement of $X$. 

Since $\phi$ is surjective, by (the proof of) \cite[ Proposition 3.1.2]{SchStack} the map $\smallint_\phi$ admits a continuous section, and hence it would suffice to show that this complex is exact. 

For every $f\in \Ker(\smallint_\phi)$, set 
\begin{equation}
\label{def_H_1}
\mdef{H(f)(x,t)} := \int_{-\infty}^t f(x,\tau) d\tau
\end{equation}
and note that since $f\in \Ker(\smallint_\phi)$ we also have
\begin{equation}
\label{def_H_2}
H(f)(x,t) = -\int_{t}^\infty f(x,\tau) d\tau.
\end{equation}

We claim that for $f\in \Ker(\smallint_\phi)$, the function $H(f)$ is a Schwartz function  on $X$ which satisfies $\partial_t H(f) = f$. The second claim is a direct application of the Newton-Lienbnitz rule, so we shall now prove the first statement.  

We need to prove that, for every polynomial differential operator $D$ on $Y\simeq \RR^k$ and every $k,\ell\ge 0$, the function $D t^\ell \partial_t^m H(f)$ is a bounded function on $Y\times \RR$ which vanishes outside of $X$. By interchanging the derivative and integral, we get 
 \[
D t^\ell \partial_t^m H(f) =t^\ell \partial_t^m H(D f) = 
\begin{cases} 
t^\ell \partial_t^{m-1}D f & \text{ if } m > 0 \\ 
t^\ell H(D f) & \text{ if } m = 0 \\ 
\end{cases}
\]
Hence, for $m>0$, we get a function obtained from $f$ by applying a polynomial differential operator, and the result is clear. For $m=0$, replacing $f$ by $D f$ if necessary, we may assume that $D=1$. 

It remains to show that $t^\ell H(f)$ vanishes outside of $X$ and that it is bounded. 
 Let $y \in Y$ and let $(a,b)$ be the (possibly infinite) interval such that $\phi^{-1}(y) = (a,b)\times \{y\}$. For $t
 \le a$ we have $H(f)(y,t) = 0$ by formula (\ref{def_H_1}) for $H(f)$, and for $t\ge b$ we have $H(f)(y,t) = 0$ by formula (\ref{def_H_2}) for $H(f)$. Finally, since $f$ is a Schwartz function, there is a constant $C>0$ such that 
 \[
 |f(y,t)|\le \frac{C}{(|t|+1)^{\ell+1}}
 \]
 and hence we have
 \[
 |H(f)(y,t)| \le C |t|^\ell  \int_{|t|}^\infty \frac{1}{(1 + |\tau|)^{\ell + 1}} d\tau \le 2C.
 \]

\end{proof}

Using this result, we prove the local version of the relative de Rham theorem. 

\begin{thm}[Local Relative de Rham Theorem]
\label{local_relative_de_rham}
Let $\phi \colon X\to Y$ be a Nash submersion and let $\sE\in \Vect_Y$. The map 
\[
\smallint\nolimits^\vee_\phi \colon \Sdrt_{\phi,\sE} \to \phi^!\Sc_{\sE}
\]
is an isomorphism. 
\end{thm}

\begin{proof}
First, note that the statement is local over $X$, so by \Cref{Implicit_Function_Theorem}, we may assume that $\phi$ is a composition 
$X\into Y\times \RR^n \oto{\pi_n} Y$. Next, for such a map, the morphism $\smallint^\vee_\phi$ is just the restriction to $X$ of $\smallint_{\pi_n}^\vee$, and hence it would suffice to prove that the map 
\[
\smallint\nolimits_{\pi_n}^\vee \colon \Sdrt_{\pi_n,\sE} \to \pi_n^!\Sc_{\sE}
\]
is an isomorphism. We prove this result by induction on $n$. 

Consider first the case $n=1$. In this case, by \Cref{base_gen_shv} and \Cref{fam_int_base}, it would suffice to show that, for every $U\subseteq Y\times \RR$ which is a pointed family of intervals over its image, the map $\smallint_{\pi_1}^\vee \colon \Sdrt_{\pi_1,\sE}(U) \to \pi_1^!\Sc_\sE(U)$ is an isomprphism in $\infFre$. 
We have, by \Cref{fam_int_triv_shape} and \Cref{formula_push_along_locally contractible}:
\[
\pi_1^!\Sc_\sE(U) \simeq \Sc_{\sE}({\pi_1}_\sharp U) \simeq  \Sc_{\sE}(\pi_1(U)).
\]
Via this identification, the map $\smallint_{\pi_1}^\vee \colon \Sdrt_{\pi_1,\sE}(U) \to (\pi_1^!\Sc_\sE)(U)$ corresponds to the integration of sections 
\[
\smallint_{\pi_1}\colon  \Sdrt_{\pi_1,\sE}(U)\to \Sc_{\sE}(\pi_1(U)),
\]
which is an isomorphism by \Cref{poincare_lemma}.

Next, let $n>1$ and assume the result for all smaller $n$-s. 
Choose a splitting $\RR^n = \RR^{n-1} \times \RR$ and let $x_1,..,x_{n-1}$ be coordinates on $\RR^{n-1}$ and $t$ be a coordinate on $\RR$. Let $\pi_1 \colon Y\times \RR^n \to Y\times \RR^{n-1}$ and $\pi_{n-1}\colon Y\times \RR^{n-1} \to Y$ be the projections. 
Using our coordinates, we can present the cofiber of the map $\smallint_{\pi_n}^\vee \colon \Sdrt_{\pi_n,\sE} \to \Sc_{\sE}$ as the realization of the total complex of
\begin{equation}
\label{eq:decom_de_rham}
\xymatrix{
\Sc_{\pi_n^* \sE} \ar^-{d}[r] \ar^{\partial_t(-)dt}[d] &\Sc_{\pi_n^* \sE \otimes \pi_1^*\Omega_{\pi_{n-1}}} \ar^-{d}[r] \ar^{\partial_t(-)dt}[d]& \hdots \ar^-{d}[r]& \Sc_{\pi_n^* \sE \otimes \pi_1^*\Omega_{\pi_{n-1}}^{n-1}}\ar^{\partial_t(-)dt}[d]  \\ 
\Sc_{\pi_n^* \sE \otimes \Omega_{\pi_1}} \ar^-{d}[r]\ar^0[d] &\Sc_{\pi_n^* \sE \otimes \pi_1^*\Omega_{\pi_{n-1}} \otimes \Omega_{\pi_1}} \ar^-{d}[r]\ar^0[d] & \hdots \ar^-{d}[r]& \Sc_{\pi_n^* \sE \otimes \pi_1^*\Omega_{\pi_{n-1}}^{n-1} \otimes \Omega_{\pi_1}} \ar^{\smallint_{\pi_n}^\vee}[d] \\ 
0 \ar[r] &0 \ar[r] & \hdots  \ar[r]& \pi_n^! \Sc_{\sE} \\
}    
\end{equation}

By the case $n=1$, we have cofiber sequences
\[
\Sc_{\pi_n^*\sE \otimes \Omega_{\pi_{n-1}}^k} \oto{\partial_t(-)dt} 
\Sc_{\pi_n^*\sE \otimes \Omega_{\pi_{n-1}}^k\otimes \Omega_{\pi_1}} \oto{\smallint_{\pi_1}^\vee} \pi_1^! \Sc_{\pi_{n-1}^*\sE\otimes \Omega_{\pi_{n-1}}^k}.
\]
These cofiber sequences assemble into an isomorphism of the cosheaf $\pi_1^! \Sdrt_{\pi_{n-1},\sE}$ with the total complex of the first two rows of (\ref{eq:decom_de_rham}), in such a way that the map $\smallint_{\pi_n}^\vee$ correspond to the map
\[
\pi_1^! \smallint\nolimits_{\pi_{n-1}}^\vee\colon \pi_1^! \Sdrt_{\pi_{n-1},\sE}\to \pi_1^!\pi_{n-1}^!\Sc_\sE \simeq \pi_n^! \Sc_\sE.
\] 
The result now follows from the inductive assumption that $\smallint_{\pi_{n-1}}^\vee$ is an isomorphism.  

\end{proof}

\subsubsection{The Global Relative de Rham Theorem}

We now turn to the global version of the relative de Rham theorem. Given the local contractability of Nash submersions, it is a formal consequence of the local version. 
\begin{thm}[Global Relative de Rham Theorem]
\label{global_relative_de_rham}
Let $\phi\colon X\to Y$ be a Nash submersion and let $\sE$ be a Nash vector bundle on $Y$. We have a natural isomorphism 
\[
\phi_!\Sdrt_{\phi,\sE} \simeq \phi_\sharp(\RR_X) \scten \Sc_{\sE}.
\]
\end{thm}

\begin{proof}
By \Cref{local_relative_de_rham}, the map $\smallint_\phi^\vee \colon \Sdrt_{\phi,\sE} \to \phi^!\Sc_\sE$ is an isomorphism. Applying $\phi_!$ we get an isomorphism
\[
\phi_! \Sdrt_{\phi,\sE} \simeq \phi_!\phi^!\Sc_\sE.
\]
Finally, by \Cref{second_projection_formula} applied to $\Sc_\sE$ and the constant sheaf on the value $\RR$, we obtain 
\[
\phi_! \phi^! \Sc_\sE \simeq \phi_!(\RR_X \scten \phi^! \Sc_\sE)\simeq  \phi_\sharp(\RR_X)\scten \Sc_\sE.
\]
\end{proof}

We can derive a similar result to the Schwartz sections of the relative de Rham complex. For a Nash submersion $\phi\colon X\to Y$, set 
\[
\mdef{\phi_!(\sF)} := \phi_\sharp(\sF \otimes \Du_\phi^{-1}).
\]

\begin{thm} [Global relative de Rham Theorem, untwisted version]
Let $\phi\colon X\to Y$ be a Nash submersion, and let $\sE$ be a Nash vector bundle on $Y$. We have a natural isomorphism 
\[
\phi_!\Sdr_{\phi,\sE} \simeq \phi_!(\RR_X)\scten \Sc_{\sE}. 
\]
\end{thm}

\begin{proof}
By definition, 
$\Sdrt_{\phi,\sE} \simeq \Du_\phi \scten \Sdr_{\phi,\sE}$. Consequently, by \Cref{local_relative_de_rham} we have 
\[
\Sdr_{\phi,\sE} \simeq \Du_\phi^{-1} \scten\phi^!\Sc_{\sE}.
\] 
As for the twisted version, this gives
\[
\phi_!(\Sdr_{\phi,\sE}) \simeq \phi_!(\Du_\phi^{-1} \scten\phi^!\Sc_{\sE}) \simeq 
\phi_\sharp(\Du_\phi^{-1})\scten \Sc_\sE \simeq \phi_!(\RR_X)\scten \Sc_\sE.
\]
\end{proof}

\section{Schwartz Sections of Constructible Sheaves}
\label{sec:Sch_const}

The relative de Rham theorem identifies the relative de Rham complex of a Nash submersion $\phi\colon X\to Y$, with the cosheaf $\Sc_\sE \scten \phi_\sharp \RR_X$. To effectively use this result, we need to control the sheaves of the form $\phi_\sharp \RR_X$ and their interaction with tensor product operation. 
It turns out that those sheaves are always \emph{constructible}, and that this property has significant effect on the resulting cosheaves $\Sc_\sE \scten \phi_\sharp \RR_X$.

In this section,
we study the behavior of the cosheaves of the form $\sF \scten \Sc_\sE$ for a general  constructible sheaf valued in $\Der(\RR)$. In other words, we consider $\sF$-s, which are bounded and have constructible homologies.
Our main objective is to show that, when $\sF$ is constructible, the topological vector spaces $H_i(\sF\scten \Sc_{\sE}(X))$ are Hausdorff (\Cref{scten_const_hauss}), and deduce that the Schwartz sections of the relative de Rham complex has Hausdorff homologies (\Cref{rel_der_haus}). 
We also use the constructability of $\phi_\sharp \RR_X$ to show that $\Sdrt_{\phi,\sE}$ depends only on the "homology type" of the fibers of $\phi$ (\Cref{homology_inv_rel_der}). 

\subsection{Constructible Sheaves}
Classically, one studies complexes of sheaves on a space with constructible cohomologies. This notion has a convenient $\infty$-categorical manifestation. 

\begin{defn}
Let $X$ be a closureful restricted topological space and let $\cD$ be a compactly generated $\infty$-category. 
A sheaf $\sF\in \Shv(X,\cD)$ is called
\begin{itemize}
\item \tdef{$\alpha$-stratified} for a stratification $\alpha \colon X\to P$, if for every $p\in P$ the restriction $\sF|_{X_p}$ is locally constant. 
\item \tdef{stratified} if it is $\alpha$-stratified for some stratification $\alpha$.  
\item \tdef{($\alpha$-)constructible} if it is ($\alpha$-)stratified and all its stalks are compact objects of $\cD$. 
\end{itemize}
\end{defn}

We denote the full subcategories of $\Shv(X,\cD)$ spanned by the ($\alpha$-)stratified and ($\alpha$-)constructible sheaves on $X$ by $\mdef{\Shv_\st^{(\alpha)}(X,\cD)}$ and $\mdef{\Shv_c^{(\alpha)}(X,\cD)}$ respectively.

In the case of restricted topological space of \emph{finite inductive dimension}, the notion of a stratified sheaf can be simplified. 

\begin{prop}
\label{strat_have_triv_strat}
Let $X$ be a closureful restricted topological space of finite inductive dimension. For every stratified sheaf $\sF \in \Shv_\st(X,\cD)$, there is a proper stratification $\alpha\colon X\to P$ such that $\sF|_{X_p}$ is a \emph{constant} sheaf for every $p\in P$. In particular, there are disjoint open sets $U_k$ in $X$ such that $\bigcup_k U_k$ is dense in $X$ and $\sF|_{U_k}$ is constant.  
\end{prop}

\begin{proof}
Let $\beta \colon X\to Q$ be a stratification for which the restrictions of $\sF$ to the strata are locally constant. For every $q\in Q$, since $\sF|_{X_q}$ is locally constant, we can find a cover 
\[X_q = U_{1,q}\cup...\cup U_{k,q}\] such that $\sF|_{U_{i,q}}$ is constant. Let $\Theta$ be the collection of all the $U_{j,q}$-s. By \Cref{Existence_classical_stratification_compatible_with_collection}, we can find a proper stratification $\alpha \colon X\to P$ for which every member of $\Theta$ is a union of strata. Such a stratification satisfies the required property. For the "in particular" part, just choose the $U_k$-s to be the open strata of such a stratification.
\end{proof}

While constructible sheaves can be defined over every closureful restricted topological space, the theory works better under certain assumptions on the space $X$. Hence, we consider from now on only constructible sheaves over the following type of restricted topological spaces. 
\begin{defn} \label{def: well behaved}
    Let $X$ be a restricted topological space. We say that $X$ is \tdef{well behaved} if $X$ is closureful of finite inductive dimension, and every locally closed subset of $X$ is of topological shape.
\end{defn}
\begin{example}
    Every semi-algebraic space is well behaved. Indeed, this follows from Propositions \ref{Semialgebraic_Spaces_closureful_and_finite_dimensional} and \ref{Comparison_Shape_Semi_Algebraic_Space}.
\end{example}
Unlike for general sheaves, isomorphisms between stratified sheaves over well behaved restricted topological spaces (in the sense of \Cref{def: well behaved}) \emph{can} be detected on the stalks.
\begin{prop}
\label{iso_const_detect_by_stalks}
Let $X$ be a well behaved restricted topological space and let $\cD$ be a compactly generated $\infty$-category. Then, a morphism $f\colon \sF \to \sG$ in  $\Shv_\st(X,\cD)$ is an isomorphism if and only if, for every $x\in X$, the map $f_x \colon \sF_x \to \sG_x$ is an isomorphism. 
\end{prop}

\begin{proof}
By passing to a common refinement, we may assume that both $\sF$ and $\sG$ are stratified for the same stratification $\alpha\colon X\to P$.  
By \Cref{shv_rec} applied iteratively, it would suffice to show that for every $p\in P$, the map $f$ induces an isomorphism $\sF|_{X_p}\simeq \sG|_{X_p}$. Hence, by restricting both $\sF$ and $\sG$ to $X_p$ we may assume that $\sF$ and $\sG$ are locally constant. 

Since $X$ is well behaved, the $\infty$-topos $\Shv(X)$ is essential, so that by \Cref{Local_systems_sheaves_on_shape} we can regard $\sF$ and $\sG$ as functors $|X|\to \cD$. 
Finally, every point $x\in X$ determines a point in $|X|$, and these points cover all the connected components of $|X|$. Hence, the result follows from the fact that isomorphisms of functors from $|X|$ can be tested pointwise. 
\end{proof}

Stratified and constructible sheaves are closed under various categorical operations.

\begin{prop}
\label{cl_struct_shv}
Let $X$ be a well behaved restricted topological space and $\cD$ be a compactly generated $\infty$-category. 
\begin{enumerate}
    \item For every stratification $\alpha \colon X\to P$, the subcategory $\Shv_\st^\alpha(X,\cD)\subseteq \Shv(X,\cD)$ is closed under all colimits and finite limits in $\Shv(X,\cD)$. 
    \item $\Shv_\st(X,\cD)$ is closed under  finite limits, finite colimits and retracts in $\Shv(X,\cD)$. 
    \item $\Shv_c(X,\cD)$ is closed under finite colimits and  retracts in $\Shv(X,\cD)$.  
    \end{enumerate}
\end{prop}

\begin{proof}
For $(1)$, restriction of sheaves to every stratum $X_p$ of $\alpha$ preserves all colimits and finite limits, and hence it suffices to prove the claim for local systems on a restricted topological space which is locally of constant shape. Since in this case $\Shv(X)$ is essential, $\Loc(X,\cD)$ is the essential image of the fully faithful embedding $\psi^*\colon \Shv(|X|,\cD)\to \Shv(X,\cD)$, which is left exact and commutes with all small colimits. 

$(2)$ follows from $(1)$ since $\Shv_\st(X,\cD) = \bigcup_\alpha \Shv_\st^\alpha(X,\cD)$.

Finally, $(3)$ follows from $(2)$ and the fact that compact objects in a compactly generated $\infty$-category are closed under finite colimits and retracts.
\end{proof}

\subsubsection{Intrinsic Characterisation of Constructible Sheaves}
It turns out that for \emph{stable} $\infty$-category of coefficients, we can identify the constructible sheaves on $X$ categorically. 

\begin{prop} \label{const_comp} 
    Let $X$ be a well behaved restricted topological space and let $\cD$ be a stable compactly generated $\infty$-category. Then $\Shv_c(X,\cD)$ coincides with the full subcategory of $\Shv(X,\cD)$ spanned by the compact objects. Moreover, it is the stable subcategory of $\Shv(X,\cD)$ generated by either of the following collections 
    \begin{enumerate}
        \item $j_!d_U$ for an open embedding $j\colon U\into X$ and $d\in \cD$ a compact object. 
        \item $i_*d_Z$ for $i\colon Z\into X$ a closed embedding and $d\in \cD$ a compact object. 
    \end{enumerate}
\end{prop}

\begin{proof}
First, note that the collections in items $(1)$ and $(2)$ generate the same stable subcategory. Indeed, this follows from the fact that, for every open set $U\subseteq X$ and compact object $d\in \cD$, we have a cofiber sequence $j_!(d_U)\to d_X \to i_*d_Z$. Here, $j\colon U\into X$ and $i\colon Z\into X$ are the inclusions of $U$ and its complement. Hence, it would suffice for the second part of the claim to show that the stable subcategory generated by the union of these two collections coincides with $\Shv_c(X,\cD)$. Let $\sA\subseteq \Shv(X,\cD)$ be the stable subcategory generated by these two collections, and let $\Shv(X,\cD)^\omega$ be the full subcategory spanned by the compact objects. We will prove the result by showing that 
\[
    \Shv_c(X,\cD) \subseteq \sA \subseteq \Shv(X,\cD)^\omega \subseteq \Shv_c(X,\cD). 
\]

First, note that every sheaf of the form $j_!d_U$ is a compact object by \Cref{cpt_gen_shv_res}. We immediately deduce that $\sA\subseteq \Shv(X,\cD)^\omega$. 
Next,
the sheaves of the form $j_! d_U$ generate $\Shv(X,\cD)$ under colimits. Thus, every compact object of $\Shv(X,\cD)$ is a retract of a finite colimit of sheaves of the form $j_! d_U$. It follows from \Cref{cl_struct_shv}(3) and the fact that those sheaves are constructible, that every compact object of $\Shv(X,\cD)$ is constructible. 

It remains to prove that $\Shv_c(X,\cD)\subseteq \sA$. We prove this by induction on the dimension of $X$. Let $\sF\in \Shv_c(X,\cD)$. By \Cref{strat_have_triv_strat}, there is are disjoint open sets $j_k \colon U_k \into X$ whose union is dense in $X$ and such that $j_k^*\sF\simeq (d_k)_{U_k}$ for compact objects $d_k\in \cD$. Let $j\colon \bigcup_k U_k \into X$ be the inclusion of the union. Then, $j_!j^*\sF$ is a direct sum of the sheaves $(j_k)_!j_k^*\sF$ which belong to $\sA$, and hence $j_!j^*\sF$ belongs to $\sA$. 

Let $i\colon Z \into X$ be the closed complement of $U$ in $X$. Since $Z$ is of a smaller dimension than $X$, we may assume by induction that $i^*\sF$ is in the stable subcategory generated by the sheaves in item ($2$). This implies that the same holds for $i_*i^*\sF$ since $i$ is a closed embedding. The result now follows from the cofiber sequence 
\[
j_!j^*\sF \to \sF \to i_*i^*\sF
\]
from \Cref{rec_props} and the stability of $\cD$.   
\end{proof}

Using this intrinsic description of the constructible sheaves, we can show that they are preserved under the exceptional push-forward functor. 
\begin{prop} \label{push_const_to_const}
Let $X$ and $Y$ be well behaved restricted topological spaces and let $\phi \colon X\to Y$ be a map inducing an essential geometric morphism 
$\Shv(X)\to \Shv(Y)$. For every compactly generated stable $\infty$-category $\cD$, the functor  
\[
    \phi_\sharp\colon \Shv(X,\cD)\to \Shv(Y,\cD)
\] 
takes constructible sheaves to constructible sheaves. 
\end{prop}

\begin{proof}
In view of \Cref{const_comp}, it suffices to show that $\phi_\sharp$ preserves compact objects. This follows from the fact that its right adjoint, $\psi^*$, preserves filtered colimits.  
\end{proof}

\subsubsection{Comparison of Topologies for Constructible Sheaves}
For a restricted topological space $X$, the functor $\real_*\colon \Shv(X^\topify,\cD) \to \Shv(X,\cD)$ is fuly faithful (\Cref{Comparison_map_fuly_faithful}), but it is far from an isomorphism. For constructible sheaves, the comparison of topologies is better behaved.
To state the precise result, we first define stratified and constructible sheaves for the \emph{topology} of a restricted topological space (rather than the restricted topology). 
\begin{defn}
Let $X$ be a closureful restricted topological space. A sheaf $\sF\in \Shv(X^\topify,\cD)$ is called
\begin{enumerate}
    \item \tdef{$\alpha$-stratified} for a stratification $\alpha\colon X\to P$ if $\sF_{X_p^\topify}$ is locally constant for every $p\in P$. 
    \item \tdef{Stratified} if it is $\alpha$-stratified for some stratification $\alpha$. 
    \item \tdef{($\alpha$-)constructible} if it is ($\alpha$-)stratified and has compact stalks. 
\end{enumerate}
\end{defn}

We denote the relevant $\infty$-categories by $\Shv_\st^{(\alpha)}(X^\topify,\cD)$ and $\Shv_c^{(\alpha)}(X^\topify,\cD)$. 

\begin{warning}
Note that the notions of stratified and constructible sheaf on $X^\topify$ depend on the restricted topology on $X$ and not only on the topology it generates. Indeed, we require sheaves in $\Shv_\st(X^\topify,\cD)$ to be stratified for a stratification of $X$.
\end{warning}

For restricted topological spaces which are locally of topological shape, we had a comparison result for local systems (\Cref{top_shape_loc_sys}). For well-behaved restricted topological spaces, we can extend this comparison to constructible sheaves as well.
We start with the following useful special case of this comparison. 
\begin{lem} \label{gens_constructible_topify}
Let $X$ be a well behaved restricted topological space, and let $\cD$ be a compactly generated stable $\infty$-category. Then $\Shv_c(X^\topify,\cD)$ is the stable subcategory of $\Shv(X^\topify,\cD)$ generated by the sheaves of the form $i_*d_{Z^\topify}$ for $i\colon Z\into X$ a closed subset and $d\in \cD$ a compact object. 
\end{lem}  

\begin{proof}
As in the proof of \Cref{const_comp}, by induction on the dimension of $X$ it would suffice to show that there are disjoint open sets $U_k$ in $X$ such that $\sF|_{U_k^\topify}$ is constant and $\bigcup_k U_k$ is dense in $X$. Choose a stratification $\alpha\colon X\to P$ such that $\sF$ is locally constant on $X_p^\topify$ for every $p\in P$. By replacing $X$ with the union of the open strata of $\alpha$, we may assume that $\sF$ is locally constant. 

Since $X$ is well behaved, and in particular, locally of topological shape, we deduce from \Cref{top_shape_loc_sys} that $\sF\simeq \real^*\sF_0$ for some $\sF_0\in \Loc(X,\cD)$. Since $\sF_0$ is constructible (in fact, locally constant) we can find disjoint open sets $U_k\subseteq X$ with dense union, such that $(\sF_0)|_{U_k}$ is a constant sheaf. This implies that $\sF|_{U_k^\topify}$ is a constant sheaf and the result follows.  
\end{proof}
We are ready to prove our comparison result for constructible sheaves. 

\begin{thm} \label{comp_const_st}
    Let $X$ be a well-behaved restricted topological space and let $\cD$ be a stable compactly generated $\infty$-category.
    Then, the adjunction $\real^*\colon \Shv(X,\cD) \adj  \Shv(X^\topify,\cD) : \real_*$ restricts to an isomorphism
\[
    \Shv_c(X,\cD) \simeq \Shv_c(X^\topify,\cD)
\]
    Moreover, for every stratification $\alpha\colon X\to P$, the isomorphism $\real^*$ takes the full subcategories 
$\Shv_c^\alpha(X,\cD)$ isomorphically onto $\Shv_c^\alpha(X^\topify,\cD)$.
\end{thm}

\begin{proof}
First, note that $\real_*$ takes constructible sheaves to constructible sheaves. 
Indeed, by \Cref{gens_constructible_topify}, it suffices to show that $\real_* i_*d_{Z^\topify}$ is constructible for every closed embedding $i\colon Z\into X$ and compact object $d\in \cD$. But since $Z$ is locally of topological shape, we deduce that $\real_*i_*d_{Z^\topify}\simeq u_*d_Z$, which is constructible. 

Since $\real_*$ is fully faithful, to show that it restricts to an isomorphism $\Shv_c(X^\topify,\cD)\iso \Shv_c(X,\cD)$, it remains to show that the unit $\sF\to \real_*\real^*\sF$ is an isomorphism for every $\sF\in \Shv_c(X,\cD)$. By \Cref{const_comp}, it would suffice to show this for sheaves of the form $\sF\simeq i_*(d_Z)$. 
By \Cref{Base_change_comparison}, we can identify the unit map $i_*d_Z \to \real_*\real^*i_*d_Z$ with the image of the unit map $d_Z\to \real_*\real^*d_Z$ under the functor $i_*$. Since $Z$ is itself well behaved, we reduced to show that this unit is an isomorphism at constant sheaves, and this case follows from \Cref{top_shape_loc_sys}. 

Finally, to show that $\real^*$ take $\Shv_c^\alpha(X,\cD)$ onto $\Shv_c^\alpha(X^\topify,\cD)$, just note that for every constructible sheaf $\sF$ on $X$ and locally closed subset $S\subseteq X$, the restriction $\sF|_S$ is locally constnat if and only if $\real^*\sF|_{S^\topify}$ is locally constant, again by \Cref{top_shape_loc_sys}. 
\end{proof}

\begin{rmk}
One can similarly show that $\real^*$ induces an isomorphism on the full subcategories spanned by the ($\alpha$-)stratified sheaves. However, we shall not need this result, so we leave the necessary modification for the interested reader. 
\end{rmk}
\subsubsection{The Relative de Rham Complex \& Homological Equivalences}
Informally, one way to interpret the relative de Rham Theorem (\Cref{global_relative_de_rham}) is that for a Nash submersion $\phi\colon X\to Y$ the complex $\Sdrt_{\phi,\sE}$ is a mixture of the Schwartz space of $Y$ and the homologies of the fibers of $\phi$. We shall now justify this intuition by showing that the complex $\Sdrt_{\phi,\sE}$ is invariant under (fiberwise) homology equivalences relative to $Y$.  
 
First, we make a general remark about the relation between the shape of a semi-algebraic space and its singular homology. For a topological space $X$, we denote by $C_*(X,\RR)\in \Der(\RR)$ the singular chain complex of $X$ with values in $A$. 

\begin{prop} \label{sing_shape_Nash_man}
    Let $X$ be a semi-algebraic space. Then, we have a natural  isomorphism 
\[
    C_*(X^\topify,\RR)\simeq |X|\otimes \RR. 
\]
\end{prop}
 
Note that here $|X|\otimes \RR$ is defined via the canonical tensoring of $\Der(\RR)$ over $\Spaces$, and agrees with the real singular homology of the homotopy type $|X|$. 
 
 \begin{proof}
    First, by \Cref{Comparison_Shape_Semi_Algebraic_Space} we have $|X|\simeq |X^\topify|$. Now, $X^\topify$ is a locally contractible, paracompact topological space, and hence it is locally of singular shape in the sense of \cite[Definition A.4.9]{HA} (see, e.g., \cite[Remark A.4.11]{HA}). We deduce that 
 \[
    |X^\topify|\otimes \RR \simeq \Sing(X^\topify)\otimes \RR \simeq C_*(X^\topify,\RR).  
 \]
 \end{proof}
 
 We are ready to show that the relative de Rham complex is invariant under relative homology equivalences.

\begin{thm} \label{homology_inv_rel_der}
    Let 
\[
    \xymatrix{
    X_0 \ar^\phi[rr] \ar^{\alpha_0}[rd] &   & X_1 \ar^{\alpha_1}[ld] \\ 
    & Y &  
    }
\]
    be a commutative diagram of Nash manifolds, with $\alpha_i$ Nash submersions. If, for every $y\in Y$, the induced map 
    $\alpha_*\colon C_*((X_0)_y^\topify,\RR)\to C_*((X_1)_y^\topify,\RR)$ on the singular chains of the fibers is an isomorphism in $\Der(\RR)$, then there is a natural isomorphism
\[
    \Sdrt_{\alpha_0,\sE} \simeq \Sdrt_{\alpha_1,\sE} \qin \CShv(Y,\infFre).
\]
\end{thm}

\begin{proof}
By \Cref{global_relative_de_rham}, we have natural isomorphisms 
\[
(\alpha_i)_!\Sdrt_{\alpha_i,\sE} \simeq (\alpha_i)_\sharp\RR_X \scten \Sc_{\sE}, \quad i=0,1.
\]
Hence, it would suffice to show that $\phi$ induces an isomorphism $(\alpha_0)_\sharp \RR_X \iso (\alpha_1)_\sharp \RR_Y$.

By \Cref{push_const_to_const} the sheaves $(\alpha_i)_\sharp \RR_X$ are constructible, and hence by \Cref{iso_const_detect_by_stalks} it would suffice to show that $\phi$ induces isomorphisms $((\alpha_0)_\sharp \RR_X)_y \iso ((\alpha_1)_\sharp \RR_X)_y$ for every $y\in Y$. Combining \Cref{Proper base-change} and \Cref{sing_shape_Nash_man}, we can identify these maps with the maps $\phi_*\colon C_*((X_0)_y^\topify,\RR) \to C_*((X_1)_y^\topify,\RR)$, which are isomorphisms by our assumption on $\phi$.
\end{proof}

\begin{rmk}
If, in the result above, $\phi$ is itself a Nash submersion, one can realize the claimed isomorphism via integration of relative differential forms along the fibers of $\phi$.
\end{rmk}

\subsection{Pseudo-Free Resolutions}

In order to compute the cosheaves of the form $\sF\scten \Sc_\sE$ appearing in the statement of the relative de Rham theorem, we shall use a convenient presentation of real-valued constructible sheaves. 

As for cosheaves, for a (restricted) topological space $X$, we have a functor 
\[
[-]\colon \Ch_b(\Shv(X,\RR))\to \PShv(X,\Der(\RR))\oto{L_X} \Shv(X,\Der(\RR)) 
\]
which we refer to (once again) as the \tdef{realization functor}. 

\begin{defn} (see \cite[Definition 1]{yekutieliderived})
Let $X$ be a restricted topological space. A \tdef{pseudo-free sheaf} of real vector spaces on $X$ is a sheaf of the form $\bigoplus_{i=1}^n \RR_{U_i}$ for some open sets $U_i\in \Op(X)$. A \tdef{pseudo-free resolution} of a sheaf $\sF\in \Shv(X,\Der(\RR))$ is an isomorphism $\sF \simeq [C_\bullet]$ where $C_\bullet$ is a bounded complex of pseudo-free sheaves on $X$.   
\end{defn}

This section aims to show that, on a Nash manifold, every finite constructible sheaf admits a pseudo-free resolution. 
\begin{rmk}
Of course, this result is expected to hold in much greater generality (e.g., for semi-algebraic spaces), but we restrict our attention to the case that is used in the proof of the Hausdorffness of the homologies of $\Sdrt_{\phi,\sE}$.
\end{rmk}
\subsubsection{Constructible Sheaves on Simplicial Complexes}
We start by providing pseudo-free resolutions for constructible sheaves on the geometric realizations of almost simplicial complexes (see \Cref{def:almost_simp_comp}).
For simplicity, and since this is the only case we need, we consider only almost simplicial complexes which are open in their simplicial closure. 
\begin{defn}
Let $K$ be an almost simplicial complex on a set of vertices $A$. We denote by $\overline{K}$ the simplicial closure of $K$, that is, the minimal simplicial complex on the set of vertices $A$ which contains $K$. We say that $K$ is \tdef{locally closed} if $K$ is open in $\overline{K}$.
\end{defn}

Equivalently, $K$ is locally closed if for every $\sigma_0\subseteq\sigma_1 \subseteq \sigma_2 \in \Pow(A)$, if $\sigma_0$ and $\sigma_2$ belong to $K$ then so does $\sigma_1$. 

For a locally closed almost simplicial complex $K$, we can give a convenient presentation for the constructible sheaves on $\triangle_K$. 
Let $X$ be a closureful restricted topological space, and let $\alpha\colon X\to P$ be a stratification. We have a canonical functor $P^\op\oto{(-)^\star} T(X) \oto{\Yo} \Shv(X)$, which takes $p\in P$ to the sheaf represented by the open star $p^\star$. Since the target of this functor admits all small colimits, it extends uniquely to a colimit preserving functor 
\[
    \Fun(P,\Spaces)\simeq \PShv(P^\op)\to \Shv(X).
\]
For $\cD\in \Prl$, we can now tensor with $\cD$ this map to get a colimit preserving functor
\[
    \mdef{\alpha^\star} \colon \Fun(P,\cD)\to \Shv(X,\cD).
\]

\begin{prop}\label{const_shv_simp}
Let $K$ be a locally closed almost simplicial complex, and let $\cD$ be a stable, compactly generated $\infty$-category. Let $\cD^\omega\subseteq \cD$ be the full subcategory spanned by the compact objects. The composition 
\[
     \Fun(K,\cD^\omega) \subseteq \Fun(K,\cD) \oto{\alpha_K} \Shv(\triangle_K,\cD)
\]
is fully faithful, with essential image $\Shv_c^{\alpha_K}(\triangle_K,\cD).$
\end{prop}

\begin{proof}
First, by \Cref{comp_const_st} we have $\Shv_c^{\alpha_K}(\triangle_K,\cD)\simeq \Shv_c^{\alpha_K}(\triangle_K^\topify,\cD)$. 
Next, we shall show that $\Fun(K,\cD)\iso\Shv_\st^{\alpha_K}(\triangle_K^\topify,\cD).$
If $\cD\simeq \Spaces$ and $K$ is a simplicial complex, then this is a consequence of \cite[Theorem A.6.10]{HA} and \cite[Theorem A.9.3]{HA}. 
The case where $K$ is only a locally closed almost simplicial complex, follows from the case of a simplicial complex since $\Shv_\st^{\alpha_K}(\triangle_K^\topify,\cD)$ is the full subcategory of $\Shv_\st^{\alpha_{\overline{K}}}(\triangle_{\overline{K}}^\topify,\cD)$ spanned by the sheaves with vanishing stalks at the interiors of simplices that does not belong to $K$, while $\Fun(K,\Spaces)$ identifies, via left Kan extension, with the full subcategory of $\Fun(\overline{K},\Spaces)$ spanned by the functors which vanish on simplices not in $K$.  
Finally, to show the claim for general $\cD$, it suffices to show that $\Shv_\st^{\alpha_K}(\triangle_K^\topify,\cD) \simeq \Shv_\st^{\alpha_K}(\triangle_K^\topify)\otimes \cD$. This follows from \Cref{Local_systems_sheaves_on_shape} and the fact that the tensor product with $\cD$ in $\Prl_\st$ commutes with pullbacks (this, in turn, follows from the dualizability of $\cD$, see, e.g., \cite[\S 0.6.7]{GR}). Indeed, we have  
\[
\Shv_\st^{\alpha_K}(\triangle_K^\topify,\cD) \simeq \Shv(\triangle_K^\topify,\cD)\times_{\Shv(\coprod_{\sigma\in K} \triangle_\sigma,\cD)} \Loc(\coprod_{\sigma\in K} \triangle_\sigma,\cD) 
\]
and each of the terms in this pullback is obtained by tensoring with $\cD$ the corresponding category of sheaves of spaces. 

It remains to show that the isomorphism $\Fun(K,\cD)\iso \Shv_\st^{\alpha_K}(\triangle_K^\topify,\cD)$ carry $\Fun(K,\cD^\omega)$ isomorphically onto $\Shv_c^{\alpha_K}(\triangle_K^\topify,\cD)$. In other words, we have to verify that for every $\sF\in \Fun(K,\cD)$, the sheaf $\alpha^\star(\sF)$ has compact stalks if and only if $\sF$ has only compact values. In one direction, since $K$ is a finite poset, every $\cD^\omega$-valued functor on $K$ can be presented as a finite colimit of functors of the form $\Map(\sigma,-)\otimes d$ for $d\in \cD^\omega$ and $\sigma \in K$. Since the functor $\alpha_K^\star$ takes this sheaf to the extension by zero $d_{\sigma^\star}$ which has compact stalks, we deduce that this is the case for the entire essential image of $\alpha_K^\star$. 

In the opposite direction, the functor $\Shv_c^{\alpha_K}(\triangle_K^\topify,\cD) \to \Fun(K,\cD)$ takes a sheaf $\sF$ to the functor $\sigma \mapsto \sF(\sigma^\star)$. Hence, it remains to show that the sections of a constructible sheaf on every open set of $\triangle_K$ is a compact object of $\cD$. By \Cref{const_comp}, it suffices to show this for sheaves of the form $i_*d_{Z^\topify}$ for a closed subset $Z\subseteq \triangle_K$. In this case, 
\[
\Gamma(\sigma^\star,i_*d_{Z^\topify})\simeq \Gamma(\sigma^\star\cap Z^\topify,\Gamma^*d)\simeq d^{|\sigma^\star\cap Z^\topify|},
\]
and the result follows from the fact that $|\sigma^\star\cap Z^\topify|$ is a finite space, in this case, being the shape of a topological space that admits a finite triangulation.
\end{proof}

We shall now specialize to the case $\cD= \Der(\RR)$ and exploit the presentation of $\Shv_c^{\alpha_K}(\triangle_K,\cD)$ to resolve its objects by some specific pseudo-free sheaves. 

\begin{prop} \label{pf_res_simp_comp}
Let $K$ be a locally closed almost simplicial complex. Then, every object of $\Shv_c^{\alpha_K}(\triangle_K,\Der(\RR))$ admits a finite pseudo-free resolution by a complex of the from 
\[
\dots \to  \bigoplus_{t\in T_{k+1}}\oto{d} \RR_{U_t}\bigoplus_{t\in T_k} \RR_{U_t} \oto{d} \bigoplus_{t\in T_{k-1}} \RR_{U_t} \to \dots
\]
for (finitely many) finite sets $T_k$ and open sets $U_t$, which are open stars of simplices in $K$.
\end{prop}

\begin{proof}
By the equivalence 
\[
\Shv_c^{\alpha_K}(\triangle_K,\Der(\RR))\simeq \Fun(K,\Der(\RR)^\omega)
\]
provided by \Cref{const_shv_simp}, it would suffice to resolve every object of $\Fun(K,\Der(\RR)^\omega)$ by a bounded complex of functors which are direct sums of functors of the form $\Map(\sigma,-)\otimes \RR$ for $\sigma \in K$. 
Since the sheaves of the form $\RR[\Map(\sigma,-)]$ for $\sigma\in K$ are projective generators of $\Fun(K,\Der(\RR))$ and of the abalian subcategory $\Fun(K,\Mod_\RR)$, we see that $\Fun(K,\Der(\RR))\simeq \Der(\Fun(K,\Mod_\RR))$. Hence, we can use the classical theory of representations of finite posets to tackle the problem. Specifically, the result now follows from \cite[Proposition 2.4]{perling2005resolutions}.
\end{proof}

\subsubsection{pseudo-free resolutions of constructible sheaves}
Combining the result above with the triangulation results we have for Nash manifolds, we now show 
\begin{thm}
\label{exs_ps_free_res}
Let $X$ be a Nash manifold. Every object of $\Shv_c(X,\Der(\RR))$ admits a finite pseudo-free resolution.
\end{thm}

\begin{proof}
Assume first that $X$ is affine, and let $\sF\in \Shv_c(X,\Der(\RR))$. Then, by \Cref{Existence_of_triangulation}, $X$ admits a triangulation $\triangle_K\iso X$, which can be chosen to be compatible with a stratification trivializing $\sF$. Since $X$ is a Nash manifold, $K$ is necessarily locally closed. Hence, we have $\sF\in \Shv_c^{\alpha_K}(\triangle_K,\Der(\RR))$ and we have a bounded pseudo-free resolution for $\sF$ by \Cref{pf_res_simp_comp}. 

Now, let $X$ be any Nash manifold, and write $X = \bigcup_{i=1}^n U_i$ with $U_i$ affine. We have seen the result for $n=1$ above, and we proceed by induction on $n$. 
Let $\mdef{U} := U_1$, $\mdef{V}:= \bigcup_{i=2}^nU_i$ and $\mdef{W}:= U\cap V$. Denote by $j_U,j_V$ and $j_W$ their embeddings into $X$ respectively. 
By the inductive hypothesis, $j_U^*\sF$ and $j_V^*\sF$ admit bounded pseudo-free resolutions $C_\bullet$ and $D_\bullet$ respectively. Since $W$ is itself an affine Nash manifold, we can choose a triangulation $\triangle_K\iso W$, which is moreover compatible with a trivializing stratification of $j_W^*\sF$, and also with trivializing stratifications of all the sheaves $j_W^*C_k$ and $j_W^*D_k$ for $k\in \ZZ$.  
By \Cref{pf_res_simp_comp}, we have a pseudo-free resolution 
\[
E_\bullet \colon 
\dots\oto{d}
\bigoplus_{t\in T_k}\RR_{U_t} \oto{d} \bigoplus_{t\in T_{k-1}}\RR_{U_{t}}\oto{d}
\dots
\]
of $j_W^*\sF$ such that the $U_t$-s are open stars of simplices in $K$.

We can now write the sheaf $\sF$ as a pushout 
\begin{equation}
\label{eq:sq_push_res}
\xymatrix{
[(j_W)_! E_\bullet] \ar[r] \ar[d]& [(j_V)_!D_\bullet] \ar[d]
\\ 
[(j_U)_!C_\bullet] \ar[r] & \sF 
}
\end{equation}

The map $[(j_W)_! E_\bullet] \to [(j_U)_! C_\bullet]$ has a mate, which is a map $f\colon [E_\bullet] \to [j_W^*C_\bullet]$. Both the source and the target of this map are the realizations of complexes in $\Shv_c^{\alpha_K}(\triangle_K,\Mod_\RR) \simeq \Fun(K,\Mod_\RR^\omega)$, and all the terms of the complex $E_\bullet$ are projective objects. Hence, we can present $f$ as a map of complexes $f_\bullet \colon E_\bullet \to j_W^*C_\bullet$. Passing back to the mate and arguing similarly for $D_\bullet$, we see that all the maps in the diagram (\ref{eq:sq_push_res}) can be lifted to maps of complexes of sheaves (rather than maps of sheaves of complexes). It follows that $\sF$ can be presented as the cone of a map of complexes of $\Mod_\RR$-valued sheaves 
$(j_W)_!E_\bullet \to  (j_V)_!D_\bullet \oplus (j_U)_!C_\bullet$. Since the cone of a map of complexes of pseudo-free sheaves is a complex of pseudo-free sheaves, we get the result. 
\end{proof}

\subsection{Hausdorffness of the Schwartz Sections}
In this final section, we turn to the second part of \Cref{intro-rel_der} by showing that the relative de Rham complex of a Nash submersion has Hausdorff homology spaces. We will show this more generally for complexes of the form $(\sF \scten \Sc_\sE)(X)$ where $\sF$ is a real-valued constructible sheaf on the Nash manifold $X$.

\subsubsection{Criterion for Hausdorffness}
Our proof of the Schwartz sections' Hausdorffness uses a variant of the notion of pseudo-inverse.

\begin{defn}
Let $f:A\to B$ be a morphism in a ($1$-)category. A \tdef{right pseudo-inverse}
for $f$ is a morphism $g:B\to A$ such that $fgf=f$. 
\end{defn}

Note that if a map $f:A\to B$ of topological vector spaces admits
a right pseudo-inverse, then its image is closed. Indeed, we have, $\im(f)=\Ker(fg-\Id)$. For our application, we need a slightly stronger statement:
\begin{lem}
\label{quasi_pseudo_inverse_then_closed}
Let $A_{0}\subseteq A$
and $B_{0}\subseteq B$ be closed embeddings of topological vector
spaces. Let $f:A\to B$ be a continuous operator such 
$
\overline{f(A)}\cap B_{0}=\overline{f(A_{0})}
$
Suppose
that there is a map $g:B\to A$ satisfying:
\begin{enumerate}
    \item $g(B_{0})\subseteq A_{0}$.
    \item The induced map on the quotients, $\bar{g}:B/B_{0}\to A/A_{0}$ is a right pseudo-inverse to the map $\bar{f}\colon A/A_0\to B/B_0$. 
\end{enumerate} 
If $f(A_{0})$ is closed then $f(A)$ is closed.
\end{lem}

Note that the assumption on $f$ implies that $f(A_0)\subseteq B_0$ so that $\bar{f}$ is well defined.  
\begin{proof}
Let $\alpha\in\overline{f(A)}$. We wish to show that $\alpha \in \im(f)$. 
For this, it would suffice to show that $\alpha-f(g(\alpha))\in f(A_0)$. Since $\alpha\in\overline{f(A)}$ we get 
\[
\alpha-f(g(\alpha)) \in(\Id-fg)\left(\overline{f(A)}\right)\subseteq\overline{\im(f-fgf)}\overset{(!)}{\subseteq}\overline{f(A)\cap B_{0}}\subseteq\overline{f(A)}\cap B_{0}\overset{(!!)}{=}\overline{f(A_{0})}\stackrel{(!!!)}{=}f(A_{0})
\]
where the inclusion $(!)$ follows from the assumption that $\bar{g}$
is a right pseudo-inverse to $\bar{f}$, the equality $(!!)$ is one of our assumptions on $f$ and $(!!!)$ follows from the assumption
that $f(A_{0})$ is closed.
\end{proof}

\subsubsection{Combinatorial Maps} 
Recall that our goal is to analyse the cosheaves of the form $\sF \scten \Sc_{\sE}$ for $\sE$ a Nash vector bundle over a Nash manifold $X$ and $\sF$ a constructible $\Der(\RR)$-valued sheaf on $X$. 
Using \Cref{exs_ps_free_res}, we can represent $\sF$ as (the sheafification of) a finite complex 
\[
\dots \oto{d}C_k \oto{d}C_{k-1} \oto{d} \dots
\]
where $C_k \simeq \bigoplus_{t\in T_k} \RR_{U_t}$ for some finite collection $\{U_t\}_{t\in T_k}$ of open sets in $X$. For every integer $k$, the map $d\colon \bigoplus_{t\in T_k} \RR_{U_{t}}\to \bigoplus_{s\in T_{k-1}} \RR_{U_{t}}$ can be presented by a $T_k \times T_{k-1}$-matrix 
\[
\{M_s^t(x)\}_{t\in T_k, s\in T_{k-1}},
\]
where, for every $t\in T_k$ and $s\in T_{k-1}$, the function $M_s^t(x)$ is a locally-constant real valued function on $U_t$ which vanishes outside of $U_s$. 

Tensoring this presentation of $\sF$ with $\Sc_{\sE}$, and using the fact that the tensor product operation $(-)\scten\Sc_{\sE}$ is colimit-preserving, we get a presentation of $\sF \scten \Sc_{\sE}$ as the realization of an explicit complex of $\Fre$-valued cosheaves on $X$. Namely, let us denote by  $j_t\colon U_t \into X$ the embeddings. 
Then, $\sF \scten \Sc_{\sE}$ is the realization of the complex
\[
\dots \oto{d}\bigoplus_{t\in T_k} (j_t)_!j_t^!\Sc_\sE \oto{d} \bigoplus_{t\in T_{k-1}} (j_t)_!j_t^!\Sc_\sE \oto{d}\dots.  
\]
The differential  
$d(U)\colon \bigoplus_{t\in T_k} \Sc_\sE(U\cap U_t) \to
\bigoplus_{s\in T_{k-1}} \Sc_\sE(U\cap U_t)
$
takes a tuple of Schwartz sections $(\alpha_t)_{t\in T_{k}}$, to the tuple of Schwartz sections $(\beta_s)_{s\in T_{k-1}}$ given by 
\[
\beta_s(x) = \sum_{t\in T_k} M^t_s(x)\alpha_t(x),
\]
(where we set $\alpha_t(x)=0$ for $x\notin U_t$). Next, we observe that, by splitting each $U_t$ into its connected components, it suffices to consider the case where the functions $M^t_s(x)$ as above are all \emph{constant}. 
This motivates the following definition.
\begin{defn} \label{def:comb_map}
 Let $X$ be a Nash manifold. 
 \begin{enumerate}
     \item Let $\{U_t\}_{t\in T}$ be a finite collection of open subsets of $X$, and let $j_t\colon U_t \into X$ denote the embedding. For a Nash vector bundle $\sE$ on $X$, we refer to a cosheaf of the form $\bigoplus_{t\in T} (j_t)_! j_t^! \Sc_\sE$ on $X$ as an ($\sE$-valued) \tdef{combinatorial cosheaf}. 
 \item  For finite collections $\{U_t\}_{t\in T}$ and $\{V_s\}_{s\in S}$ of open subsets of $X$ and a real valued matrix $(M^t_s)_{t\in T,s\in S}$ for which $M_s^t\ne 0$ only when $U_t\subseteq V_s$, we can associate a map 
 \[
    f\colon \bigoplus_{t\in T} (j_t)_! j_t^!\Sc_{\sE} \to \bigoplus_{s\in S} (j_s)_!j_s^!\Sc_\sE, 
 \]
 given on the sections over $U\in \Op(X)$ by the formula
 \[
 f(\alpha)_s = \sum_{t\in T} M_s^t\alpha_t \qin \bigoplus_{s\in S} \Sc_\sE(U\cap V_s)
 \]
 for
 $
 \alpha = (\alpha_t)_{t\in T} \in \bigoplus_{t\in T} \Sc_{\sE}(U_t \cap U)$. 
 We refer to such a map as a \tdef{combinatorial map}.
 \end{enumerate}
 
\end{defn}

The discussion above can now be summarized as follows. 
\begin{prop}
\label{sch_sec_combi}
Let $X$ be a Nash manifold, let $\sE$ be a Nash vector bundle on $X$, and let $\sF\in \Shv_c(X,\Der(\RR))$. Then, the cosheaf $\sF \scten\Sc_\sE$ can be represented by a finite complex of cosheaves of the form $\bigoplus_{t\in T} (j_t)_! \Sc_{\sE}$ with combinatorial differentials. 
\end{prop}
\begin{proof}
Since every constructible sheaf on $X$ admits a pseudo-free resolution, and by the discussion above, it would suffice to verify that for a pseud-free complex $C_\bullet$ we have $[C_\bullet] \scten \Sc_\sE \simeq [C_\bullet \scten \Sc_\sE]$, where on the target we use the fact that $C_i\scten \Sc_\sE$ is a cosheaf constentrated in degree $0$, since $C_i$ is pseudo-free. Let $\sA_0$ be the additive category of combinatorial sheaves on $X$. We can regard both sides of the above identity as exact functors $\Ch_b(\sA_0)\to \CShv(X,\infFre)$. Since they clearly agree on $\sA_0$-itself, the result follows from the universal property of $\Ch_b(\sA_0)$.
\end{proof}

\subsubsection{Proof of the 
Hausdorffness of Schwartz Sections}
Given \Cref{sch_sec_combi}, to show that $(\sF\scten\Sc_\sE)(X)$ has Hausdorff homology spaces, it suffices to show that every combinatorial map induces a map with closed image on the global sections. We now show this by verifying the conditions of
\Cref{quasi_pseudo_inverse_then_closed} for a suitable filtration on a combinatorial map. 
To verify the first condition, we need the following analytic fact.
\begin{prop}[Approximation of Unity]
\label{approx_unity}
Let $X$ be a Nash manifold and let $\sE$ be a Nash vector bundle on $X$. There exists a sequence of Schwartz functions $a_k(x)\in \Sc(X)$ such that, for every Schwartz section $\alpha \in \Sc_\sE(X)$, we have 
\[
\lim_{k\to \infty}(a_k\cdot \alpha) = \alpha\qin \Sc_\sE(X).
\] 
\end{prop}

\begin{proof}
If $\sE$ is trivial and $X$ is an open subset of $\RR^n$, we can choose $a_k(x) = e^{-\frac{|x|^2}{k}}$. A straight forward computation shows that this sequence indeed satisfies the required property. In general, we can cover $X$ by finitle many open sets $U_i$, each of which trivializing $\sE$ and isomorphic to open subsets of $\RR^n$. Choose an appropriate sequence $a_{k,i}(x)$ for every $i$, and a tempered partition of unity $\{\rho_i\}$ for the cover. Then, $a_k(x) = \sum_{i} \rho_i a_{k,i}(x)$ has the desired property.    
\end{proof}

\begin{rmk}
Though we shall not use this fact nor prove it, one can show that the $f_i$-s can be chosen independently of $\sE$.
\end{rmk}
We are now ready to show that a combinatorial map satisfies the first condition of \Cref{quasi_pseudo_inverse_then_closed}.
\begin{prop}
\label{comb_map_good_filt}
Let $X$ be a Nash manifold and let $\sE$ be a Nash vector bundle on $X$. Let 
\[
f\colon \sG_0 = \bigoplus_{t\in T} (j_t)_!j_t^! \Sc_{\sE} \to \bigoplus_{s\in S} (j_s)_!j_s^! \Sc_{\sE} = \sG_1
\]
be a combinatorial map. For every open subset $V\subseteq X$, we have 
\[
\overline{f(\sG_0(X))} \cap \sG_1(V) = \overline{f(\sG_0(V))},
\]
where the closures are taken in $\sG_1(X)\simeq \bigoplus_{s\in S}\Sc_\sE(U_s)$.
\end{prop}

\begin{proof}
First, we note that $f(\sG_0(V))\subseteq \sG_1(V)$, and since $\sG_1(V)$ is closed in $\sG_1(X)$, we conclude that 
\[
\overline{f(\sG_0(V))}\subseteq \sG_1(V)\cap \overline{f(\sG_0(X))}.
\]
We shall now show the opposite inclusion. 
Let $\alpha \in \sG_1(V)\cap \overline{f(\sG_0(X))}$, so that $\alpha = \sum\limits_{s\in S}\alpha_s$ for some $\alpha_s \in \Sc_{\sE}(V\cap U_s)$. Using approximation of unity for $V$ (\Cref{approx_unity}) we can find a sequence $a_k \in \Sc(V)$ such that 
\[
\lim_{k \to \infty} a_k \alpha_s = \alpha_s, \quad \forall s\in S
\]
and hence 
$
\lim\limits_{k\to \infty} a_k \alpha = \alpha.
$

Since $\overline{f(\sG_0(V))}$ is closed in $\sG_1(X)$, it would suffice to show that, for every $k\in \NN$, we have \[
a_k\alpha \in \overline{f(\sG_0(V))}.
\]
To achieve that, we note that since $\alpha \in \overline{f(\sG_0(X))}$ and $a_k \in \Sc(V)$, we have 
\[
a_k\alpha \in a_k\cdot \overline{f(\sG_0(X))} \subseteq \overline{a_k\cdot f(\sG_0(X))} = \overline{f(a_k\cdot \sG_0(X))} \subseteq \overline{f(\sG_0(V))}.   
\]
\end{proof}
To apply \Cref{quasi_pseudo_inverse_then_closed}, we now need to construct a map $g \colon \sG_1(X) \to \sG_0(X)$ which gives a right pseudo-inverse for $f$ after taking the quotient by the subsepaces $\sG_0(V)$ and $\sG_1(V)$. To construct such a map $g$, we need an extra assumption on $V$. 

\begin{defn}
Let $X$ be a Nash manifold and let $\{U_t\}_{t\in T}$ be a collection of open subsets of $X$. We say that $V\in \Op(X)$ is \tdef{in good position} with respect to the $U_t$-s, if for every $t\in T$ either $U_t\subseteq V$ or $U_t\cup V = X$; Equivalently, if for every $t\in T$ the complement, $Z = X\setminus V$, is either contained in $U_t$ or disjoint from $U_t$. 
\end{defn}

\begin{prop}
\label{good_pos_psu_inv}
Let $X$ be a Nash manifold and let $\sE$ be a Nash vector bundle on $X$. Let 
\[
f\colon \sG_0 = \bigoplus_{t\in T} (j_t)_! j_t^!\Sc_{\sE} \to \bigoplus_{s\in S} (j_s)_!j_s^! \Sc_{\sE} = \sG_1
\]
be a combinatorial map. For every open subset $V\subseteq X$ which is in good position with respect to the subsets $\{U_t\}_{t\in T\cup S}$, there is a map 
$g\colon \sG_1 \to \sG_0$ for which the induced map 
\[
\overline{g}\colon \sG_1(X)/\sG_1(V)\to \sG_0(X)/\sG_0(V) 
\]
is a right pseudo-inverse for the map $\overline{f}\colon \sG_0(X)/\sG_0(V)\to \sG_1(X)/\sG_1(V) $.
\end{prop}

\begin{proof}
Let $Z = X\setminus V$, let $T_1 = \{t\in T : Z\subseteq U_t\}$ and let $S_1 = \{s\in S : Z\subseteq U_s\}$.
For every $t\in T\sqcup S$ we have 
\[
\Sc_{\sE}(U_t) / \Sc_{\sE}(U_t \cap V) \simeq 
\begin{cases}
\Sc_\sE(X)/\Sc_{\sE}(V) & \text{ if } Z\subseteq U_t \\ 
0 & \text{ else }
\end{cases}
\]
from which we deduce that 
\[
\sG_0(X)/\sG_0(V)\simeq \Sc_\sE(X)/\Sc_\sE(V)\otimes \RR[T_1]
\]
and 
\[
\sG_1(X)/\sG_1(V)\simeq \Sc_\sE(X)/\Sc_\sE(V)\otimes \RR[S_1].
\]
The combinatorial map, $f$, is associated with a real $T\times S$ matrix $M_s^t$. Let $\overline{M}_s^t$ be the $T_1\times S_1$-submatrix of $M_s^t$, so that $\overline{M}_s^t$ represents a linear map $\overline{M} \colon \RR[T_1] \to \RR[S_1]$. Since $\overline{M}$ is a linear map of finite dimensional real vector spaces, it admits a right pseudo-inverse $\overline{N}$, representable by  an $S_1\times T_1$-matrix $\overline{N}_t^s$.
We can now extend $\overline{N}_t^s$ to a $S\times T$-matrix $N_t^s$ by adding $0$ entries, i.e.,  
\[
N_{t}^s = \begin{cases}
\overline{N}_t^s & \text{ if } s\in S_1 \text{ and } t\in T_1 \\ 
0 & \text{ else }
\end{cases}.
\]
Let $W = \bigcap_{t\in T_1\cup S_1} U_t$ and let $j\colon W\into X$ denote the embedding. 
For every $s\in S$ and $t\in T$, if $N_t^s \ne 0$ then, in particular, $t\in T_1$ and so 
\[
U_t \cap W = W \supseteq U_s \cap W.
\]
Consequently, the matrix $N_t^s$ determines a combinatorial map 
$\widetilde{g}\colon j^!\sG_1 \to j^!\sG_0$. 

Choose a tempered function $\rho \colon X\to \RR$ such that $\rho$ equals $1$ in a neighborhood of $Z$ and is flat outside of $W$ (such exist, e.g., by a tempered partition of unity). 
In this case,  multiplication by $\rho$ gives a map $\sG_1 \to j_!j^! \sG_1$. We now define the morphism $g\colon \sG_1 \to \sG_0$ 
to be the  composition \[
g\colon \sG_1 \oto{\cdot \rho} j_!j^! \sG_1 \oto{j_! \widetilde{g}} j_!j^!\sG_0 \oto{\con} \sG_0.
\] 
Namely, for $\alpha \in \Sc_\sE(U_s)$ we have 
\[
g(\alpha)_t = \sum_s\rho \cdot N_t^s \cdot \alpha_s \in \Sc_\sE(U_t).
\]
 It remains to show that $\bar{g} \colon \sG_1(X)/\sG_1(V) \to \sG_0(X)/\sG_0(V)$ is a right pseudo-inverse of $\bar{f}$. 
 Since $\rho=1$ in a neighborhood of $Z$, the map  
\[
\Sc_\sE(X)/\Sc_\sE(V) \otimes \RR[S_1] \simeq \sG_1(X)/\sG_1(V) \oto{\bar{g}} \sG_0(X)/\sG_0(V)\simeq \Sc_\sE(X)/\Sc_\sE(V)\otimes \RR[T_1]
\]
coincides with the map $1\otimes \overline{N}$. Since $\overline{N}$ is a right pseudo-inverse of $\overline{M}$, we deduce that $\bar{g}$ is a right pseudo-inverse of $\bar{f}$.
\end{proof}

\begin{corl}
\label{ind_step_closed_image}
With the same conditions as in \Cref{good_pos_psu_inv}, if the map $f(V)\colon \sG_0(V)\to \sG_1(V)$ has closed image, then $f(X)\colon \sG_0(X)\to \sG_1(X)$ has closed image. 
\end{corl}

\begin{proof}
Combining \Cref{comb_map_good_filt} and \Cref{good_pos_psu_inv}, we deduce that the map $f(X)$ satisfies the assumptions of \Cref{quasi_pseudo_inverse_then_closed} with respect to the sub-spaces $\sG_0(V)\subseteq \sG_0(X)$ and $\sG_1(V)\subseteq \sG_1(X)$. 
\end{proof}

We can now show that conbinatorial maps has closed image.

\begin{prop}
\label{Comb_closed_image} 
Let $X$ be a Nash manifold and let $\sE$ be a Nash vector bundle on $X$. Let $f\colon   \bigoplus_{t\in T} (j_t)_!j_t^!\Sc_{\sE} \to \bigoplus_{s\in S} (j_s)_!j_s^!\Sc_\sE$ be a combinatorial map. Then, the map 
\[
f(X)\colon \bigoplus_{t\in T} \Sc_{\sE}(U_t) \to \bigoplus_{s\in S} \Sc_{\sE}(U_s)
\]
induced from $f$ on the global sections
has closed image. 
\end{prop}

\begin{proof}
Let $\alpha\colon X\to P$ be the stratification of $X$ associated with the collection of open sets $\{U_t\}_{t\in T\cup S}$ as in \Cref{cov_Strat}.
We prove the result by induction on the size of $P$. 
If $P=\varnothing$ then $X=\varnothing$ and there is nothing to prove. Otherwise, let $Z\subseteq X$ be a closed stratum, and let $V = X\setminus Z$. Then $V$ is in good position with the collection of open sets $\{U_t\}_{t\in T\cup S}$. By \Cref{ind_step_closed_image}, to show that $f(X)$ has closed image, it would suffice to show that the map 
\[
f(V)\colon \bigoplus_{t\in T} \Sc_\sE(U_t \cap V) \to \bigoplus_{s\in S} \Sc_\sE(U_s \cap V)
\]
has closed image. Since the map $f(V)$ is (the global sections of) a combinatorial map over $V$, and since the stratification of $V$ corresponding to the open sets $\{U_t \cap V\}_{t\in T\cup S}$ has strictly fewer strata than $\alpha$, the result now follows from the inductive hypothesis. 
\end{proof}

Finally, we arrive at the goal of this section.
\begin{thm}
\label{scten_const_hauss}
For every
Nash bundle $\sE$ on a Nash manifold $X$ and every constructible
sheaf $\sF\in \Shv_c(X,\Der(\RR))$, the complex $(\sF\scten\Sc_{\sE})(X)$
has Hausdorff homologies. 
\end{thm}

\begin{proof}
Since, by \Cref{sch_sec_combi}, $\sF \scten \Sc_\sE$ is the realization of a complex of cosheaves with combinatorial differentials, the result follows from \Cref{Comb_closed_image}
\end{proof}

 \subsection{Proof of \Cref{intro-rel_der}}
For the reader's convenience, let us indicate how the statement of the main theorem of the paper follows from the results in the body of this text. Recall that, for an infinity category $\cC$, we denote by $h\cC$ the homotopy category of $\cC$, obtained by replacing the morphism spaces in $\cC$ by the sets of their connected components. 

Let $X$ be a Nash manifold. First, since the statement of the theorem involves only constructible sheaves, we can work with the restricted topology on $X$ rather than the classical topology by \Cref{comp_const_st}. 
We have a realization functor of stable $\infty$-categories $\Der(\Shv(X,\RR))\to \Shv(X,\Der(\RR))$. Passing to the homotopy categories and restrict to the constructible objects, we obtain a functor $D_c(X,\RR)\to h\Shv_c(X,\Der(\RR))$ from the constructible derived category of sheaves of real vector spaces on $X$, to the homotopy category of $\Shv_c(X,\Der(\RR))$.
We can now define a functor $D_c(X,\RR)\to h\infFre$ by the composition 
\[
    D_c(X,\RR) \to h\Shv_c(X,\Der(\RR)) \oto{(-)\scten \Sc_\sE} h\CShv(X,\infFre)\oto{\Gamma_!} h\infFre. 
\]
Now, using the combinatorial presentations of the cosheaves of the form $\sF\scten \Sc_\sE$ (\Cref{sch_sec_combi}), we deduce that the image of this composite lands in the full subcategory $h\Ch_b(\Fre^\kappa)$, so we get the type of functor we wanted to construct. Property $(1)$ in the statement of \Cref{intro-rel_der} is now an immediate consequences of \Cref{global_relative_de_rham}. It remains to show that property $(2)$ holds. 
Namely, we need to show the following
\begin{prop} \label{rel_der_haus}
Let $\phi:X\to Y$ be a Nash submersion and let $\sE$ be a Nash
vector bundle on $Y$. Then, the complex $\Sdrt_{\phi,\sE}(X)$ has Hausdorff
homologies. The same holds also for $\Sdr_{\phi,\sE}$.
\end{prop}

\begin{proof}
By \Cref{global_relative_de_rham} we have 
$\Sdrt_{\phi,\sE}\simeq \phi_\sharp \RR_X \scten \Sc_\sE$. By \Cref{cl_struct_shv}, $\phi_\sharp \RR_X$ is constructible. Hence, the result follows from \Cref{scten_const_hauss}. The argument for $\Sdr_{\phi,\sE}$ is similar, using the fact that the dualizing local system of $\phi$ is itself constructible.  
\end{proof}
This concludes the proof of \Cref{intro-rel_der} and our investigation of relative de Rham theory for Nash submersions.
\bibliographystyle{alpha}
\bibliography{ref}
 
\end{document}